\documentclass{amsart}
\usepackage{etex}

\usepackage[left=1.3in, right=1.3in, bottom=1.25in, top=1.25in]{geometry}
\usepackage{chngcntr}
\usepackage{amsfonts}
\usepackage{sseq}
\usepackage{caption}
\usepackage{amssymb}
\usepackage{mathrsfs}
\usepackage{xcolor,graphicx}
\usepackage{float}
\usepackage{amsmath}
\usepackage{url}
\usepackage{verbatim}

\usepackage{tikz}

\newcommand{\circledbullet}{
\tikz{
\pgfsetbaselinepointlater{\pgfpointanchor{X}{base}}
\pgfcircle{\pgfpointorigin}{0.12cm}
\pgfusepath{stroke}
\node (X) {$\bullet$};
}}
\newcommand{\cbull}{ 
\text{\footnotesize \circledbullet}}

\usepackage{xcolor} 
\usepackage{hyperref}
\definecolor{dark-red}{rgb}{0.5,0.15,0.15}
\definecolor{dark-blue}{rgb}{0.15,0.15,0.6}
\definecolor{dark-green}{rgb}{0.15,0.6,0.15}
\hypersetup{
    colorlinks=true, linkcolor=dark-red,
    citecolor=dark-blue, urlcolor=dark-green
}

\usepackage{xy}
\input xy
\xyoption{all}
\newcommand{\spec}{\mathrm{Spec}}
\newcommand{\proj}{\mathrm{Proj}}
\newcommand{\Proj}{\mathrm{Proj}}

\newcommand{\md}{\mathrm{Mod}}
\newcommand{\Z}{\mathbb{Z}}

\newcommand{\loc}{\mathrm{Loc}}

\DeclareMathOperator{\im}{im}
\DeclareMathOperator{\Tot}{Tot}

\newcommand{\pic}{\mathrm{Pic}}

\newcommand{\gl}{\mathfrak{gl}}
\newcommand{\qcoh}{\mathrm{QCoh}}

\renewcommand{\mod}{\mathrm{Mod}}

\newcommand{\pics}{\mathfrak{pic}}

\usepackage{amsthm}
\usepackage{cleveref}
\crefformat{equation}{(#2#1#3)}

\renewcommand{\sp}{\mathrm{Sp}}

\newcommand{\otop}{\mathcal{O}^{{top}}}

\newcommand{\mell}{M_{{ell}}}
\newcommand{\mellb}{\mathfrak{M}_{{ell}}}
\newcommand{\mellc}{\overline{M}_{ell}}
\newcommand{\mellbc}{\overline{\mathfrak{M}}_{ell} }

\renewcommand{\hom}{\mathrm{Hom}}
\newcommand{\map}{\mathrm{Map}}

\numberwithin{equation}{section}
\newtheorem{lm}{lm}[subsection]
\newtheorem{lemma}[lm]{Lemma}
\newtheorem{corollary}[lm]{Corollary}

\newtheorem{theorem}[lm]{Theorem}
\newtheorem{proposition}[lm]{Proposition}
\newtheorem{tool}[lm]{Comparison Tool}

\newtheorem{abc}{abc}
\newtheorem{ABCthm}[abc]{Theorem}

\newtheorem*{quest}{Question}

\renewcommand{\rightrightarrows}{\begin{smallmatrix} \to \\
\to \end{smallmatrix} }
\newcommand{\triplearrows}{\begin{smallmatrix} \to \\ \to \\ 
\to \end{smallmatrix} }

\theoremstyle{definition}
\newtheorem{definition}[lm]{Definition}

\newtheorem{question}[lm]{Question}

\newcommand{\Tmf}{{Tmf}}
\newcommand{\TMF}{{TMF}}
\renewcommand{\ell}{\mathrm{Ell}}

\theoremstyle{definition}

\newtheorem{construction}[lm]{Construction}
\newtheorem{remark}[lm]{Remark}

\newtheorem{example}[lm]{Example}

\newcommand{\noteV}[1]{{\color{blue} \bf #1}} 
\newcommand{\e}[1]{\mathbf{E}_{#1}}

\begin{document}

\title{The Picard group of topological modular forms via descent theory}
\date{\today}
\author{Akhil Mathew}
\address{Department of Mathematics, Harvard University, Cambridge, MA}
\email{amathew@math.harvard.edu} 

\author{Vesna Stojanoska}
\address{Department of Mathematics, University of Illinois at Urbana-Champaign, Urbana, IL}
\email{vesna@illinois.edu}
\thanks{The first author is partially supported by the NSF Graduate Research
Fellowship under grant DGE-110640, and the second author is partially
supported by NSF grant DMS-1307390.}

\begin{abstract}
This paper starts with an exposition of descent-theoretic techniques in the
study of Picard groups of $\e{\infty}$-ring spectra, which naturally lead to
the study of Picard spectra. We then develop tools for the efficient and
explicit determination of differentials in the associated descent spectral
sequences for the Picard spectra thus obtained. As a major application, we
calculate the Picard groups of the periodic spectrum of
topological modular forms $\TMF$ and  the
non-periodic and non-connective $Tmf$. We find that $\pic (\TMF)$ is cyclic of
order 576, generated by the suspension $\Sigma\TMF $ (a result originally due
to Hopkins), while $\pic(\Tmf) = \Z\oplus \Z/24$. In particular, we show that  there exists an invertible $\Tmf$-module which is not equivalent to a suspension of $\Tmf$.
\end{abstract}

\maketitle
\tableofcontents
\listoffigures


\section{Introduction}

Elliptic curves and modular forms occupy a central role in modern stable
homotopy theory in the guise of the variants of \emph{topological modular
forms}: the connective $tmf$, the periodic $\TMF$, and $Tmf$ which interpolates
between them. These are structured ring spectra which have demonstrated
surprising connections between the arithmetic of elliptic curves and
$v_2$-periodicity in stable homotopy. For example, $tmf$ detects a number of
$2$-torsion and $3$-torsion classes in the stable homotopy groups of spheres
through the Hurewicz image. Even more interestingly, the more geometric-natured
$\TMF$ can be used to detect and describe, using congruences between modular
forms, the $2$-line of the Adams-Novikov spectral sequence at primes $p\geq 5$, according to \cite{behrensCongruences}.

From a different perspective, the structure of topological modular forms as $\e{\infty}$-ring spectra leads to symmetric monoidal $\infty$-categories of modules which give rise to well-behaved invariants of algebraic or algebro-geometric type. For instance, \cite{meier} has studied $TMF$-modules which become free when certain level structures are introduced; these can be thought of as locally free sheaves with respect to a pre-determined cover.

Our goal in this paper is to understand another such invariant, the Picard
group. Any symmetric monoidal category has an associated group of
isomorphism classes of objects
invertible under the tensor product, which is called the \emph{Picard group}.
The classical examples are the Picard group $\pic(R)$ of a ring $R$, i.e., of the
category $\md(R)$ of $R$-modules, or the Picard group of a scheme $X$, i.e., of
the category $\md({\mathcal{O}_X})$ of quasi-coherent modules over its structure sheaf. In
homotopy theory, the interest in Picard groups arose when Mike Hopkins made the
observation that the homotopy categories of $E_n$-local and $K(n)$-local
spectra have interesting Picard groups, particularly when the prime
at hand is small in comparison with $n$. Here, $E_n$ is the Lubin-Tate spectrum
and $K(n)$ is the Morava $K$-theory spectrum at height $n$. In the few existing
computations of such groups, notably \cite{HMS, HoveySadofsky,
KamiyaShimomura, GHMR, Heard}, one often uses that an invertible $E_n$-module must be a suspension of $E_n$ itself. 

The $K(2)$-localization of any of the three versions of topological modular
forms gives a spectrum closely related to the Lubin-Tate spectrum $E_2$;
namely, this localization is a finite product of homotopy fixed point spectra of  finite
group actions on $E_2$ (or slight variants of $E_2$ with larger residue fields). More generally, each $E_n$ is an $\e{\infty}$-ring
spectrum with an action, through $\e{\infty}$-ring maps, by a profinite group
$\mathbb{G}_n$ called the Morava stabilizer group (see \cite{rezknotes} for
the $\e{1}$-ring case). The $K(n)$-local sphere is obtained then as
the Devinatz-Hopkins homotopy fixed points. However, $\mathbb{G}_n$ also has interesting finite
subgroups when the prime is relatively small with respect to $n$. If $G $ is
such a subgroup, the homotopy fixed points $E_n^{h G}$ are an $\e{\infty}$-ring spectrum, which is in theory easier to study than the $K(n)$-local sphere, but hopefully contains a lot of information about the $K(n)$-local sphere. 
For instance, Hopkins has observed that in all known examples, the Picard group of $E_n^{hG}$ (unlike that of the
$K(n)$-local category) is very simple as it only contains suspensions of $E_n^{hG}$, and raised the following natural question.

\begin{quest}[Hopkins]
Let $G$ be a finite subgroup of the Morava stabilizer group $\mathbb{G}_n$ at height $n$. 
Is it true that any invertible $K(n)$-local module over $E_n^{hG}$ is a suspension of $E_n^{hG}$? 
\end{quest}

The periodic $\TMF$ is closer to its $K(2)$-localization than $Tmf$, and this is demonstrated by the following result, originally due to Hopkins but unpublished.
\begin{ABCthm}[Hopkins]
\label{hopkins1}
The Picard 
group of $\TMF$ is isomorphic to $\mathbb{Z}/576$, generated by the
suspension $\Sigma \TMF$. 
\end{ABCthm} 
\newtheorem*{introtheoremA}{Theorem~\ref{hopkins1}}
 In the paper at hand, we prove Theorem~\ref{hopkins1} using a
 descent-theoretic approach. In particular, our method is different
 from Hopkins's.
The descent-theoretic approach also enables us to prove that, nonetheless, the non-connective, non-periodic flavor of topological modular forms  
$\Tmf$ behaves differently and has a more interesting Picard group. 

\begin{ABCthm}
\label{thm:PicTmf}
The Picard 
group of $\Tmf$ is isomorphic to $\mathbb{Z} \oplus \mathbb{Z}/24$, generated by the
suspension $\Sigma \Tmf$ and a certain 24-torsion invertible object. 
\end{ABCthm} 
\newtheorem*{introtheoremB}{Theorem~\ref{thm:PicTmf}}

In addition, we explicitly construct the $24$-torsion module in Construction~\ref{construct}.
We note that, after the initial submission of this paper, the preprint of Hill-Meier \cite{HillMeier} appeared, in which the authors use techniques from
$C_2$-equivariant stable homotopy to construct exotic torsion elements in the Picard
group of $Tmf_1(3)$. In contrast, our construction is given by an unusual gluing of locally trivial modules.

We hope that our method of proof of 
Theorem~\ref{hopkins1} and Theorem~\ref{thm:PicTmf}, which is very general, 
will also be of interest to those not directly concerned with $\TMF$. Our method is
inspired by and analogous to the forthcoming work of Gepner-Lawson \cite{GL} on
Galois descent of
Brauer as well as Picard groups, though the key ideas are classical.

Take, for example, the periodic variant $\TMF$. Its essential property is that it arises as the global sections of the structure sheaf $\otop$ of a
regular ``derived stack''  $(\mellb, \otop)$ refining the moduli stack of elliptic
curves $\mell$. Thus
\[ \TMF = \Gamma( \mellb, \otop)  = \varprojlim_{\spec R \to \mell} \Gamma(
\spec R, \otop), \]
where the maps $\spec R \to \mell$ range over all \'etale morphisms from affine
schemes to $\mell$. Moreover, the $\e{\infty}$-ring spectra $\Gamma(\spec R,
\otop)$ are weakly even periodic; thus we have $\TMF$ as the homotopy limit of
a diagram of weakly even periodic $\e{\infty}$-rings.
It follows by the main result of \cite{affine} that the \emph{module category}
of $\TMF$ can also be represented as the inverse limit of the module categories
$\md( \otop( \spec R))$, that is, as quasi-coherent sheaves on the derived
stack. In any analogous situation, our descent techniques for calculating Picard groups apply.

Over an affine chart $\spec R \to \mell$, 
the Picard group of $\Gamma( \spec R, \otop)$ (i.e. that of an \emph{elliptic
spectrum}) is purely algebraic, by a classical argument of
\cite{HMS, BakerRichter} with ``residue fields.'' 
This results from the fact that the ring $\pi_* \Gamma( \spec R, \otop)$ is
\emph{homologically} simple: in particular, it has finite global dimension,
which makes the study of $\Gamma( \spec R, \otop)$-modules much easier. 
One attempts to use this together with descent theory to compute the Picard group of
$\TMF$ itself; however, doing so necessitates the consideration of higher homotopy
coherences. For this, it is important to work with Picard \emph{spectra}
rather than Picard groups, as they have a better formal theory of descent. 

The Picard spectrum $\pics(A)$ of an $\e{\infty}$-ring $A$ is  an important
spectrum associated to $A$ that deloops the space of units $GL_1(A)$ of \cite{MQRT}\footnote{See \cite{5author} for a very important application.}: it is connective, its $\pi_0$ is the Picard
group of $A$, and its $1$-connective cover $\tau_{\geq 1} \pics(A) $ is
equivalent to $\Sigma \mathfrak{gl}_1(A)$ for $\mathfrak{gl}_1(A)$ the
\emph{spectrum of units} of \cite{MQRT}. We find that
the \emph{Picard spectrum} of $\TMF$ is the connective cover of the homotopy limit 
of $\pics( \otop( \spec R))$, taken over \'etale maps $\spec R \to \mell$. 
This statement is a homotopy-theoretic expression of the descent theory that
we need. 
Thus, we get a \emph{descent spectral sequence} for the homotopy groups of
$\pics(\TMF)$, which is a computational tool for understanding the
aforementioned homotopy coherences concretely. 
We use this technique to compute $\pi_0 ( \pics(\TMF))$, the group
we are after. 

The descent spectral sequence has many consequences in cases where
it degenerates simply for dimensional reasons, or in cases 
where the information sought is coarse. For instance, 
in a specific example (\Cref{picrational}),
we show that the Picard group of the $\e{\infty}$-ring $C^*(S^1;
\mathbb{Q}[\epsilon]/\epsilon^2)$ is given by $\mathbb{Z} \times \mathbb{Q}$,
which yields  a counterexample to a general conjecture of Balmer \cite[Conjecture
74]{balmerICM} on the
Picard groups of certain tensor-triangulated categories. 
We also prove the following general results in \Cref{sec:firstexamples,sec:SSLemma}.
\begin{ABCthm} 
\label{lnlexact}
Let $A$ be a weakly even periodic  Landweber exact $\e{\infty}$-ring with $\pi_0A$
regular noetherian. 
Let $n \geq 1$ be an integer, and let $L_n$ denote localization with respect to the Lubin-Tate spectrum $E_n$. 
The Picard group of $L_n A$ is 
$$
\mathrm{Pic}(L_n A) = \mathrm{Pic}(\pi_*A) \times 
\pi_{-1}(L_n A) ,
 $$
 where $\pic(\pi_*A)$ refers to the (algebraic) Picard group of the
 graded commutative ring $\pi_*A$. 
\end{ABCthm} 
\newtheorem*{introtheoremC}{Theorem~\ref{lnlexact}}
Note that $\pic(\pi_*A)$ sits in an extension
 \[ 0 \to \pic(\pi_0A) \to \pic(\pi_*A) \to \Z/2 \to 0, \]
which is split if $A$ is \emph{strongly} even periodic.

\begin{ABCthm} \label{thm:charzero}
Let $A$ be an $\e{\infty}$-ring such that $\pi_0 A$ is a field of
characteristic zero and
such that $\pi_i A = 0$ for $i  > 0$. Then $\pic(A)$ is infinite cyclic, generated by
$\Sigma A$. 
\end{ABCthm} 
\newtheorem*{introtheoremD}{Theorem~\ref{thm:charzero}}

\begin{ABCthm} \label{thm:GaloisOrder}
Let $G$ be a finite group, and let $A \to B$ be a faithful $G$-Galois extension of $\e{\infty}$-rings in the
sense of Rognes \cite{rognes}. Then the relative Picard
group of $B/A$ (i.e., the kernel of $\pic(A) \to \pic(B)$) is $|G|$-power torsion of finite exponent. 
\end{ABCthm} 
\newtheorem*{introtheoremE}{Theorem~\ref{thm:GaloisOrder}}

For $\TMF$, the descent spectral sequence does not degenerate so nicely, and
we need to work further to obtain our main results. 
The homotopy groups of the Picard spectrum of an $\e{\infty}$-ring $A$, starting with $\pi_2$, 
are simply those of $A$: in fact, we have a natural equivalence of \emph{spaces}
\[ \Omega^{\infty+2} \pics(A) \simeq \Omega^{\infty + 1} A.  \]
This determines the $E_2$-page and many of the differentials in the descent spectral sequence for $
\pic(\TMF)$, but not all the ones that affect $\pi_0$. 
A key step in our argument is the identification of the differentials of the
descent spectral sequence for the Picard spectra, in a certain \emph{range of
dimensions}, with that of the (known) descent spectral sequence for
$\pi_*(\TMF)$. We prove this in a general setting in \Cref{sec:SSLemma} below. 

At the prime $2$, this technique is not sufficient to determine all the
differentials in the descent spectral sequence, and we need to determine in
addition the first ``unstable''  differential in the Picard spectral sequence
(in comparison to the usual descent spectral sequence). We give a ``universal''
formula for this first differential in \Cref{difftheorem}, which we hope
will have further applications.

\subsection*{Conventions} 
Throughout, we will write 
$\mathcal{S}$ for the $\infty$-category of spaces, $\mathcal{S}_*$ for the
$\infty$-category of pointed spaces, and $\sp$ for the $\infty$-category of
spectra. 
We will frequently identify abelian groups $A$ with their associated
Eilenberg-MacLane spectra $HA$.
Finally, all spectral sequences are displayed with the Adams indexing
convention, i.e. the vertical axis represents the cohomological degree, and the
horizontal axis represents the total topological degree.


\subsection*{Acknowledgments} We would like to thank Mike Hopkins for
suggesting this project, and David Gepner, Hans-Werner Henn, and Tyler Lawson  for helpful
discussions. We are grateful to Justin Noel for catching an error in some of
the computations and to
Paul Goerss for pointing us to the work of Priddy \cite{Priddy73}. We would also like to thank the referee for numerous helpful
comments. 
We created the spectral sequence figures using Tilman Bauer's sseq
\LaTeX\ package.

\part{Generalities}

\section{Picard groups}

We begin by giving an introduction  to Picard groups in stable homotopy
theory. General references here include \cite{HMS, May}. 

\subsection{Generalities}
Let $(\mathcal{C}, \otimes, \mathbf{1})$ be a symmetric monoidal category. 

\begin{definition} The \emph{Picard
group}  of $\mathcal{C}$ is the group of isomorphism classes of objects $x \in
\mathcal{C}$ which are \emph{invertible}, i.e. such that there exists an
object $y \in \mathcal{C}$ such that $x \otimes y \simeq  \mathbf{1}$. 
We will denote this group by $\pic(\mathcal{C})$. 
\end{definition}

\begin{remark} 
If $\mathcal{C}$ is a large category, then it is not necessarily clear that the
Picard group is a set. However, in all cases of interest, $\mathcal{C}$ will be
\emph{presentable} so that this will be automatic (see \Cref{useofprcat} below).
\end{remark}

When $\mathcal{C}$ is the category of quasi-coherent sheaves on a scheme (or
stack) $X$, then 
this recovers the usual Picard group of $X$: line bundles are precisely the
invertible objects. The principal goal of this paper is to compute a Picard group in a
homotopy-theoretic setting.

We will repeatedly use the following simple principle, which follows from the
observation that tensoring with an invertible object induces an
autoequivalence of categories. 
\begin{proposition} \label{principle}
Let $\mathcal{C}_0 \subset \mathcal{C}$ be a full subcategory that is preserved
under any autoequivalence of $\mathcal{C}$. 
Suppose the unit object $\mathbf{1} \in \mathcal{C}$ belongs to $\mathcal{C}_0$. Then any $x \in
\pic(\mathcal{C})$ belongs to $\mathcal{C}_0$ as well. 
\end{proposition} 

For example, if $\mathbf{1}$ is a compact object (that is, if
$\hom_{\mathcal{C}}(\mathbf{1}, \cdot)$ commutes with filtered colimits), then
so is $x$.

Suppose now that, more generally, $\mathcal{C}$ is a symmetric monoidal
$\infty$-category in the sense of \cite{higheralg}, which is the setting that
we will be most interested in. Then we can still define
the Picard group $\pic(\mathcal{C})$ of $\mathcal{C}$, which is the same as $\pic( \mathrm{Ho}(
\mathcal{C}))$. 
Moreover, \Cref{principle} is valid, but where one is allowed
to (and often should) use $\infty$-categorical properties. 

\begin{remark}\label{useofprcat} The theory of \emph{presentable} $\infty$-categories
\cite[\S 5.5]{highertopos} enables one to address set-theoretic concerns.
If $\mathcal{C} $ is a presentable symmetric monoidal $\infty$-category, then
the unit of $\mathcal{C}$ is $\kappa$-compact for some regular cardinal
$\kappa$. Therefore, by \Cref{principle} (strictly speaking, its
$\infty$-categorical analog), every invertible object of
$\mathcal{C}$ is $\kappa$-compact, and the collection of $\kappa$-compact
objects of $\mathcal{C}$ is essentially small. 
In particular, the collection of isomorphism classes forms a set and the Picard
group is well-defined.
\end{remark}

\begin{example} 
Suppose that $\mathcal{C}$ is a symmetric monoidal stable $\infty$-category
such that the tensor product commutes with finite colimits in each variable. Then one has a
natural homomorphism 
\[ \mathbb{Z} \to \pic(\mathcal{C}),  \]
sending $n \mapsto \Sigma^n \mathbf{1}$. 
\end{example} 

\begin{example} \label{pics0}
Let $\sp$ be the $\infty$-category of spectra with the smash product. Then
it is a classical result \cite[p. 90]{HMS} that
$\pic( \mathcal{C}) \simeq \mathbb{Z}$, generated by the sphere $S^1$. A quick
proof based on the above principle (which simplifies the argument in
\cite{HMS} slightly) is as follows. If $T \in \sp$ is invertible, so that there exists a
spectrum $T'$ such that $T \wedge T' \simeq S^0$, then 
we need to show that $T$ is a suspension of $S^0$. 

Since the unit object $S^0 \in \sp$ is compact, it follows that 
$T$ is compact: that is, it is a finite spectrum. 
By suspending or desuspending, we may assume that $T$ is
connective,\footnote{We always use ``connective'' to mean
``$(-1)$-connected.''} and that
$\pi_0 T \neq 0$. By the K\"unneth formula, it follows easily that $H_*(T;
F)$ is concentrated in one dimension for each field $F$. Since $H_*(T;
\mathbb{Z})$ is finitely generated, an argument with the universal coefficient
theorem implies that $H_*(T; \mathbb{Z})$  is torsion-free of rank one and is
concentrated in dimension zero: i.e. $H_0(T; \mathbb{Z}) \simeq \mathbb{Z}$. By the Hurewicz theorem, 
$T \simeq S^0$. 
\end{example} 

\begin{example}
Other variants of the stable homotopy category can have more complicated Picard
groups. For instance, if $E \in \sp$, one can consider the $\infty$-category
$L_E \sp$ of \emph{$E$-local spectra,} with the symmetric monoidal structure
given by the $E$-localized smash product $(X, Y) \mapsto L_E(X \wedge Y)$. The
Picard group of $L_E \sp$ is generally much more complicated than $\mathbb{Z}$. 
When $E$ is given by the Morava $E$-theories $E_n$ or the Morava $K$-theories
$K(n)$, the resulting Picard groups have been studied in \cite{HMS} and
\cite{HoveySadofsky}, among other references. 
\end{example}

Another important example of this construction arises for $R$ an $\e{\infty}$-ring, 
when we can consider the symmetric monoidal $\infty$-category 
$\mod(R)$ of $R$-modules.

\begin{definition} Given an $\e{\infty}$-ring $R$, we write $\pic(R)$ for the Picard group
$\pic(\mod(R))$. 
\end{definition}

Using the same argument as in \Cref{pics0}, it follows that any invertible
$R$-module is necessarily compact (i.e. perfect): in particular, the
invertible modules actually form a set rather than a proper class. 
Note that if $R$ is simply an $\e{2}$-ring spectrum, then 
$\md(R)$ is a monoidal $\infty$-category, so one can still define a  Picard group. 
This raises the following natural question.
\begin{question} 
Is there an example of an $\e{2}$-ring whose Picard group is nonabelian? 
\end{question} 

We will only work with $\e{\infty}$-rings in the
future, as it is for these highly commutative multiplications that we will be
able to obtain good (from the point of view of descent theory) infinite loop spaces that realize $\pic(R)$ on $\pi_0$. 
\newcommand{\Pic}{\mathcal{P}\mathrm{ic}}

\subsection{Picard $\infty$-groupoids}

If $(\mathcal{C}, \otimes, \mathbf{1})$ is a symmetric monoidal $\infty$-category, we 
reviewed in the previous section the \emph{Picard group} of $\mathcal{C}$. 
There is, however, a more fundamental invariant of $\mathcal{C}$, where we
remember all isomorphisms (and higher isomorphisms), and which behaves better
with respect to descent processes. 

\begin{definition} 
Let
$\Pic(\mathcal{C})$ denote the $\infty$-groupoid (i.e. space) of
\emph{invertible objects} in $\mathcal{C}$ and equivalences between them.  We will refer to this
as the \emph{Picard $\infty$-groupoid} of $\mathcal{C}$; it is a group-like
$\e{\infty}$-space, and thus \cite{mayiterated, segal} the delooping of a connective \emph{Picard
spectrum} $\pics(\mathcal{C})$. 
\end{definition}

We have in particular
\[ \pi_0 \Pic(\mathcal{C}) \simeq \pic(\mathcal{C}).\]
However, we can also describe the higher homotopy groups of $\Pic(\mathcal{C})$. 
Recall that since $\mathcal{C}$ is symmetric monoidal,
$\mathrm{End}(\mathbf{1})$ is canonically an $\e{\infty}$-space 
and $\mathrm{Aut}( \mathbf{1})$ consists of the grouplike components. 
Since
\[ \Omega \Pic( \mathcal{C}) \simeq \mathrm{Aut}( \mathbf{1}),  \]
we get the relations
\[ 
\quad \pi_1
\Pic(\mathcal{C}) = (\pi_0 \mathrm{End}( \mathbf{1}))^{\times} , \quad \pi_i \Pic( \mathcal{C}) = \pi_{i-1}
\mathrm{End}(\mathbf{1}) \ \text{for } i \geq 2. \]

\begin{example}
Let $R$ be an $\e{\infty}$-ring. 
We will write
\[ \Pic(R) \stackrel{\mathrm{def}}{=} \Pic( \mathrm{Mod}(R)), \quad 
\pics(R) \stackrel{\mathrm{def}}{=} \pics( \md(R))
.  \]
Then $\Pic(R)$ is a delooping of the \emph{space of units} $GL_1(R)$ studied in 
\cite{MQRT} and more recently using $\infty$-categorical techniques in \cite{5author}. 
In particular, the homotopy groups of $\Pic(R)$ look very much like those of
$R$ (with a shift), starting at $\pi_2$. 
In fact, if we take the connected components at the basepoint, we have a
natural equivalence
of spaces
\[ \tau_{\geq 1} (GL_1 R) \simeq \tau_{\geq 1}(\Omega \Pic(R)) \simeq
\tau_{\geq 1} ( \Omega^\infty R),  \]
given by subtracting $1$ with respect to the group structure on
the infinite loop space
$\Omega^\infty R$. 
Nonetheless, the \emph{spectra} $\pics(R)$ and $R$ are generally very
different: that is, the infinite loop structure on $\Pic(R)$ behaves very
differently from that of $\Omega^\infty R$. 
\end{example}

Unlike the group-valued functor $\pic$, $\Pic$ as well as $\pics$ have the
fundamental property, upon which the calculations in this paper are based, that they commute with homotopy limits. 

\begin{proposition}
\label{piclimits}
The functor
\[ \pics\colon \mathrm{Cat}^{\otimes} \to  \sp_{\geq 0},  \]
from the $\infty$-category $\mathrm{Cat}^{\otimes}$ of symmetric monoidal
$\infty$-categories to the $\infty$-category $\sp_{\geq 0}$ of connective spectra, commutes
with limits and filtered colimits, and $\Pic = \Omega^\infty \circ \pics\colon \mathrm{Cat}^{\otimes}
\to \mathcal{S}_*$ does as well. 
\end{proposition}
\begin{proof} 
\newcommand{\CAlg}{\mathrm{CAlg}}
We will treat the case of limits; the case of filtered colimits is similar and
easier.
It suffices to show that $\Pic$ commutes with homotopy limits, since
$\Omega^\infty\colon \sp_{\geq 0} \to \mathcal{S}_*$ creates limits. 
Let $\CAlg (\mathcal{S})$ be the $\infty$-category of $\e{\infty}$-spaces. 
Now, $\Pic$ is
the composite $\mathrm{inv} \circ \overline{\iota} $ where:
\begin{enumerate}
\item  
$\overline{\iota}\colon \mathrm{Cat}^{\otimes} \to \CAlg(\mathcal{S})$
sends a symmetric monoidal $\infty$-category to the symmetric monoidal
$\infty$-groupoid (i.e. $\e{\infty}$-space) obtained by excluding all
non-invertible morphisms. 
\item
$\mathrm{inv}\colon \CAlg(\mathcal{S}) \to \mathcal{S}_*$ sends an
$\e{\infty}$-space $X$ to the union of those connected components 
which are invertible in the commutative monoid $\pi_0 X$, with basepoint given
by the identity.
\end{enumerate}
It thus suffices to show that $\overline{\iota}$ and $\mathrm{inv}$ both commute with
limits. 
\begin{enumerate}
\item The functor $\iota\colon \mathrm{Cat} \to \mathcal{S}$ that sends an
$\infty$-category $\mathcal{C}$ to its \emph{core} $\iota \mathcal{C}$ commutes
with limits: in fact, it is right adjoint to the inclusion $\mathcal{S} \to
\mathrm{Cat}$ that regards a space as an $\infty$-groupoid.  See for instance
\cite[\S 17.2]{riehl}.
Now, to see that $\overline{\iota}$ commutes with limits, we observe that
limits either in $\mathrm{Cat}^{\otimes}$ or in $\CAlg(\mathcal{S})$
are calculated at the level of the underlying spaces (resp.
$\infty$-categories), so the fact that $\iota$ commutes with limits implies
that $\overline{\iota}$ does too. 
\item It is easy to see that $\mathrm{inv}$ commutes with arbitrary products.
Therefore, we need to show that $\mathrm{inv}$ turns pullbacks in
$\CAlg(\mathcal{S})$ into pullbacks in $\mathcal{S}_*$. 
We recall that if $\mathcal{A}, \mathcal{B}$ are complete $\infty$-categories,
then a functor $F \colon  \mathcal{C} \to \mathcal{D}$ preserves limits if and
only if it preserves pullbacks and products \cite[Prop. 4.4.2.7]{highertopos}.
Suppose given a
homotopy pullback
\begin{equation} \label{hpullback1} \xymatrix{
A \ar[d] \ar[r] & B \ar[d] \\
C \ar[r] &  D
}\end{equation}
in $\CAlg(\mathcal{S})$; we need to show that  
\[ \xymatrix{
\mathrm{inv}(A) \ar[d] \ar[r] & \mathrm{inv}(B) \ar[d] \\
\mathrm{inv}(C) \ar[r] &  \mathrm{inv}(D)
}\]
is one too, in $\mathcal{S}_*$. Given the construction of $\mathrm{inv} $ as a
union of connected components, it suffices to show that if $x \in \pi_0 A$ has
the property that $x$ maps to invertible elements in the monoids $\pi_0 B,
\pi_0 C$, then $x$ itself is invertible. 

To see this, consider the homotopy pullback square
\eqref{hpullback1}. Addition of $x$ induces an endomorphism of the square. Since
it acts via homotopy equivalences on $B, C, D$, it follows formally that it
must act invertibly on $A$, i.e., that $x \in \pi_0 A $ has an inverse.
\end{enumerate}
\end{proof} 
 
\subsection{Descent}
Let $R \to R'$ be a morphism of $\e{\infty}$-rings.
Recall the \emph{cobar construction}, a cosimplicial $\e{\infty}$-$R$-algebra
\[ R' \rightrightarrows R' \otimes_R R' \triplearrows \dots,\]
important in descent procedures, which receives an augmentation from $R$. 
The cobar construction is the \emph{\v Cech nerve} (see \cite[6.1.2]{highertopos})
of $R \to R'$, in the opposite $\infty$-category.

\begin{definition}[{\cite[Def. 5.2]{DAGVII}}] 
We say that $R \to R'$ is 
\emph{faithfully flat} if 
$\pi_0 R \to \pi_0 R'$ is faithfully flat and the natural
map $\pi_* R  \otimes_{\pi_0 R} \pi_0 R' \to \pi_* R'$ is an isomorphism. 
\end{definition} 
In this case, the theory of faithfully flat descent goes into
effect. 
We have:

\begin{theorem}[{\cite[Theorem 6.1]{DAGVII}}] \label{thm:ffdescent}
Suppose $R\to R'$ is a faithfully flat morphism of $\e{\infty}$-rings. Then the symmetric monoidal $\infty$-category $\mod(R)$ can be recovered as the
limit of the cosimplicial diagram
of symmetric monoidal $\infty$-categories
\[ \mod(R') \rightrightarrows \mod(R' \otimes_R R') \triplearrows \dots.  \]
\end{theorem}
As a result, by \Cref{piclimits}, $\Pic(R)$ can be recovered as a
totalization of spaces, 
\begin{equation} \label{picdesc} \Pic(R) \simeq \mathrm{Tot}( \Pic( R'^{\otimes
(\bullet + 1)}))  .
\end{equation} 
Equivalently, one has an equivalence of connective spectra
\begin{equation} \label{picdesc2}
\pics(R) \simeq \tau_{\geq 0} \mathrm{Tot}( \pics( R'^{\otimes
(\bullet + 1)})) .
\end{equation} 
In this paper, we will
apply a version of this, except that we will work with morphisms of ring
spectra that are not faithfully flat on the level of homotopy groups. 
As we will see, the descent spectral sequences given by \eqref{picdesc}
and \eqref{picdesc2} are not very useful in the faithfully flat case for our
purposes. 

\begin{example} 
A more classical example of this technique (e.g. \cite[Exercise 6.9]{hartshorne}) is as follows. Let $X$ be a nodal cubic curve over the
complex numbers $\mathbb{C}$. Then $X$ can be obtained from its normalization
$\mathbb{P}^1$ by gluing together $0$ and $\infty$. 
There is a pushout diagram of schemes
\[ \xymatrix{
\left\{0, \infty\right\} \ar[d]  \ar[r] & \ast \ar[d]  \\
 \mathbb{P}^1 \ar[r] &  X.
}\]
Therefore, one would \emph{like} to say that the $\qcoh(X)$ of on $X$ fits into
a homotopy pullback
square
\begin{equation} \label{hpullback}
\xymatrix{
\qcoh(X) \ar[d] \ar[r] & \qcoh( \ast) \ar[d]  \\
\qcoh(\mathbb{P}^1) \ar[r] &  \qcoh( \ast \sqcup \ast),
}
\end{equation}
and that therefore,
therefore, the Picard \emph{groupoid} of $X$ fits into the homotopy cartesian square
\begin{equation} \label{picnodal} \xymatrix{
\Pic(X) \ar[d] \ar[r] &  \Pic(\ast) \ar[d]  \\
\Pic(\mathbb{P}^1) \ar[r] &  \Pic( \ast ) \times \Pic(\ast).
}\end{equation}
Unfortunately, \eqref{hpullback} is not a pullback square of categories, because 
restricting to a closed subscheme is not an exact functor. It is possible to
remedy this (up to connectivity issues) by working with derived $\infty$-categories \cite[Theorem
7.1]{DAGIX}, or by noting that we are working with locally free sheaves and
applying a version of \cite[Theorems 2.1--2.3]{MilnorK}. In any
event, one can argue that \eqref{picnodal} is homotopy cartesian. 

Alternatively, we obtain a homotopy pullback diagram of \emph{connective}
spectra. 
Using the long exact sequence on $\pi_*$, it follows that we have a 
short exact sequence
\[ 0 \to \mathbb{C}^{\times} \to \pic(X) \to \pic(\mathbb{P}^1) \simeq \mathbb{Z} \to 0.  \]
The approach of this paper is essentially an elaboration of this example. 
\end{example} 

\subsection{Picard groups of $\e{\infty}$-rings}

We now specialize to the case of interest to us in this paper. Let $R$ be an
$\e{\infty}$-ring, and consider the Picard group $\pic(R)$, and better yet, the
Picard $\infty$-groupoid $\Pic(R)$ and the Picard spectrum $\pics(R)$. 
The first of these has been studied by Baker-Richter in the paper \cite{BakerRichter}, and we start by
recalling some of their results. 

We start with the following useful property.

\begin{proposition} 
The functor $R \mapsto \pic(R)$ commutes with filtered colimits in $R$. 
\end{proposition} 
\begin{proof} 
This is a consequence of a form of ``noetherian descent'' \cite[\S 8]{EGAIV}. Given an $\e{\infty}$-ring
$T$, let $\md^\omega(T)$ denote the $\infty$-category of perfect $T$-modules. If $I$ is a filtered
$\infty$-category and $\{R_i\}_{i \in I}$ is a filtered system of
$\e{\infty}$-rings indexed by $I$, then the functor of symmetric monoidal $\infty$-categories
\begin{equation} \label{filteredcolim} \varinjlim_{i \in I} \md^\omega(R_i)
\to \md^\omega( \varinjlim_I R_i)  \end{equation}
is an equivalence. We outline the proof of this below.  

Assume without loss of generality that $I$ is a filtered partially ordered
set and write $R = \varinjlim_I R_i$. 
To see that \eqref{filteredcolim} is an equivalence, observe that the $\infty$-category
$\varinjlim_{i \in I} \md^\omega(R_i)$ has objects given by pairs $(M, i)$ where $i
\in I$ and $M \in \md^\omega(R_i)$. The space of maps between $(M, i)$ and $(N,
j)$ is given by $\varinjlim_{k \geq i, j} \hom_{\md(R_k)}( R_k \otimes_{R_i}
M, R_k \otimes_{R_j} N)$. 
For instance, this implies that if $i' \geq i$, the pair $(M, i)$ is
(canonically) equivalent to the pair $(R_{i'} \otimes_{R_i} M, i')$.
Thus, the assertion that \eqref{filteredcolim} is fully faithful
is equivalent to the assertion that if $M, N \in \md^\omega(R_i)$ for some $i$,
then 
the natural map
\begin{equation} \label{filteredcolim:spaces}  \varinjlim_{j \geq i}\hom_{\md^\omega(R_j)}( R_j \otimes_{R_i} M , R_j
\otimes_{R_i} N) 
\to \hom_{\md^\omega(R)}(R \otimes_{R_i} M, R \otimes_{R_i} N)
\end{equation}
is an equivalence. But \eqref{filteredcolim:spaces} is clearly an equivalence
if $M = R_i$ for \emph{any} $N$. The collection of $M \in \md^\omega(R_i)$ such that
\eqref{filteredcolim:spaces} is an equivalence is closed under finite colimits,
desuspensions, and retracts, and therefore it is all of $\md^\omega(R_i)$. 
It therefore follows that \eqref{filteredcolim} is fully faithful. 

Moreover, the image of \eqref{filteredcolim} contains $R \in \md^\omega(R)$ and
is closed under desuspensions and cofibers (thus finite colimits). 
Let $\mathcal{C} \subset \md^\omega(R)$ be the subcategory generated by $R$
under finite colimits and desuspensions. We have shown the image of
the fully faithful functor \eqref{filteredcolim} contains $\mathcal{C}$.
Any object $M \in
\md^\omega(R)$ is a retract of an object $X \in \mathcal{C}$, associated to
an idempotent map $e\colon X \to X$. 
We can ``descend'' $X$ to some $X_i \in \md^\omega(R_i)$ and the map $e$ to a
self-map
$e_i\colon X_i \to X_i$ such that $e_i^2$ is homotopic to $e_i$. 
As is classical, we use the idempotent $e_i$ to split $X_i$; see \cite[Prop. 1.6.8]{neeman} or the older \cite{Freyd2} and \cite[Th.
5.3]{Freyd}.
Explicitly, form the 
filtered colimit $Y_i$ of $X_i  \stackrel{e_i}{\to} X_i \stackrel{e_i}{\to}
\dots$, 
which splits off $X_i$.
The tensor product  $R
\otimes_{R_i} Y_i$ is the direct summand of $X$ given by the idempotent $e$ and
is therefore equivalent to $M$. 

The association $\mathcal{C} \mapsto \Pic(\mathcal{C})$ commutes with filtered
colimits of symmetric monoidal $\infty$-categories by \Cref{piclimits}. 
Taking Picard groups in the equivalence \eqref{filteredcolim}, the proposition follows. 
\end{proof}

Purely algebraic information can be used to begin approaching $\pic(R)$.
Let $\pic(R_*)$ be the Picard group of the symmetric monoidal category of
\emph{graded} $R_*$-modules. 
The starting point of \cite{BakerRichter} 
is the following. 

\begin{construction}
There is a monomorphism $$\Phi\colon
\pic(R_*)\to \pic(R),$$ constructed as follows. 
If $M_*$ is an invertible $R_*$-module, it has to be finitely generated and
projective of rank one. Consequently, there is a finitely generated free
$R_*$-module $F_*$ of which $M_*$ is a direct summand, i.e. there is a
projection $p_*$ with a section $s_*$,
$ \xymatrix{ F_* \ar[r]_{p_*}  & M_* \ar@/_/[l]_{s_*}  }. $

Clearly, $F_*$ can be realized as an $R$-module $F$ which is a finite wedge sum of copies of $R$ or its suspensions. Let $e_*$ be the idempotent given by composition $s_*\circ p_*$. Since $F$ is free over $R$, $e_*$ can be realized as an $R$-module map $e\colon F\to F$ which must be idempotent. Define $M$
to be the colimit of the sequence
\( F \stackrel{e}{\to} F \stackrel{e}{\to} \dots,  \)
i.e. the image of the idempotent $e$. Observe that the homotopy groups of $M$
are given by $M_*$, as desired. If $M'_*$ is the inverse to $M_*$ in the
category of graded $R_*$-modules, we can construct an analogous $R$-module
$M'$, and clearly $M \otimes_R M' \simeq R$ by the degeneration of the
K\"unneth spectral sequence. Thus, $M \in \pic(R)$. 
The association $M_* \mapsto M$ defines $\Phi$. 

Note that any two $R$-modules
that realize $M_*$ on homotopy groups are equivalent by the degeneration of the
$\mathrm{Ext}$-spectral sequence, and that $\Phi$ is a homomorphism by the
degeneration 
of the K\"unneth spectral sequence.
Observe also that $\Phi$ is clearly a monomorphism as equivalences of $R$-modules are detected on homotopy groups. 
\end{construction}

\begin{definition} 
When $\Phi$ is an isomorphism, we say that $\pic( R)$ is \emph{algebraic}.
\end{definition} 

 Baker-Richter \cite{BakerRichter} determine certain conditions which imply algebraicity. 
There are, in particular, two fundamental examples. 
The first one generalizes \Cref{pics0}.

\begin{theorem}[{Baker-Richter \cite{BakerRichter}}]
\label{connectivepic}
Suppose $R$ is a connective $\e{\infty}$-ring.  
Then the Picard group of $R$ is algebraic. 
\end{theorem} 

\begin{proof} 
Since the formulation in \cite[Theorem 21]{BakerRichter} assumed a coherence hypothesis on
$\pi_* R$, we explain briefly how this (slightly stronger)  version can be
deduced from the theory of flatness of \cite[\S 8.2.2]{higheralg}. Recall that an
$R$-module $M$ is \emph{flat} if $\pi_0 M$ is a flat $\pi_0 R$-module and the
natural map
\[ \pi_* R \otimes_{\pi_0 R} \pi_0 M \to \pi_* M,  \]
is an isomorphism. 

Since the Picard group
commutes with filtered colimits in $R$, we may assume that $R$ is finitely
presented in the $\infty$-category of connective $\e{\infty}$-rings: in
particular, by \cite[Proposition 8.2.5.31]{higheralg}, $\pi_0 R$ is a finitely generated $\mathbb{Z}$-algebra and in
particular noetherian; moreover, each $\pi_j R$ is  a finitely generated
$\pi_0 R$-module. These are the properties that will be critical for us. 

Let $M$ be an invertible $R$-module. We will show that $\pi_* M$ is a flat
module over $\pi_* R$, which immediately implies the claim of the theorem. 
Localizing at a
prime ideal of $\pi_0 R$, we may assume that $\pi_0 R$ is a noetherian local
ring; in this case we will show the Picard group is $\mathbb{Z}$ generated by
the suspension of the unit. 
We saw that $M$ is perfect, so we  can
assume by shifting that $M$ is connective and that $\pi_0 M \neq 0$. 
Now for every map\footnote{Recall that we are using the same symbol to
denote an abelian group 
and its Eilenberg-MacLane spectrum.} $R \to k$, for $k$
a field, $\pi_*(M \otimes_R k)$  is
necessarily concentrated in a single degree: in fact, $M \otimes_R k$ is an invertible
object in $\md(k)$ and one can apply the K\"unneth formula to see that
$\pic(\md(k)) \simeq \mathbb{Z}$ generated by $\Sigma k$.  
By Nakayama's lemma, since $\pi_0 M \neq 0$, the homotopy groups of $M \otimes_R k
$ must be concentrated in degree zero. Thus, $M \otimes_R k \simeq k$ itself. 
Using \Cref{lem:flatness} below, it follows that $M$ is equivalent to $R$ as an
$R$-module, so we are done. 
\end{proof} 

\begin{lemma} 
\label{lem:flatness}
Let $R$ be a connective $\e{\infty}$-ring with $\pi_0 R$ noetherian local
with residue field $k$. Suppose moreover each $\pi_i R$ is a finitely
generated $\pi_0 R$-module.
Suppose $M$ is a connective,\footnote{i.e. $(-1)$-connected} perfect $R$-module. Then, for $n \geq 0$, the following are
equivalent: 
\begin{enumerate}
\item $M \simeq R^n$.  
\item $M \otimes_R k \simeq k^n$. 
\end{enumerate}
\end{lemma} 
\begin{proof} 
Suppose  $M \otimes_R k$ is 
isomorphic to $k^n$ and concentrated in degree zero.
Note that $\pi_0( M \otimes_R k) \simeq \pi_0 M \otimes_{\pi_0 R} k$. 
Choose a basis $\overline{x_1}, \dots, \overline{x_n} $ of this $k$-vector
space and lift these elements to $x_1, \dots, x_n \in \pi_0 M$. 
These define a map $R^n \to M$ which induces an equivalence after tensoring
with $k$, since $M \otimes_R k \simeq k^n$.

Now consider the cofiber $C$ of $R^n \to M$. It follows that $C \otimes_R k$ is
contractible. 
Suppose $C $ itself is not contractible. 
The hypotheses on $\pi_* R $ imply that $C$ is connective and each $\pi_j C$
is a finitely generated module over the noetherian local ring $\pi_0 R$. 
If $j$ is chosen minimal such that $\pi_j C \neq 0$, then 
\[ 0 = \pi_j(C \otimes_R k)  \simeq \pi_j C \otimes_{\pi_0 R } k , \]
and Nakayama's lemma implies that $\pi_j C = 0$, a contradiction. 
\end{proof}

Some of our analyses in the computational sections will rest upon the next result about the Picard groups
of \emph{periodic} ring spectra. 

\begin{theorem}[{Baker-Richter \cite[Theorem 37]{BakerRichter}}] 
\label{evenperiodicreg}
Suppose $R$ is a weakly even periodic $\e{\infty}$-ring with $\pi_0 R$ regular
noetherian. Then the Picard group of $R$ is algebraic. 
\end{theorem} 

The result in \cite[Theorem 37]{BakerRichter} actually assumes that $\pi_0 R$
is a complete 
regular local ring. However, one can remove the hypotheses by replacing $R$
with the localization $R_{\mathfrak{p}}$ for any $\mathfrak{p} \in \spec\, \pi_0
R$ and then forming the completion at the maximal ideal. 

We will 
need a slight strengthening of \Cref{evenperiodicreg}, though. 
\begin{corollary} \label{even2}
Suppose $R$ is an $\e{\infty}$-ring satisfying the following assumptions.
\begin{enumerate}
\item 
$\pi_0 R$ is regular
noetherian. 
\item The $\pi_0 R$-module $\pi_{2k} R$ is invertible for some $k >0$. 
\item $\pi_{i} R = 0$ if $i \not\equiv 0  \ \mathrm{mod} \ 2k$. 
\end{enumerate}
Then the Picard group  of $R$ is algebraic. 
\end{corollary} 
\begin{proof} 
Using the obstruction theory of \cite{angeltveit} (as well as localization), we can construct ``residue
fields'' in $R$ as $\e{1}$-algebras in $\md(R)$ (which will be $2k$-periodic
rather than $2$-periodic). After this, the same argument
as in \Cref{evenperiodicreg} goes through. 
\end{proof} 

\begin{remark}\label{rem:AlgGradPic}
If $R$ is a ring spectrum satisfying the conditions of \Cref{even2}, then $\pic( R) \cong \pic(\pi_*R)$ sits in a short exact sequence
\[ 0 \to \pic(\pi_0 R ) \to \pic(\pi_*R) \to \Z/(2k) \to 0.\]
The extension is such that the $(2k)$-th power of a set-theoretic lift of a generator of $\Z/(2k)$ to $\pic(\pi_* R)$ is identified with the invertible $\pi_0 R$-module $\pi_{2k} R$.
\end{remark}

An example of a non-algebraic Picard group, based on \cite[Example
7.1]{rational}, is as follows. 

\begin{proposition} 
\label{picrational}
The Picard group of the rational $\e{\infty}$-ring
$R = \mathbb{Q}[\epsilon_0, \epsilon_{-1}]/\epsilon_0^2$ (free on two generators
$\epsilon_0, \epsilon_{-1}$ of degree $0$ and $-1$, and with the relation
$\epsilon_0^2 = 0$) is given by $\mathbb{Z} \times \mathbb{Q}$.
\end{proposition}
\begin{proof}
The key observation is that $R$ is equivalent, as an $\e{\infty}$-ring, to
cochains over $S^1$ on the (discrete) $\e{\infty}$-ring
$\mathbb{Q}[\epsilon_0]/\epsilon_0^2$, because $C^*(S^1; \mathbb{Q})$ is
equivalent to $\mathbb{Q}[\epsilon_{-1}]$. 
By \cite[Remark 7.9]{galois}, we have a fully faithful,  symmetric monoidal embedding $\md(R) \subset
\loc_{S^1}( \md( \mathbb{Q}[\epsilon_0]/\epsilon_0^2))$ into the 
$\infty$-category of local systems (see \Cref{def:localsystems} below) of
$\mathbb{Q}[\epsilon_0]/\epsilon_0^2$-modules over the circle, whose image
consists of those
local systems of $\mathbb{Q}[\epsilon_0]/\epsilon_0^2$-modules such that the
monodromy action of $\pi_1(S^1)$ is ind-unipotent. 

In particular, to give an object in
$\pic(R)$ is equivalent to giving an element in $\pic(
\mathbb{Q}[\epsilon_0]/\epsilon_0^2)$ (of which there are only the suspensions
of the unit, by \Cref{connectivepic}) and an ind-unipotent
(monodromy) automorphism, which is necessarily given by
multiplication by $1 + q \epsilon_0$ for $q \in \mathbb{Q}$. 
We observe that this gives the right group structure to the Picard group
because $(1 + q \epsilon_0) ( 1 + q' \epsilon_0) = 1 + (q + q') \epsilon_0$.
\end{proof}

\Cref{picrational} provides a counterexample to \cite[Conjecture
74]{balmerICM}, which states that in a tensor triangulated category generated
by the unit with a local spectrum (e.g. with no nontrivial thick
subcategories), any element $\mathcal{L}$ in the Picard group has the property
that $\mathcal{L}^{\otimes n}$ is a suspension of the unit for suitable $n >
0$. In fact, one can take the (homotopy) category of perfect $R$-modules for
$R$ as in \Cref{picrational}, which has no nontrivial thick subcategories
 by \cite[Theorem 1.3]{rational}.

\begin{remark}
Other Picard groups of interest come from the theory of \emph{stable module $\infty$-categories}
of a $p$-group $G$ over a field $k$ of characteristic $p$, which from a
homotopy-theoretic perspective can be expressed as the module
$\infty$-categories of the Tate construction $k^{tG}$. The Picard groups of
stable module $\infty$-categories have been studied in the modular
representation theory literature (under the name \emph{endotrivial
modules}) starting with \cite{Dade2}, where it is proved that the
Picard group is algebraic (and cyclic) in the case where $G$ is elementary
abelian. The 
classification for a general $p$-group appears in
\cite{Torsionendo}. 
\end{remark}
\section{The descent spectral sequence}

In this section, we describe a descent spectral sequence for calculating Picard
groups. The spectral sequence (studied originally by Gepner and Lawson
\cite{GL} in a closely related
setting) is based on the observation (\Cref{piclimits}) that the association
$\mathcal{C} \mapsto \Pic( \mathcal{C})$, from symmetric monoidal
$\infty$-categories to $\e{\infty}$-spaces, commutes with homotopy limits. 
We will describe several examples and applications of this in the present
section. Explicit computations will be considered in later parts of this paper. 

For example, let $ \{\mathcal{C}_U\}$ be a sheaf of symmetric monoidal
$\infty$-categories on a site, and let $\Gamma( \mathcal{C})$ denote the global
sections (i.e. the homotopy limit) $\infty$-category. Then we  have an equivalence
of connective spectra
\[ \pics( \Gamma( \mathcal{C})) \simeq  \tau_{\geq 0} \Gamma( \pics(
\mathcal{C}_U)),  \]
and one can thus use the descent spectral sequence for a sheaf of spectra to
approach the computation of 
$\pics( \Gamma( \mathcal{C}))$. 
We will use this approach, together with a bit of descent theory,
to calculate $\pic( \TMF)$. The key idea is that while $\TMF$ itself has
sufficiently complicated homotopy groups that results such as \Cref{evenperiodicreg}
cannot apply, the $\infty$-category of $\TMF$-modules is built up as an inverse
limit of module categories over $\e{\infty}$-rings with better behaved homotopy
groups. 


\subsection{Refinements}
Let $X$ be a Deligne-Mumford stack equipped with a flat map $X \to M_{FG}$ to the moduli stack of formal groups. 
We will use the terminology of \cite{affine}. 
\begin{definition}
An \emph{even periodic refinement} of $X$ is a sheaf $\otop$ of
$\e{\infty}$-rings on the affine, \'etale site of $X$, such that for any \'etale
map
\[ \spec R \to X,  \]
the multiplicative homology theory associated to the $\e{\infty}$-ring
$\otop(\spec R)$ is functorially identified with the (weakly) even-periodic
Landweber-exact
theory\footnote{See \cite[Lecture 18]{luriechromatic} for an exposition of the theory of weakly
even-periodic theories.}
associated to the formal group classified by $\spec R \to X \to M_{FG}$. 
We will denote the refinement of the ordinary stack $X$ by $\mathfrak{X}$. 
\end{definition}

A very useful construction from the refinement $\mathfrak{X}$ is the
$\e{\infty}$-ring of ``global sections'' $\Gamma(\mathfrak{X}, \otop)$, which
is the homotopy limit of the $\otop(\spec R)$ as $\spec R \to X$ ranges over
the affine \'etale site of $X$. 

\begin{example} 
When $X$ is the moduli stack $\mell$ of elliptic curves, with the natural map
$\mell \to M_{FG}$ that assigns to an elliptic curve its formal group, 
fundamental work of Goerss, Hopkins, and Miller, and  (later) Lurie  constructs an even
periodic refinement $\mellb$. The global sections of $\mellb$ are defined to be
the $\e{\infty}$-ring $\TMF$ of \emph{topological modular forms}; for a survey,
see \cite{Goerss}. 
There is a similar picture for the compactified moduli stack $\mellc$, whose
global sections are denoted $\Tmf$. 
\end{example}

\begin{definition}
Given the refinement $\mathfrak{X}$, one has a natural symmetric monoidal
stable $\infty$-category
$\qcoh(\mathfrak{X})$ of \emph{quasi-coherent sheaves} on $\mathfrak{X}$, given
as a homotopy limit of the (stable symmetric monoidal) $\infty$-categories $\mod(\otop(\spec R))$ for each \'etale
map $\spec R \to X$. 
\end{definition}
There is an adjunction
\begin{equation} \label{adj} \mod( \Gamma( \mathfrak{X}, \otop))
\rightleftarrows \qcoh( \mathfrak{X}),  \end{equation}
where the left adjoint ``tensors up'' and the right adjoint takes global
sections.\footnote{One way to extract this from \cite{higheralg} is to consider
the thick subcategory $\mathcal{C}$ of $\qcoh( \mathfrak{X}, \otop)$ generated
by the unit. Then, one obtains by the universal property of $\mathrm{Ind}$ an
adjunction $\mathrm{Ind}(\mathcal{C}) \rightleftarrows \qcoh( \mathfrak{X},
\otop)$. However, the symmetric monoidal
$\infty$-category $\mathrm{Ind}(\mathcal{C})$ 
is generated under colimits by the unit, so it is by Lurie's
symmetric monoidal version \cite[Prop. 8.1.2.7]{higheralg}
of Schwede-Shipley theory equivalent to modules over $\Gamma( \mathfrak{X},
\otop)$, which is the ring of endomorphisms of the unit.} 

Our main goal in this paper is to investigate the left hand side; however, the
right hand side is sometimes easier to work with, since even periodic,
Landweber-exact spectra
have convenient properties. Therefore, the following result 
 will be helpful.

\begin{theorem}[{\cite[Theorem 4.1]{affine}}]
\label{affinethm}
Suppose  $X$ is noetherian and separated, and $X \to M_{FG}$ is quasi-affine. Then the adjunction \eqref{adj}  is an
equivalence of symmetric monoidal $\infty$-categories. 
\end{theorem} 

For example, since the map $M_{ell} \to M_{FG}$ is affine, it follows that $\md(\TMF)$ is
equivalent to $\qcoh( \mellb)$. This was originally proved by Meier, away from
the prime $2$, in
\cite{meier}. \Cref{affinethm} implies the analog for $\Tmf$ and the derived
\emph{compactified} moduli stack, as well \cite[Theorem 7.2]{affine}. 

Suppose $X \to M_{FG}$ is quasi-affine. 
In particular, it follows that there is a \emph{sheaf} of symmetric monoidal
$\infty$-categories on the affine, \'etale site of $X$, given by 
\[ ( \spec R \to X)  \to \md( \otop( \spec R)),  \]
whose global sections  are given by $\md( \Gamma( \mathfrak{X}, \otop))$. 
This diagram of $\infty$-categories is a sheaf in view of the descent theory of
\cite[Theorem 6.1]{DAGVII}, but \cite[Theorem 4.1]{affine} gives the global sections. 
We
are now in the situation of the introduction to this section. 
In particular, we obtain a descent spectral sequence for $\pics( \Gamma( X,
\otop))$, and we turn to studying it in detail.

\subsection{The Gepner-Lawson spectral sequence}
Keep the notation of the previous subsection: $X$ is a Deligne-Mumford stack
equipped with a quasi-affine flat
map $X \to M_{FG}$, and $(\mathfrak{X}, \otop)$ is an even periodic refinement. 

Our goal in this subsection is to prove:

\begin{theorem} 
\label{descentss}
Suppose that $X$ is a regular Deligne-Mumford stack with a quasi-affine flat map
$X \to M_{FG}$, and suppose $\mathfrak{X}$ is an even periodic refinement of
$X$. 
There is a spectral sequence with
\begin{equation} \label{eq:descentss}E_2^{s,t} = 
\begin{cases}
H^s( {X}, \mathbb{Z}/2) & t = 0  ,\\
H^s( X, \mathcal{O}_X^{\times}) & t  = 1    , \\
H^s( X, \omega^{(t-1)/2 }) &  t \geq 3 \ \text{odd},  \\
0 & \text{otherwise,} \\
\end{cases}
\end{equation}
whose abutment  is $\pi_{t-s} \Gamma( \mathfrak{X}, \pics( \otop))$. 
The differentials run $d_r\colon E_r^{s,t} \to E^{s + r, t + r -1}$.
\end{theorem}

The analogous spectral sequence for a faithful Galois extension has been studied in work
of Gepner and Lawson \cite{GL}, and our approach is closely based on theirs. 

\begin{proof}
In this situation, as we saw in the previous subsection, we get an equivalence of symmetric monoidal $\infty$-groupoids,
\[ \Pic( \Gamma(\mathfrak{X}, \otop)) \simeq \mathrm{holim}_{\spec R \to X}
\Pic( \otop( \spec R)),  \]
where $\spec R \to X$ ranges over the affine \'etale maps. Equivalently, we
have an equivalence of connective spectra
\[  \pics( \Gamma(\mathfrak{X}, \otop)) \simeq \tau_{\geq
0}\left( \mathrm{holim}_{\spec R \to X}
\pics (\otop( \spec R)) \right).  \]

Let us study the descent spectral sequence associated to this. 
We need to understand the homotopy group \emph{sheaves}  of the sheaf of
connective spectra $(\spec R \to X) \mapsto \pics( \otop( \spec R))$ 
(i.e. the
sheafification of the homotopy group presheaves $(\spec R \to X) \mapsto \pi_i
 \pics( \otop( \spec R))$). 
First, we know that
\[ \pi_1 \pics(\otop( \spec R)) \simeq R^{\times}, \]
and, for $i \geq 2$, we have
\[ \pi_i \left( \pics( \otop(
\spec R) \right) \simeq \pi_{i-1} \otop( \spec R) = \begin{cases} 
\omega^{(i-1)/2} & \text{$i$ odd } \\
0 & \text{$i$ even.}
 \end{cases} \]

It remains to determine the homotopy group sheaf $\pi_0$. 
If $X$ is a regular Deligne-Mumford stack, so that each ring $R$ that enters is
regular, then we can do this 
using \Cref{evenperiodicreg}. 
In fact, it follows that if $R$ is a local ring, then $\pi_0
\pics(\otop( \spec
R)) \simeq \mathbb{Z}/2$. Thus, up to suitably suspending once, invertible sheaves are
locally trivial. 
Using the descent spectral sequence for a sheaf of spectra, we get that 
the above descent spectral sequence for $\Gamma(
\mathfrak{X}, \pics( \otop))$ is almost entirely the same as the descent
spectral sequence for $\Gamma( \mathfrak{X}, \otop)$ in the sense that the cohomology
groups that appear for $t \geq 3$, i.e. $H^s(X, \omega^{(t-1)/2})$, are the same as
those that appear in the descent spectral sequence for 
$\Gamma( \mathfrak{X}, \otop)$. However, the terms for $t = 1$ are the
\'etale cohomology of $\mathbb{G}_m$ on $X$. 
In particular, 
we obtain the term
\[ H^1(X, \mathcal{O}_X^{\times}) \simeq \pic(X),  \]
which
is the Picard group of the underlying ordinary stack. 
\end{proof}

\newcommand{\Picss}{\mathcal{P}\mathrm{ic}^{\mathbb{Z}}}
\newcommand{\picss}{\pics^{\mathbb{Z}}}

\begin{remark}
One may think of the spectral sequence as
arising from a totalization, or rather as a filtered colimit of totalizations. 
Choose an \'etale hypercover $\mathfrak{A}$ given by $U_\bullet \to X$ by affine schemes $\{U_n\}$. 
For any $\e{\infty}$-ring $A$, 
denote by $\Picss(A)$ the symmetric monoidal subcategory of $\Pic(A)$ spanned by
those $A$-modules such that, after restricting to  each connected component of
$\spec \,\pi_0 A$, become equivalent to a suspension of  $A$. Denote by $\picss(A)$ the associated connective
spectrum.
Then we form the totalization 
\[ \mathrm{Tot}( \picss(\otop(U_\bullet))),   \]
whose associated infinite loop space
$\Omega^\infty\mathrm{Tot}( \picss(\otop(U_\bullet)))$ is, by descent theory, the
symmetric monoidal
$\infty$-subgroupoid
of $\Pic( \Gamma(\mathfrak{X}, \otop))$ spanned by those invertible modules
which become (up to a suspension) trivial after pullback along $U_0 \to X$. 
In particular, the filtered colimit of these totalizations is the spectrum we
are after. 
The descent spectral sequence 
of \Cref{descentss} is the filtered colimit of these $\mathrm{Tot}$ spectral
sequences. 

\end{remark}
\subsection{Galois descent}

We next describe the setting of the spectral sequence that was originally
considered in  \cite{GL}. 
Let $A \to B$ be a faithful $G$-Galois extension of $\e{\infty}$-ring spectra
in the sense of \cite{rognes}. In particular, $G$ acts on $B$ in the
$\infty$-category of $\e{\infty}$-$A$-algebras and $A \to B^{hG}$ is an
equivalence. 
Then $A \to B$ is an analog of a $G$-Galois \'etale cover in the
sense of ordinary commutative algebra or algebraic geometry. 
As in ordinary algebraic geometry, there is a good theory of \emph{Galois descent} along $A \to B$,
as has been observed by several authors, for instance \cite{GL, meier}. 

\begin{theorem}[Galois descent] Let $A \to B$ be a faithful $G$-Galois
extension of $\e{\infty}$-rings. Then there is a natural equivalence of
symmetric monoidal $\infty$-categories $\md(A) \simeq \md(B)^{hG}$. 
\end{theorem}

The ``strength'' of the descent is in fact very good. As shown in
\cite[Theorem 3.36]{galois}, any faithful Galois extension $A \to B$ satisfies a form of
descent up to nilpotence: the thick tensor-ideal that $B$ generates in
$\md(A)$ is equal to all of $\md(A)$. This imposes strong restrictions on the
{descent spectral sequences} that can arise. 

Applying the Picard functor, we get an equivalence of spaces
\begin{equation} \label{gal1} \Pic(A) \simeq \Pic(B)^{hG},  \end{equation}
or an equivalence of \emph{connective} spectra
\begin{equation} \label{gal2} \pics(A) \simeq \tau_{\geq 0} \pics(B)^{hG}.  \end{equation}
\begin{remark}\label{rem:relativePic}
The spectrum $\Sigma \gl_1 B$ is equivalent to $\tau_{\geq 1}\pics (B)$; consider the induced map of $G$-homotopy fixed point spectral sequences. All the differentials involving the $t-s=0$ line will be the same for $\pics B $ and $\Sigma \gl_1 B$. Hence, we obtain a short exact sequence 
\[ 0 \to \pi_0 (\Sigma \gl_1 B)^{hG} \to \pi_0(\pics (B))^{hG}  \to E_\infty^{0,0} \to 0,\]
where $E_\infty^{0,0}$ is the kernel of all the differentials supported on
$H^0(G, \pi_0 \pics B)$. This short exact sequence exhibits
$\pi_0 (\Sigma \gl_1 B)^{hG}$ as the \emph{relative Picard group} of $A \to B$,
which consists of invertible $A$-modules which after smashing with $B$ become isomorphic to $B$ itself.
\end{remark}

Our main interest in Galois theory, for the purpose of this paper,  comes from
the observation, due to Rognes, that there are numerous examples of $G$-Galois
extensions of $\e{\infty}$-rings $A \to B$ where the homotopy groups of $B$ are
significantly simpler than that of $A$. 
In particular, one hopes to understand the homotopy groups of $\pics(B)$, and
then use \eqref{gal1} and \eqref{gal2} together with an analysis of the associated homotopy
fixed-point spectral sequence
\begin{align}\label{ss:galDescent}
H^s(G, \pi_t \pics (B)) \Rightarrow \pi_{t-s} (\pics (B))^{hG},
\end{align}
whose abutment for $t=s$ is the Picard group $\pic(A)$.

\begin{example}[{\cite[Proposition 5.3.1]{rognes}}] 
\label{KOKUGalois}
The map $KO \to KU$ and the $C_2$-action on $KU$ arising from complex
conjugation exhibit $KU$ as a $C_2$-Galois extension of $KO$. 
\end{example} 

\Cref{KOKUGalois} is fundamental and motivational to us: the study of
$KO$-modules, which is a priori difficult because of the complicated structure
of the ring $\pi_* KO$, can be approached via Galois descent together with the
(much easier) study of $KU$-modules. In particular, we obtain
\[ \pics(KO) \simeq \tau_{\geq 0} \pics(KU)^{hC_2},  \]
and one can hope to use 
the homotopy fixed-point spectral sequence (HFPSS) to calculate $\pics(KO)$. This approach is due to Gepner-Lawson
\cite{GL},\footnote{The original calculation of the Picard group of $KO$, by
related techniques, is
unpublished work of Mike Hopkins.} and we shall give a version of it below in \Cref{sec:PicKO} (albeit using a 
different method of deducing differentials). 

Other examples of Galois extensions come from the theory of topological modular
forms with \emph{level structure.}

\begin{example} 
Let $n \in \mathbb{N}$, and let $\TMF(n)$ denote the periodic version of $\TMF$
for elliptic curves over $\mathbb{Z}[1/n]$-algebras with a \emph{full level $n$ structure.} 
Then, by \cite[Theorem 7.6]{affine}, $\TMF[1/n] \to \TMF(n)$ is a faithful $GL_2(\mathbb{Z}/n)$-Galois
extension. The advantage is that, if $n \geq 3$, the moduli stack of
elliptic curves with level $n$ structure is actually a regular affine
scheme (by \cite[Corollary 2.7.2]{katzmazur}, elliptic curves with full level
$n \geq 3$ structure have no nontrivial automorphisms). In particular, $\TMF(n)$ is even periodic with regular $\pi_0$, and one
can compute its Picard group purely algebraically by \Cref{evenperiodicreg}. 
One can then hope to use $GL_2(\mathbb{Z}/n)$-descent to get at the Picard
group of $\TMF[1/n]$. We will take this approach below. 
\end{example} 

\subsection{The $E_n$-local sphere}

In addition, descent theory can be used to give a spectral sequence for $\pics(
L_n S^0)$. 
This is related to work of Kamiya-Shimomura \cite{KamiyaShimomura} and the
upper bounds that they obtain on $\pic(L_n S^0)$.

Consider the cobar construction on $L_n S^0 \to E_n$, i.e. the cosimplicial $\e{\infty}$-ring
\[ E_n \rightrightarrows E_n \wedge E_n \triplearrows \dots,  \]
whose homotopy limit is $L_n S^0$. It is a consequence of the Hopkins-Ravenel
smash product theorem 
\cite[Ch. 8]{ravenelorange}
that this cosimplicial diagram has ``effective descent.''

\begin{proposition}
The natural functor
\[ \md( L_n S^0) \to \mathrm{Tot} \left(\md( E_n^{\wedge(\bullet+1)})\right),  \]
is an equivalence of symmetric monoidal $\infty$-categories. 
\end{proposition}
\begin{proof} 
According to the Hopkins-Ravenel smash product theorem, the map of $\e{\infty}$-rings
$L_n S^0 \to E_n$ 
has the property that the thick tensor-ideal that $E_n$ generates in $\md(L_n
S^0)$ is all of $\md( L_n S^0)$.\footnote{The argument in \cite[Ch.
8]{ravenelorange} is stated for the uncompleted Johnson-Wilson theories, but
also can be carried out for the completed ones. We refer in particular to the
lecture notes of Lurie \cite{luriechromatic}; Lecture 30 contains the
necessary criterion for constancy of the $\mathrm{Tot}$-tower.}

According to \cite[Proposition 3.21]{galois}, this implies the
desired descent statement (the condition is there called ``admitting descent''). 
The argument is a straightforward application of the Barr-Beck-Lurie
monadicity theorem
\cite[\S 6.2]{higheralg}. 
\end{proof} 
In particular, we find that 
\[ \pics( L_n S^0) \simeq \tau_{\geq 0} \mathrm{Tot} \, \pics( E_n^{\wedge
(\bullet+1)}).  \]
Let us try to understand the associated spectral sequence. 

The higher homotopy groups $\pi_i, i \geq 2$ of $\pics (E_n^{\wedge(\bullet+1)})$ 
are determined in terms of those of $E_n^{\wedge(\bullet+1)}$. 
Once again, it remains to determine $\pi_0$. 
Now $E_n$ is an even periodic $\e{\infty}$-ring whose $\pi_0$ is regular local,
so $\pic(E_n) \simeq \pi_0 \pics(E_n) \simeq
\mathbb{Z}/2$ by \Cref{evenperiodicreg}. The iterated smash products
$E_n^{\wedge m}$ are also even
periodic, so their Picard group contains at least a $\mathbb{Z}/2$. We do not need 
to know their exact Picard groups, however, to run the spectral sequence, as
only the $\mathbb{Z}/2$ component is relevant for the spectral sequence (as it
is all that comes from $\pi_0 \pics(E_n)$).

Next, we need to determine the algebraic Picard group. 
After taking $\pi_0$, the simplicial scheme
\[ \dots \triplearrows \spec \pi_0(E_n \wedge E_n) \rightrightarrows \spec
\pi_0 E_n , \]
is a presentation of the moduli stack $M_{FG}^{\leq n}$ of formal groups (over
$\mathbb{Z}_{(p)}$-algebras) of height at most $n$. 

\begin{proposition} 
$\pic( M_{FG}^{\leq n}) \simeq \mathbb{Z}$, generated by $\omega$. 
\end{proposition} 
\begin{proof} 
We use the presentation of $M_{FG}$ (localized at $p$) via the simplicial stack
\begin{equation} \label{presMFG}\dots \triplearrows ( \spec (MU \wedge MU)_*)/\mathbb{G}_m \rightrightarrows
(\spec MU_*)/\mathbb{G}_m . \end{equation}
Since the Picard group of  a polynomial ring over $\mathbb{Z}_{(p)}$ is
trivial,\footnote{Since the Picard group commutes with filtered colimits, one
reduces to the case of a polynomial ring on a finite number of variables, and
here it follows from unique factorization.} and each smash power of $MU$ has a polynomial ring for $\pi_*$,  the Picard group of each of the terms in the simplicial stack
\emph{without} the $\mathbb{G}_m$-quotient is
trivial, and the group of units is $\mathbb{Z}_{(p)}^{\times}$, constant
across the simplicial object. 
In other words, the \emph{Picard groupoid} 
of each $\spec (MU^{\wedge (s+1)})_*$ is $B \mathbb{Z}_{(p)}^{\times}$. 
When we add the $\mathbb{G}_m$-quotient, we get $\mathbb{Z} \times B
\mathbb{Z}_{(p)}^{\times}$ for the Picard groupoid of each term in the simplicial stack
because of the possibility of twisting by a character of $\mathbb{G}_m$: this
twisting corresponds to the powers of $\omega$. By descent theory, this shows that $\pic(M_{FG})
\simeq \mathbb{Z}$, generated by $\omega$. More precisely, the Picard groupoid
of $M_{FG}$ is the totalization of the Picard groupoids of $\spec (MU^{\wedge
(s+1)})_*/\mathbb{G}_m$, and each of these is $\mathbb{Z} \times B
\mathbb{Z}_{(p)}^{\times}$: that is, the cosimplicial diagram of Picard
groupoids is constant and the totalization is $\mathbb{Z} \times B
\mathbb{Z}_{(p)}^{\times}$ again.

When we replace $M_{FG}$ by $M_{FG}^{\leq n}$, we can replace the above
presentation by excising from each term the closed substack cut out by $(p,
v_1, \dots, v_{n})$. This does not affect the Picard \emph{groupoid} since the
codimension of the substack removed is at least $2$ (i.e. neither the Picard
group nor the group of units is affected).\footnote{Once again, this is a
familiar result for regular rings, and here one must pass to filtered colimits
since one is working with polynomial rings on infinitely many variables.} That is, when we modify each
term in \eqref{presMFG} to form the associated presentation of $M_{FG}^{\leq
n}$, the Picard groupoid is unchanged. It follows by faithfully flat descent that the inclusion
$M_{FG}^{\leq n} \to M_{FG}$ induces an isomorphism on Picard groups (or
groupoids) and that the Picard group is generated by $\omega$. 
\end{proof} 

We obtain the following result.
\begin{theorem} There is a spectral sequence
\[ E_2^{s,t} = 
\begin{cases}
\mathbb{Z}/2 & t =  0 ,\\
H^s( M_{FG}^{\leq n}, \mathcal{O}_{M_{FG}}^{\times}) & t = 1    , \\
H^s( M_{FG}^{\leq n}, \omega^{(t-1)/2 }) &  t \geq 3 \ \text{odd},  \\
0 & \text{otherwise,} \\
\end{cases}
\]
which converges for $t-s \geq 0$ to $\pi_{t-s}\pics(L_n S^0)$. The relevant occurrences of
the second case are $H^0(M_{FG}^{\leq n}, \mathcal{O}_{M_{FG}}^{\times}) \simeq
\mathbb{Z}_{(p)}^{\times}$ and $H^1(M_{FG}^{\leq n},
\mathcal{O}_{M_{FG}}^{\times})
\simeq \mathbb{Z}$.
\end{theorem}

Note in particular that the $E_2$-term is determined entirely in terms of the
Adams-Novikov spectral sequence for the $E_n$-local sphere. 
As we will see in \Cref{sec:SSLemma}, many of the differentials are also
determined by the ANSS.

\section{First examples}\label{sec:firstexamples}

In this section, we will give several examples where descent theory gives a 
quick calculation of the Picard group. In these examples, we will not need to
analyze differentials in the descent spectral sequence \eqref{ss:galDescent}.
The main examples of interest, where there will be a number of 
differentials to determine, will be treated in the last part
of this paper.

\subsection{The faithfully flat case}

We begin with the simplest case. Suppose $R \to R'$ is a morphism of
$\e{\infty}$-rings which is faithfully flat. 
In this case, we know from \cite[Theorem 6.1]{DAGVII} that the tensor-forgetful adjunction $\md(R)
\rightleftarrows \md(R')$ is comonadic and we get a descent spectral sequence
for the Picard group of $R$, as 
\[ \pics(R) \simeq \tau_{\geq 0}\mathrm{Tot}\; \pics( R'^{\otimes (\bullet + 1)}).  \]

This spectral sequence, however, gives essentially no new information that is
not algebraic in nature. That is, the entire $E_2$-term $E_2^{s,t}$ for $t> 1$
vanishes, as it can be identified with the $E_2$-term for the cobar resolution
$R'^{\otimes (\bullet + 1)}$ of $R$, and this cobar resolution has a degenerate
spectral sequence with non-zero terms only for $s =0$ at $E_2$. 
For example, an element in $\pic(R)$ is algebraic if \emph{and only if} its
image in $\pic(R')$ is algebraic,  by faithful flatness. 

Thus, faithfully flat descent will be mostly irrelevant to us as a tool of
computing the non-algebraic parts of Picard groups. 
In the examples of interest, we want
$\pi_* R'$
to be significantly simpler homologically than $\pi_* R$, so that we will be
able to conclude (using results such as \Cref{evenperiodicreg}) 
that the 
Picard group of $R'$ is entirely algebraic. But if $\pi_*R'$ is faithfully
flat over $\pi_*R$, it cannot be much simpler homologically. (Recall for
example that \emph{regularity} descends under faithfully flat extensions of
noetherian rings.)

\subsection{Cochain $\e{\infty}$-rings and local systems}
In this subsection, we give another example of a family of $\e{\infty}$-ring
spectra whose Picard groups can be determined, or at least bounded. 

Let $X$ be a space and $R$ an $\e{\infty}$-ring. 
Let $R^X = C^*(X; R)$ be the $\e{\infty}$-ring of $R$-valued cochains on
$X$. 

\begin{definition} 
\label{def:localsystems}
Let $\loc_X( \md(R)) = \mathrm{Fun}(X, \md(R))$ denote the $\infty$-category of
\emph{local systems} of $R$-module spectra on $X$.
\end{definition} 

Then we have a fully
faithful embedding of symmetric monoidal $\infty$-categories
\[ \md^\omega( R^X)  \subset \loc_X(\md(R)), \]
which sends $R^X$ to the constant local system at $R$ and is determined by that. 
As
discussed in \cite[\S 7]{galois}, this embedding is often useful for relating invariants
of $R^X$ to those of $R$.  
In particular, since any invertible $R^X$-module is perfect, we have a fully faithful functor
of $\infty$-groupoids
\[ \Pic( R^X) \to \Pic( \loc_X( \md(R))) = \map( X, \Pic(\md(R))),  \]
where the last identification follows because $\Pic$ commutes with homotopy
limits (\Cref{piclimits}). 
Thus, we get the following useful \emph{upper bound} for the Picard group of
$R^X$. 

\begin{proposition} 
If $R$ is an $\e{\infty}$-ring and $X$ is any space, then $\pic( R^X)$ is
a subgroup of $\pi_0( \pics(R)^X)$.
\end{proposition} 

Without loss of generality, we will assume that $X$ is connected. 
Note that we have a cofiber sequence
\[ \Sigma \mathfrak{gl}_1(R) \to \pics(R) \to H( \pic(R)),  \]
where $H(\pic(R))$ denotes the Eilenberg-MacLane spectrum associated to the
group $\pic(R)$. 
If we take the long exact sequence after taking maps from $X$, we get
an exact sequence
\begin{equation} \label{RXexact} 0 \to \pi_{-1}( \gl_1(R)^X) \to  \pi_0 ( \pics(R)^X)  \to
\pic(R) .  \end{equation}
 Our object of interest, $\pic(R^X)$, is a subobject of the middle
term, by the above proposition.

Let us unwind the exact sequence further. 
First, the composite map $\pic(R^X) \to \pi_0( \pics(R)^X) \to \pic(R)$ comes from the map of $\e{\infty}$-rings $R^X \to R$ given by
choosing a basepoint of $X$. In particular, it is \emph{split surjective} as it has
a section given by $R \to R^X$ (so \eqref{RXexact} is a split exact sequence). 
Next, observe that, using the truncation map $\gl_1(R) \to H R_0^{\times}$, we have a
map $\pi_{-1}( \gl_1(R)^X) \to \pi_{-1}( (H
R_0^{\times})^X) = \hom( \pi_1(X), R_0^{\times})$. 
We can understand this map in terms of $\pic(R^X)$. 
Very explicitly, suppose given an invertible $R^X$-module $M$ with 
associated local system $\mathcal{L} \in \loc_X(\md(R))$. 
Then if the image of $M$ in $\pic(R)$ is trivial, we conclude that
$\mathcal{L}_x \simeq R$ for any basepoint $x \in X$. An element in
$\pi_1(X, x)$ induces a monodromy automorphism of $\mathcal{L}_x$ and
thus defines an element of $R_0^{\times}$. This defines a map in $\hom(
\pi_1(X, x), R_0^{\times})$. 
Let $\pic^0(R^X)$ be the kernel of $\pic(R^X) \to \pic(R)$. Then we have just
described the map
\begin{equation} \label{phiL}\pic^0(R^X) \stackrel{\phi}{\to} \hom( \pi_1(X, x), R_0^{\times}),
\end{equation}
that comes from the exact sequence \eqref{RXexact}.

The monodromy action cannot be arbitrary, since  this local system is not
arbitrary: it is in the image of $\md^\omega(R^X)$ and therefore belongs to
the thick subcategory generated by the unit. 
As in \cite[\S 8]{galois}, it follows that the monodromy action of any element of the fundamental group must be
\emph{ind-unipotent}.
In particular, 
fix an element $M$ of $\pic^0(R^X)$.
Given any loop $\gamma \in \pi_1(X, x)$, the associated
element $u = u_{\gamma, M} \in R_0^{\times}$ under the homomorphism $\phi(M)\colon \pic^0(R^X) \to
\hom(\pi_1(X, x), R_0^{\times})$ of \eqref{phiL} must have the property
that $u-1$ is nilpotent.

Hence if $R_0$ is a \emph{reduced} ring, we deduce from \eqref{RXexact}
the following conclusion.

\begin{corollary} \label{picRX}
If $R$ is an $\e{\infty}$-ring with $\pi_0 R$ reduced, and $X $ is any
connected space, then we have a split short exact sequence
\[ 0 \to A \to \pic(R^X) \to \pic(R) \to 0,  \]
where $A \subset \pi_{-1}( \gl_1(R)^X)$ is actually contained in $\pi_{-1}(
(\tau_{\geq 1} \gl_1(R))^X) \subset \pi_{-1}(
(\gl_1(R))^X) $.
In particular, if 
$\pi_{-1}(
(\tau_{\geq 1} \gl_1(R))^X) = 0$, then $\pic(R) \to \pic(R^X)$ is an
isomorphism. 
\end{corollary} 

Again, we note that the map $\pi_{-1}( (\tau_{\geq 1} \gl_1(R))^X) \to
\pi_{-1}(  \gl_1(R)^X)$ is injective, by the long exact sequence and the fact
that $\pi_0 ( \gl_1(R)^X) \to \pi_0 (( H R_0^{\times})^X ) \simeq R_0^{\times}$ is surjective.

As an application, we obtain a calculation of the Picard group of a
nonconnective $\e{\infty}$-ring in a setting
far from regularity. 

\begin{theorem} 
Let $A$ be any finite \emph{abelian} group and let $E_n$ be Morava $E$-theory.
Then the Picard group of $E_n^{BA}$
is $\mathbb{Z}/2$, generated by the suspension $\Sigma E_n^{BA} $. 
The same conclusion holds 
for any finite group $G$ whose $p$-Sylow subgroup is abelian, where $p$ is the prime of definition for $E_n$.
\end{theorem} 
\begin{proof} 
We induct on the $p$-rank of $A$. When $A $ has no $p$-torsion, then $E_n^{BA} \simeq E_n$ and
\Cref{evenperiodicreg} implies that the Picard group is $\mathbb{Z}/2$. 

If the $p$-rank of $A$ is positive, write
$A \simeq \mathbb{Z}/p^m \times A'$ where the $p$-rank of $A'$ has smaller
cardinality than that of $A$. The inductive
hypothesis gives 
us that the Picard group of $E_n^{BA'}$ is $\mathbb{Z}/2$. Now
$E_n^{BA} \simeq (E_n^{BA'})^{B\mathbb{Z}/p^m}$. Moreover, $E_n^{BA'}$ is
well-known to be even periodic (though its $\pi_0$ is not regular).\footnote{We
refer to \cite[\S 7]{HKR} for a general analysis of the question of when $E_n^{BG}$
is even-periodic for $G$ a finite group.} 

We claim now that $\pi_{-1}( (\tau_{\geq 1} \gl_1( E_n^{BA}))^{B
\mathbb{Z}/p^m}) = 0$. To see this, we observe that the homotopy
groups of $\tau_{\geq
1}\gl_1(E_n^{BA'})$ are concentrated in even degrees and are all
given by torsion-free $p$-complete abelian groups. 
Therefore, the cohomology groups $H^i(\mathbb{Z}/p^m, \pi_j 
\tau_{\geq 1} \gl_1( E_n^{BA}))$ \emph{vanish} if $i$ is odd, since the
$\mathbb{Z}/p^m$-action on them is trivial. In the homotopy fixed point spectral
sequence for 
$(\tau_{\geq 1} \gl_1( E_n^{BA}))^{B
\mathbb{Z}/p^m} $ (i.e. the Atiyah-Hirzebruch spectral sequence), there is no
room for contributions to $\pi_{-1}$. In fact, there is no room for
differentials at all, which indicates that any $\lim^1$ terms cannot occur
either. Now
\Cref{picRX} shows that the map $E_n^{BA'} \to E_n^{BA}$ induces an equivalence 
on Picard groups, which completes the inductive step. 

For the last claim, fix any finite group $G$ with an abelian $p$-Sylow
subgroup $A \subset G$. For any connected space $X$, denote as before
$\pic^0(R^X)$ the kernel of $\pic(R^X) \to \pic(R)$. We have a commutative square
\[ \xymatrix{
\pic^0(E_n^{BG}) \ar[r]\ar@{^{(}->}[d]  &  \pic^0(E_n^{BA}) \ar@{^{(}->}[d]  \\
\pi_{-1}( \tau_{\geq 1} \gl_1(E_n)^{BG}) \ar[r] & \pi_{-1}( \tau_{\geq 1} \gl_1(E_n)^{BA})
}\]
The bottom horizontal map is injective since $\tau_{\geq 1}\gl_1(E_n)$ is
$p$-local and $BG$ is $p$-locally a wedge summand of $BA$ in view of the
transfer $\Sigma^\infty_+ BG \to \Sigma^\infty_+ BA$, which has the property
that the composite $\Sigma^\infty_+ BG \to \Sigma^\infty_+BA
\to \Sigma^\infty_+BG$ is a $p$-local equivalence by inspection of $p$-local homology. 
It follows that $\pic^0(E_n^{BG}) \to \pic^0( E_n^{BA})$ is injective, and
since the latter is zero, the former must be as well. 
\end{proof}

Recall that the spectrum $E_1$ is $p$-complete complex $K$-theory. 
\begin{proposition} 
Let $G$ be any finite group. Then the Picard group of $E_1^{BG}$ is
finite. 
\end{proposition} 
\begin{proof} 
In fact, $\pi_{-1}\left(\tau_{\geq 1} \gl_1(E_1)^{BG}\right)$ is finite. 
We know that $\tau_{\geq 3}\gl_1(E_1) \simeq \Sigma^4 {ku}\hat{ _p}$ by a
theorem of Adams-Priddy \cite{AdamsPriddy}.
Moreover, $({ku}\hat{ _p})^*(BG)$ is finite in each odd dimension, by
comparing with $E_1^*(BG)$ which vanishes in odd dimensions. It follows now from
\Cref{picRX} that the desired Picard group has to be finite. 
\end{proof}

\begin{question}
Let $G$ be any finite group. Can the Picard group of $E_1^{BG}$ be any larger
than $\mathbb{Z}/2$? What about the higher Morava $E$-theories?
\end{question}

\subsection{Coconnective rational $\e{\infty}$-rings}

We can also determine the Picard groups 
of coconnective rational $\e{\infty}$-ring spectra. 
A rational $\e{\infty}$-ring $R$ is said to be \emph{coconnective} if
\begin{enumerate}
\item $\pi_0 R$ is a field (of characteristic zero). 
\item $\pi_i R = 0$ for $i > 0$. 
\end{enumerate}

\begin{introtheoremD} 
If $R$ is a coconnective rational $\e{\infty}$-ring, then 
the Picard group $\pic(R) \simeq \mathbb{Z}$, generated by $\Sigma R$.
\end{introtheoremD} 
\begin{proof} 
Let $k = \pi_0 R$. 
We use \cite[Proposition 4.3.3]{DAGVIII} to conclude that $R \simeq
\mathrm{Tot}( A^\bullet )$, where  $A^\bullet$ is a cosimplicial
$\e{\infty}$-$k$-algebra with each $A^i$ of the form $k \oplus V[-1]$, where $V$ is a
discrete $k$-vector space; the $\e{\infty}$-structure given is the
``square-zero'' one. 

We thus begin with the case of $R = k \oplus V[-1]$: we will show that
$\pic(R) \simeq \mathbb{Z}$ in this case. 
Since $\pic$ commutes with filtered colimits, we may assume that $V$ is a
finite-dimensional vector space. In this case, 
\[ R \simeq k^{S^1 \vee \dots \vee S^1},  \]
where the number of copies of $S^1$ in the wedge summand is equal to the
dimension $n = \dim_k V$; by \cite[Proposition 4.3.1]{DAGVIII}, any rational
$\e{\infty}$-ring with these homotopy groups is equivalent to  $k \oplus V[-1]$. 
But we can now use \Cref{picRX} to see that the Picard group of $k^{S^1 \vee \dots
\vee S^1}$ is $\mathbb{Z}$, generated by the suspension, because $\tau_{\geq 1}
\gl_1( k) = 0$.

Now suppose that $R$ is arbitrary. As above, we have an equivalence $R \simeq
\mathrm{Tot}(A^\bullet)$ where each $A^i$ is a coconnective $\e{\infty}$-ring
of the form $k \oplus V[-1]$ for $V$ a discrete $k$-vector space. We have seen
above that $\pic(A^i) \simeq \mathbb{Z}$. 
We know, moreover, that we have a fully faithful embedding
of symmetric monoidal $\infty$-categories
\[ \md^\omega(R) \subset  \mathrm{Tot}( \md(A^\bullet)),  \]
which implies that we have a fully faithful functor
of $\infty$-groupoids
\[ \Pic(R) \to \mathrm{Tot}( \Pic(A^\bullet)).  \]
But each $\Pic(A^i)$, as an $\infty$-groupoid, has homotopy groups given by 
\[ \pi_j \Pic(A^i) \simeq \begin{cases} 
\mathbb{Z}  & j = 0 \\
k^{\times} & j =1
 \end{cases} ,\]
and in particular, in the cosimplicial diagram $\Pic(A^\bullet)$, all the maps
are \emph{equivalences.} This is a helpful consequence of coconnectivity. Thus, we find that $\mathrm{Tot}(\Pic(A^\bullet))$ maps
by equivalences to each $\Pic(A^i)$, and we get an upper bound of
$\mathbb{Z}$ for $\Pic(R)$. This upper bound is realized by the suspension
$\Sigma R$ (which hits the generator of $\mathbb{Z} \simeq \pi_0 \Tot(
\Pic(A^\bullet))$).
\end{proof} 

\begin{remark} 
If $k = \mathbb{Q}$, then a large class of coconnective $\e{\infty}$-rings with
$\pi_0 \simeq \mathbb{Q}$
(e.g. those with reasonable finiteness hypotheses and vanishing $\pi_{-1}$)
arise as cochains on a simply connected space, by Quillen-Sullivan's rational
homotopy theory. The comparison with local systems can be carried out directly
here to prove Theorem~\ref{thm:charzero} for these $\e{\infty}$-rings.
\end{remark}

\subsection{Quasi-affine cases}

We now consider a case where the descent spectral sequence
enables us to \emph{produce} nontrivial elements in the Picard group. 
Let $A$ be a weakly even-periodic $\e{\infty}$-ring with $\pi_0 A $ 
regular noetherian, and write $\omega = \pi_2 A$.
Then $A$ leads to a sheaf of $\e{\infty}$-rings
on the affine \'etale site of $\spec\, \pi_0 A$. That is, for every \'etale
$\pi_0 A$-algebra $A'_0$, there is (functorially) associated \cite[\S
8.5]{higheralg}  an $\e{\infty}$-ring $A'$ under $A$ with $\pi_0 A' \simeq A'_0$ and
$A'$ flat over $A$. We will denote this sheaf by $\otop$.

Now let $a_1, \dots, a_n \in \pi_0 A$ be a regular sequence, for $n \geq 2$. We consider the
complement $U$ in $\spec\, \pi_0 A $ of the closed subscheme $V(a_1, \dots,
a_n)$ and the sections $\overline{A} = \Gamma(U, \otop)$.  $\overline{A}$ is an
$\e{\infty}$-$A$-algebra and is a type of localization of $A$, albeit not
(directly) an
arithmetic one.\footnote{A piece of forthcoming work of Bhatt and Halpern-Leinster
identifies the universal property of $\overline{A}$.} Note that $\pic(A)$ is algebraic by \Cref{evenperiodicreg}, but
the situation for $\overline{A}$ is more complicated. 

The homotopy groups 
$\pi_*(\overline{A})$ are given by the abutment of a descent spectral sequence
\begin{align}\label{eq:AbarSS}
 H^s( U, \omega^{\otimes t}) \implies \pi_{2t - s}( \overline{A}).  
 \end{align}
We can first determine the zero-line. We have 
\[ H^0(U, \omega^{\otimes t}) = H^0(\spec\, \pi_0 A, \omega^{\otimes t}),  \]
because $\spec\, \pi_0 A$ is regular and $U \subset \spec\, \pi_0 A$ is obtained by removing a codimension
$\geq 2$ subscheme. 

\begin{proposition} 
The only other nonzero term in the descent spectral sequence \eqref{eq:AbarSS} occurs for $s =
n-1$. The descent spectral sequence degenerates. 
\end{proposition} 
\begin{proof} 
Cover the scheme $U$ by the $n$ open affine subsets $U_i = \spec \pi_0(A) \setminus
V(a_i)$, for $1 \leq i \leq n$. Given any quasi-coherent sheaf $\mathcal{F}$ on
$U$, it follows that the coherent cohomology $H^*(U, \mathcal{F})$ is that of the
\v{C}ech
complex (which starts in degree zero)
\[ \bigoplus_{i = 1}^n\mathcal{F}(U_i) \to \bigoplus_{i < j} \mathcal{F}(U_i
\cap U_j) \to \dots \to \mathcal{F}(U_1 \cap \dots \cap U_n).  \]
Let $R = \pi_0 A$, and suppose $\mathcal{F}$ is the restriction to $U \subset
\spec R$ of
the quasi-coherent sheaf
$\widetilde{M}$ on $\spec\, R$ for an $R$-module $M$. Then the final term is the cokernel of
the map
\[ \bigoplus_{i = 1}^n M[( a_1 \dots \hat{a_i} \dots a_n)^{-1}] \to M[(a_1
\dots a_n)^{-1}],  \]
where the hat denotes omission. If $M$ is flat, the complex is exact away from
degrees $0$ and $n-1$
as the sequence $a_1, \dots, a_n$ is regular, using a Koszul complex argument
(see \cite{24hrs} for a detailed treatment or \cite{GMcompletion} for a short
exposition with a view towards topological applications), and the zeroth cohomology is given by $M$
itself.

Now, in view of the map $A \to \overline{A}$, clearly everything in the
zero-line of the $E_2$-page of the spectral sequence survives, so the spectral
sequence
must degenerate. 
\end{proof} 

We now study the Picard group of $\overline{A}$: as above,
$\pi_* \overline{A}$ is not regular but instead has a great deal of
square-zero material. 
Let $\mathfrak{U} = (U, \otop|_{U})$ denote the derived scheme consisting
of the topological space $U \subset \spec\, \pi_0 A$, but equipped with the 
sheaf $\otop$ of $\e{\infty}$-rings restricted to $U$. 
$\overline{A}$ arises as the global sections of the structure sheaf $\otop$
over the derived scheme $\mathfrak{U}$. 

Since $U$ is quasi-affine as an (ordinary!) scheme, it follows by
\cite[Corollary 3.24]{affine} that 
the global sections functor is the right adjoint of an inverse equivalence
\[ \md( \overline{A}) \rightleftarrows \qcoh( \mathfrak{U}) ,  \]
of symmetric monoidal $\infty$-categories. 
In particular, the Picard group $\pic(\overline{A})$ 
can be computed as $\pic( \qcoh( \mathfrak{U}))$. 

As before, we have a descent spectral sequence
\eqref{eq:descentss} converging to $\pi_{t-s} \pics(\overline{A})$. But from
\eqref{eq:descentss}, we know that almost all of the terms at $E_2$ are
identified with the descent spectral sequence for $\pi_{*}\overline{A}$. 
In addition, we know that $H^1(U, \mathcal{O}_U^{\times}) \simeq
\pic(\pi_0 A)$, as $\pi_0 A$ is regular and the complement of $U$ has
codimension $\geq 2$. These classes must be permanent cycles as they are
realized in $\pic(\overline{A})$: in fact, they are realized in $\pic(A)$
itself. 
Thus, the descent spectral sequence for $\pics$ degenerates as well. We get
three contributions to the Picard group: $\mathbb{Z}/2$ and $\pic(\pi_0
A)$, which together build $\pic(\pi_*A)$ (compare \Cref{rem:AlgGradPic}),
and a group that is identified with $\pi_{-1} \overline{A}$. 
The relevant extension problem is solved because of the map 
$\pic(\pi_* A) \cong \pic(A) \to \pic(\overline{A})$
realizing the algebraic part of the Picard group. We get: 

\begin{theorem} 
\label{q-affine}
Let $\overline{A} = \Gamma(U, \otop)$ as above. Then we have a natural
isomorphism
\[ \pic(\overline{A}) \simeq \pic( \pi_* A) \times
\pi_{-1}(\overline{A}).  \]
\end{theorem} 

Moreover, observe that
\begin{equation}\pi_{-1}(\overline{A})  = \begin{cases} 
\mathrm{coker}\left(  \bigoplus_{i = 1}^n \omega^{n/2-1}[(a_1 \dots \hat{a_i} \dots
a_n)^{-1}] \to \omega^{n/2-1}[(a_1 \dots a_n)^{-1}] \right) & n \geq 4\text{ even}  \\
0 & n \text{ odd}.
 \end{cases} \end{equation}


\begin{example} 
Let $A$ be a Landweber-exact weakly even periodic $\e{\infty}$-ring with $\pi_0 A$
regular noetherian; for instance, $A$ could be Morava $E$-theory $E_n$. 
In this case, we take $a_1, \dots, a_k  = p, v_1, \dots, v_{k-1}$, so that
$\overline{A} \simeq L_k A$. 
This gives Theorem~\ref{lnlexact} as a special case of \Cref{q-affine}. 
\end{example} 
\part{Computational tools}

\section{The comparison tool in the stable range}\label{sec:SSLemma}

This is a technical section in which we develop a tool that will enable us to
compare many of the differentials in a Picard  spectral sequence for Galois or
\'etale descent with the analogous differentials in the corresponding descent
spectral sequence before taking the Picard functor (i.e. for the
$\e{\infty}$-rings themselves). For example, in the Galois descent setting, we
are given a $G$-Galois extension $A\to B$, and we know the descent, i.e. homotopy fixed point, spectral sequence for $A\simeq B^{hG}$. The tool we develop in this section will allow us to deduce many differentials in the homotopy fixed point spectral sequence for $(\pics (B))^{hG}$.

For a spectrum or a pointed space $X$, and integers $a,b$, we denote by
$\tau_{\geq a} X$, $\tau_{\leq b} X$, and $\tau_{[a,b]} X$ the truncations of $X$ with homotopy groups in the designated range. 
Our main observation is that if $R$ is any $\e{\infty}$-ring, then for any $n \geq
2$, there is a natural equivalence of spectra
\[ \tau_{[n, 2n-1]} R \simeq \tau_{[n, 2n-1]} \gl_1(R).  \]
This equivalence is natural at the level of $\infty$-categories, and enables us
to identify a large number of differentials in descent spectral sequences for
$\gl_1$ and therefore also for $\pics$.
This observation, however, fails if we increase the range by $1$, and an
identification of the relevant discrepancy (as observed in such spectral
sequences) will be the subject of the following section and the 
formula \eqref{univformula}. 

The main result of \Cref{sec:truncated} is essentially a formulation of the
classical concept of the ``stable range''
in $\infty$-categorical terms, as can be seen from the fact that the major ingredients of the proof are Freudenthal's suspension theorem as well as the existence of Whitehead products in the unstable setting. Nonetheless, our formulation will be extremely
useful in the sequel. 

\subsection{Truncated spaces and spectra}\label{sec:truncated}
Throughout, $n \geq 2$.

\begin{definition}Let $\sp_{[n, 2n-1]} \subset \sp$ denote the $\infty$-category of spectra with homotopy
groups concentrated in degrees $[n, 2n-1]$. Let $\mathcal{S}_{*}$ denote the
$\infty$-category of pointed spaces, and let $\mathcal{S}_{*, [a,b]} \subset
\mathcal{S}_*$ denote
the subcategory spanned by those pointed spaces whose homotopy groups are
concentrated in the interval $[a,b]$. 
\end{definition}

The main goal of this subsection is to prove 
the following result identifying spaces and spectra whose homotopy groups are
concentrated in a range of dimensions.

\begin{theorem} \label{omegainfinity}
The functor $\Omega^\infty\colon \sp_{[n, 2n-1]} \to \mathcal{S}_*$ is fully
faithful. 
The functor $\Omega^\infty\colon \sp_{[n, 2n-2]} \to \mathcal{S}_{*, [n, 2n-2]}$ is
an equivalence of $\infty$-categories. 
\end{theorem} 
\begin{proof} 
Let $X, Y \in \sp_{[n, 2n-1]}$. 
We want to show that
the natural map
\begin{equation} \label{oinfinitymap}\hom_{\sp}(X, Y) \to \hom_{\mathcal{S}_*}(
\Omega^\infty X, \Omega^\infty Y)\end{equation}
is a homotopy equivalence. 
By adjointness, we can identify this with the map
\[ \hom_{\sp}(X, Y) \to \hom_{\sp}( \Sigma^\infty \Omega^\infty X, Y) 
\]
that arises from the counit map $\Sigma^\infty \Omega^\infty X \to X$.
Observe that we have a natural equivalence $\hom_{\sp}( \Sigma^\infty \Omega^\infty X , Y) \simeq \hom_{\sp}(
\tau_{\leq 2n-1}\Sigma^\infty \Omega^\infty X , Y)$ because $Y$ is
$(2n-1)$-truncated. In particular, to prove \Cref{omegainfinity}, it will
suffice to show that the natural map
of spectra
\[ \tau_{\leq 2n-1} \Sigma^\infty \Omega^\infty X \to X \simeq \tau_{\leq
2n-1} X, \]
is an equivalence, for any $X \in \sp_{[n, 2n-1]}$. 
Equivalently, we need to show that for any such spectrum $X$, the map
\begin{equation} \label{pikmap} \pi_k( \Sigma^\infty \Omega^\infty  X) \to
\pi_k(  X) \end{equation}
is an isomorphism for $k \leq 2n-1$. But we have maps of \emph{spaces}
\[ \Omega^\infty X \to \Omega^\infty \Sigma^\infty  \Omega^\infty X \to \Omega^\infty X,  \]
where the composite is the identity. 
The first map is the unit $Y \to \Omega^\infty \Sigma^\infty Y$ applicable for
any $Y \in \mathcal{S}_*$, and the second map is $\Omega^\infty$ applied to the
counit. By the Freudenthal suspension theorem, the first map induces an
isomorphism on homotopy groups $\pi_k$ for $ k \leq 2n-1$, and therefore the second
map does as well. This proves the claim that \eqref{pikmap} is an equivalence 
and the first part of the theorem. 

The functor
$\Omega^\infty\colon \sp_{[n, 2n-1]} \to \mathcal{S}_{*, [n, 2n-1]}$ is not 
essentially surjective, because spaces with homotopy groups concentrated in degrees $[n,
2n-1]$ can still have \emph{Whitehead products}, and spaces with nontrivial
Whitehead products can never be in the image of $\Omega^\infty$. 
However, we claimed in the statement of the theorem that the functor $\Omega^\infty\colon \sp_{[n, 2n-2]} \to \mathcal{S}_{*, [n, 2n-2]}$ is
an equivalence of $\infty$-categories. 
To show this, it suffices to show that the
functor is essentially surjective. 

Given a pointed space $X$ with homotopy groups in the desired range, 
we suppose inductively (on $k$) that $\tau_{\leq k} X$ is in the image of
$\Omega^\infty$. 
If $k \geq 2n-2$, then we are done. 
Otherwise, we have a pullback square
\[ \xymatrix{
\tau_{\leq k+1}X \ar[d] \ar[r] & \ast \ar[d]  \\
\tau_{\leq k} X \ar[r] &  K( \pi_{k+1} X, k+2).
}\]
Observe that the three pointed spaces 
$\tau_{\leq k} X, K(\pi_{k+1} X, k+2)$, and $\ast$ are all in the image of
$\Omega^\infty$ (the first by the inductive hypothesis), and $K(\pi_{k+1} X, k+2) \in \mathcal{S}_{*, [n, 2n-1]}$. Moreover, 
the \emph{maps} in the diagram are in the image of $\Omega^\infty$ by the previous
part of the result. Therefore, the object $\tau_{\leq k+1} X$  is in the image
of $\Omega^\infty$, as $\Omega^\infty$ preserves homotopy fiber squares. 
\end{proof}

Given an integer $k$, we could precompose the functor of \Cref{omegainfinity} with
the equivalence $\Omega^k\colon \sp_{[n+k, 2n+k-1]} \to \sp_{[n,2n-1]}$, and obtain
the following:
\begin{corollary}
For any integer $k$, the functor $\Omega^{\infty+k}\colon\sp_{[n+k,
2n+k-1]}\to \mathcal{S}_*$ is fully faithful.
\end{corollary}

\subsection{Comparisons for $\e{\infty}$-rings}
Our basic example for all this comes from the spectrum $\gl_1(R)$ associated to
an $\e{\infty}$-ring $R$, and the comparison between the two. This comparison
is the main obstacle in understanding the descent spectral sequence for the
Picard group: it is generally easier to understand descent spectral
sequences for the $\e{\infty}$-rings themselves (e.g. for $\TMF$). 

We emphasize again that given an $\e{\infty}$-ring $R$, 
the \emph{spectra} $R$ and $ \gl_1(R)$ are generally very different, and for an illustration we provide the following example. 
\begin{example}[T. Lawson \cite{tyler_MO}] 
\label{tylerexample}
Consider the commutative differential graded algebra $\mathbb{F}_2[x]/x^3$
where $|x| = 1$ and $dx = 0$ (so $d \equiv 0$). Let $R$ be the associated
$\e{\infty}$-ring under $\mathbb{F}_2$. Then $\gl_1(R)$ has homotopy groups in
dimensions $1,2$ given by $\mathbb{F}_2$; however, they are connected by
multiplication by $\eta$. In particular, $\gl_1(R)$ is not an
$\mathbb{F}_2$-module spectrum. 

More generally, let $R$ be the $\e{\infty}$-ring associated to the commutative 
differential graded algebra $\mathbb{F}_2[x]/x^3$ where $|x|=n$, $dx = 0$. 
$R$ can also be constructed by applying the Postnikov section $\tau_{\leq 2n}$
to the free $\e{\infty}$-$\mathbb{F}_2$-algebra on a class in degree $n$. 
Then $\pi_n(\gl_1(R)) \simeq \pi_{2n}( \gl_1(R)) \simeq \mathbb{F}_2$ and all
the other homotopy groups of $\gl_1(R)$ vanish. 
Therefore, $\gl_1(R)$ is the fiber of a $k$-invariant map $H \mathbb{F}_2[n]
\to H \mathbb{F}_2[2n+1]$. 
In this case, we can identify the $k$-invariant and thus identify $\gl_1(R)$. 
\newcommand{\sq}{\mathrm{Sq}}
\begin{proposition} 
\label{gl1cubezero}
Given $R$ as above, the $k$-invariant of $\gl_1(R)$ is given by the map \[
\mathrm{Sq}^{n+1}\colon H
\mathbb{F}_2[n] \to H \mathbb{F}_2[2n+1].\]
\end{proposition} 
\begin{proof} 
We first argue, following Lawson, that $\gl_1(R) $ cannot be the spectrum
$H\mathbb{F}_2[n] \vee H \mathbb{F}_2[2n]$. 
In fact, in this case, the map of spectra $H \mathbb{F}_2[n] \to \gl_1(R)$
would, by adjointness \cite{5author}
lead to a map of $\e{\infty}$-rings
\[ \Sigma^\infty_+ K( \mathbb{F}_2, n) \to R,  \]
carrying the class in $\pi_n K(\mathbb{F}_2, n)$ to the nonzero class in $\pi_n
R$. Smashing with $H \mathbb{F}_2$, we would get a map
of $\e{\infty}$-$H \mathbb{F}_2$-algebras
\[ H \mathbb{F}_2 \wedge \Sigma^\infty_+ K(\mathbb{F}_2, n) \to R  \]
with the same property. 
Now $\pi_n( H\mathbb{F}_2 \wedge \Sigma^\infty_+ K(\mathbb{F}_2, n)) \simeq \mathbb{F}_2$, 
with the nontrivial class coming from $\pi_n( K( \mathbb{F}_2, n))$. However,
this class squares to zero by \cite[Lemma 6.1, Ch. 1]{homology_iter} while the
nonzero class in $\pi_n R$ does not square to zero. 
This is a contradiction and proves that such a map cannot exist. 
Consequently, the $k$-invariant map for $\gl_1(R)$ must be nontrivial.

On the other hand, we know $\Omega^\infty \gl_1(R) \simeq K( \mathbb{F}_2, n) \times K(
\mathbb{F}_2, 2n)$ because $\Omega^\infty \gl_1(R)$ is the connected component
at $1$ of $\Omega^\infty R$. In particular, the $k$-invariant $H
\mathbb{F}_2[n] \to H \mathbb{F}_2[2n+1]$ defines upon applying
$\Omega^\infty$ the trivial  cohomology class in $H^{2n+1}( K( \mathbb{F}_2,
n); \mathbb{F}_2)$. 

So, for the $k$-invariant of $\gl_1(R)$, we need a nonzero element $\phi$ of
degree $n+1$ in the (mod 2) Steenrod algebra
such that, if $\iota_n \in H^n( K(\mathbb{F}_2, n); n)$ is the tautological
class, then $\phi \iota_n = 0$. By the calculation of the cohomology of
Eilenberg-MacLane spaces \cite{Serrecoh} (see also \cite[Ch. 9]{MT} for a textbook reference), the only
possibility is $\sq^{n+1}$.
\end{proof} 

\end{example}

Nonetheless, we will show that right below the range of the previous example,
the spectra $\gl_1(R)$ and $R$ can be identified. 

\begin{corollary} 
\label{gl1RvsR}
Let $n \geq 2$ and let $R$ be any $\e{\infty}$-ring. Then there is an
equivalence of spectra, functorial in $R$,
\[ \tau_{[n, 2n-1]} \gl_1(R) \simeq \tau_{[n, 2n-1]} R.  \]
Similarly, there is an equivalence of spectra, functorial in $R$,
\[ \tau_{[n+1, 2n]} \pics(R) \simeq  \Sigma\tau_{[n, 2n-1]} R. \]
\end{corollary} 
\begin{proof} 
For any $\e{\infty}$-ring $R$, the space $\Omega^\infty \gl_1(R) = GL_1(R)$ is a union of those components of
$\Omega^\infty R$ that correspond to units in $\pi_0 R$. In particular,
$\Omega^\infty \tau_{\geq 1}\gl_1(R)$ is \emph{canonically} identified with
$\Omega^\infty \tau_{\geq 1}  R$ in $\mathcal{S}_*$.
Applying \Cref{omegainfinity}, we now get a canonical identification as desired
in the corollary. 
The second half of 
\Cref{gl1RvsR} follows from the first, as $\tau_{\geq 0}\Omega \pics(R) \simeq
\gl_1(R)$ as spectra. 
\end{proof}

Take now a faithful $G$-Galois extension $A\to B$ of $\e{\infty}$-rings, and
consider the HFPSS \eqref{ss:galDescent} for the $G$-action on $\pics (B)$. 
We want to understand $\pi_0 ( \pics(B)^{hG})$, or 
equivalently $\pi_{-1}( \Omega \pics(B)^{hG})$, and we can do this by
understanding the HFPSS for the $G$-action on $\Omega \pics (B)$.
Observe first that $\pi_t   \Omega\pics(B) \simeq \pi_{t}B$ functorially for $t
\geq 1$: in fact, $\Omega^\infty(\Omega \pics(B)) \simeq GL_1(B)$. 
In other words, the spectrum $\Omega \pics(B)$ equipped with the $G$-action has
the property that, after applying $\Omega^\infty$,
it is identified with a union of connected
components of $\Omega^\infty B$ (with the $G$-action on $B$). 

As a result, we have a map of spaces with $G$-action
\[  \Omega^\infty ( \Omega \pics(B)) \to \Omega^\infty B, \]
which identifies the former with a union of connected components of the latter. 
As a result, we can identify 
the respective HFPSS for the spaces $\Omega^\infty ( \Omega \pics(B))$, $\Omega^\infty B$
for $t > 0$, both at $E_2$ and differentials (including the ``fringed''
ones). This identification comes from the map $\tau_{\geq 1} GL_1(B) \to
\Omega^\infty B $ given by subtracting one.

In particular, shifting by one again, most of the differentials in the HFPSS for $\pics(B)$ are
determined by the HFPSS for $B$. More precisely, any differential out of $E_r^{s,t}$ for $t-s
>  0$, $s>0$, 
depends only on the $G$-space $\Omega  \Pic(B)$, so the equivalence of $\Omega
\Pic(B)$ with a union of connected components of $\Omega^\infty B$ implies
that the differential \emph{can be identified} with the analogous
differential in the HFPSS for $B$.

However, to understand $\pi_0 ( \pics(B)^{hG}) \simeq \pi_0 ( \Pic(B)^{hG})
\simeq \pic(A)$, we need to determine differentials out of $E_r^{s,t}$ with $t
= s$. 
These differentials cannot be determined by $\Omega \Pic(B)$, as a space with a
$G$-action. 
Our strategy to determine these differentials is to use the equivalence of
spectra with $G$-action
\[ \tau_{[n+1, 2n]}\pics (B) \simeq  \Sigma \tau_{[n, 2n-1]} B,  \]
which is a special case of  \Cref{gl1RvsR}. 

Assume that  $r\leq t-1$. 
In this case, any differential $d_r\colon E_*^{s,t} \to E_*^{s+r, t+r-1}$ in the
HFPSS for $\pics(B)$ is determined by the $G$-action on  $\tau_{[t,
t+r-1]}\pics(B)$. Since we have an equivalence $\tau_{[t,t+r -1]}\pics(B) \simeq \Sigma
\tau_{[t-1, t+r - 2]} B$, compatible with the $G$-actions, we can identify the
differentials. 

Denote the differentials in the homotopy fixed point spectral sequence
\[ H^s(G, \pi_t \pics B) \Rightarrow \pi_{t-s} (\pics B)^{hG}  \]
by $d_r^{s,t}(\pics B)$, and similarly $d_r^{s,t}(B)$ for those in the HFPSS for $B$.
The upshot of this discussion is the following. 

\begin{tool}\label{toolGalois}
Let $A\to B$ be a $G$-Galois extension of $\e{\infty} $-rings. Whenever $2 \leq
r\leq t-1$, we have an equality of differentials $d_r^{s,t}(\pics B) = d_r^{s, t-1}( B)$.
\end{tool}

Of course, we also have an identification of differentials out of $(s,t)$ if
$t-s > 0, s>0$. 

\begin{remark}
Our original approach to the Comparison Tool~\ref{toolGalois} was somewhat more complicated than
the above and has been described in \cite{tots}. 
Namely, our strategy was to identify the HFPSS with a Bousfield-Kan spectral
sequence for a certain cosimplicial space $X^\bullet$ built from $\Pic(B)$
with its $G$-action, and  argue that these differentials only depended on the
fiber of $\mathrm{Tot}_{t+r}(X^\bullet) \to \mathrm{Tot}_{t-1}(X^\bullet)$
(as well as the other fibers in between). In
the appropriate range, these fibers depend only on $\Omega X^\bullet$ as a
cosimplicial space. However, $\Omega X^\bullet$ can be  (almost) identified with the
analogous cosimplicial space for the $G$-action on $\Omega^{\infty -
1}(\tau_{\geq 0}B)$
because $\Omega \Pic(B) $ is a union of components of $\Omega^\infty B$. 
This forces the differentials to correspond to one another. 
\end{remark}

For the same reasons, we have analogous comparison results for the
spectral sequence as \Cref{descentss}. Again, any differential in the
descent spectral sequence for $\pics( \Gamma( X, \otop))$ that only
depends on the \emph{diagram} $\tau_{[n+1, 2n]} \pics(\otop)$ can be identified
with the corresponding differential in the descent spectral sequence for
$\Gamma(X, \otop)$, thanks to the equivalence of \emph{diagrams} of spectra
$\tau_{[n+1, 2n]} \pics(\otop) \simeq \Sigma \tau_{[n, 2n-1]} \otop$. 

\begin{remark}
The equivalence $\tau_{[n, 2n-1]} R \simeq \tau_{[n, 2n-1]} \gl_1(R)$ 
resembles  the following observation in commutative algebra. Let $A$ be an
ordinary commutative ring and let $I \subset A$ be a square-zero ideal. 
Then $1 + I \subset A^{\times}$ and there is an isomorphism of groups
\[ I \simeq 1 + I \subset A^{\times}, \quad x \mapsto 1 + x . \]
This correspondence is a very degenerate version of the exponential and logarithm. 

Suppose $p$ is a prime number and $(p-1)!$ is invertible in $A$. Then if $J
\subset A$ is an ideal with $J^p = 0$, we have $1 + J \subset A^{\times} $ and
a natural isomorphism of groups
\[ J \simeq 1 + J , \quad x \mapsto 1 + x + \frac{x^2}{2} + \dots +
\frac{x^{p-1}}{(p-1)!},  \]
given by a $p$-truncated exponential. 

Similarly, let $R$ be an $\e{\infty}$-ring with $(p-1)!$ invertible. Motivated by the above, for any $n\geq
1$, one could surmise a
\emph{functorial} equivalence of spectra $\tau_{[n, pn-1]} R \simeq \tau_{[n, pn-1]}
\gl_1(R)$.
We expect to construct such an equivalence in ongoing joint work with Clausen and Heuts.

\end{remark} 

\subsection{A general result on Galois descent}

As a quick application of the preceding ideas, we can prove a general result
about  Galois descent for Picard groups.

\begin{introtheoremE} 
Let $A \to B$ be a faithful $G$-Galois extension of $\e{\infty}$-rings. Then the relative Picard
group of $B/A$ is $|G|$-power torsion of finite exponent. 
\end{introtheoremE} 
\begin{proof} 
We know that the relative Picard group of $A \to B$ is given by $\pi_{-1}(
\gl_1(B)^{hG})$ (compare \Cref{rem:relativePic}). 
There is a HFPSS that converges to the homotopy groups, which begins with the
group cohomology of $G$ with coefficients in $\pi_* ( \gl_1(B))$. Every
contributing term is $|G|$-power torsion: in fact, every term is a $H^i(G,
\cdot)$ for $i>0$ and is thus killed by $|G|$. 
However, in view of the potential infiniteness of the filtration, as well as
the possibilities of nontrivial extensions, this alone
does not force $\pi_{-1}( \gl_1(B)^{hG})$ to be $|G|$-power torsion.  

Our strategy is to compare the HFPSS for $\pi_{-1}( \gl_1(B)^{hG})$ with that
of $\pi_{-1}( B^{hG})$. The map $A \to B$ admits descent in the sense of
\cite[Definition 3.17]{galois}. In particular, by \cite[Corollary
4.4]{galois}, the descent spectral sequence for $A \to B$
(equivalently, the HFPSS) has a \emph{horizontal vanishing line} at a finite
stage. It follows that, above a certain filtration, everything in the HFPSS
for $\pi_*(A) \simeq \pi_*(B^{hG})$ is killed by a
$d_k$ for $k$ bounded. 

In view of our Comparison Tool~\ref{toolGalois}, it follows that any class in the relative
Picard group has bounded filtration (though possibly the bound is weaker than
the analog in $\pi_{-1}(B)$). Since every contributing term in the spectral sequence is
killed by $|G|$, the theorem follows. 
\end{proof}


\section{The first unstable differential}

\subsection{Context}
Let $R^\bullet$ be a cosimplicial $\e{\infty}$-ring, and 
consider the Bousfield-Kan spectral sequences  (BKSS)
$\left\{E_{r}^{s,t}\right\}$ and $
\left\{\overline{E}_{r}^{s,t}\right\}$ for the two cosimplicial objects
 $R^\bullet$ and 
 $\gl_1(R^\bullet)$, converging to $\pi_{t-s}$ of the
respective totalizations in $\sp$. 

For $t -s \geq 0$, the spectral sequences and the differentials are mostly
identified with one another, as the space $\Omega^\infty \gl_1(R)$ is a union
of connected components of $\Omega^\infty R$. 
But for $t - s = -1$, 
we get differentials
\[ d_r\colon E_r^{ t+1, t} \to E_r^{t+r+1, t+r-1}, 
\quad
\overline{d}_r\colon \overline{E}_r^{ t+1, t} \to \overline{E}_r^{t+r+1, t+r-1}.
\]
These depend on more than the spaces $\Omega^\infty R^\bullet, \Omega^\infty
\gl_1(R^\bullet)$:
they require the one-fold deloopings. 
As we saw in \Cref{gl1RvsR}, for any $n \geq 2$, in the range $[n, 2n-1]$, the
cosimplicial \emph{spectra}
$\tau_{[n, 2n-1]} R^\bullet, \tau_{[n, 2n-1]} \gl_1(R^\bullet)$ are
identified. As a result, for $r \leq t$, the groups
in question are (canonically) identified and
$d_r = \overline{d}_r$. 

But in general, $d_{t+1} \neq \overline{d}_{t+1}$. Since all the previous
differentials entering or leaving this spot between the two spectral sequences were identified, 
the groups in question are identified. 
We let the correspondence $E_{t+1}^{t+1,t} \simeq \overline{E}_{t+1}^{t+1, t}$
be given as $x \mapsto \overline{x}$.
Similarly, we have a correspondence $E_{t+1}^{2t+2, 2t} \simeq
\overline{E}_{t+1}^{2t+2, 2t}$. 

In this subsection, we will give a universal formula for the first differential
out of the stable range. 
We will need this in \Cref{sec:TMF1/3} to obtain the 2-primary Picard group of $\TMF$.

\begin{theorem} 
\label{difftheorem}
We have the formula
\begin{equation} \label{univformula}
\overline{d}_{t+1}(\overline{x}) = \overline{d_{t+1}(x) + x^2}, \quad x \in E_{t+1}^{t+1, t}.
\end{equation} 
\end{theorem}

\begin{remark} 
The above formula actually makes $\overline{d}_{t+1}$ into a linear operator.
This follows from the graded-commutativity of the BKSS for $R^\bullet$. Note in
particular that the difference between $\overline{d}_{t+1}$ and $d_{t+1}$ is
annihilated by two.
\end{remark} 

\subsection{The universal example}
The proof of \eqref{univformula} follows a standard technique in algebraic
topology: we reduce to a
``universal'' case and show that \eqref{univformula} is essentially the only possibility. 
We want to consider the universal case of a cosimplicial $\e{\infty}$-ring
$R^\bullet$ with
a class in $E_{t+1}^{t+1, t}$. This class represents an element in $\pi_{-1}
\mathrm{Tot}_{2t+1}( R^\bullet)$ trivialized in $\mathrm{Tot}_{t}(R^\bullet)$; the differential $d_{t+1}$ represents the
obstruction to lifting to $\mathrm{Tot}_{2t+2}$. 
So, we need to make the analysis of differentials in the cosimplicial $\e{\infty}$-ring which corepresents the functor 
$R^\bullet \mapsto \mathfrak{A}(R^\bullet) = \Omega^\infty \left( \Sigma^{-1}\mathrm{fib}\left( \mathrm{Tot}_{2t+1}(R^\bullet) \to
\mathrm{Tot}_t(R^\bullet)\right) \right)$.

\newcommand{\lan}{\mathrm{Lan}}
\newcommand{\fp}{\mathscr{F}}
The relevant cosimplicial $\e{\infty}$-ring $\mathscr{X}^\bullet$ can be constructed as
follows. 
\begin{definition}
Let $\lan$ denote the operation of left Kan extension, and let
$\lan_{\Delta^{\leq t} \to \Delta}(\ast)$ denote the left Kan extension of the
constant functor $\Delta^{\leq t} \to \mathcal{S}$ at a point to $\Delta$. 
Similarly, define 
$\lan_{\Delta^{\leq 2t+1} \to \Delta}(\ast)$. Consider the homotopy pushout
\begin{equation} \label{defF} 
\xymatrix{
\lan_{\Delta^{\leq t} \to \Delta}(\ast)_+ \ar[d]  \ar[r] & \ast 
\ar[d] \\
\lan_{\Delta^{\leq 2t+1} \to \Delta}(\ast)_+ \ar[r] &  \fp^\bullet,
}\end{equation}
where $\fp^\bullet\colon \Delta \to \mathcal{S}_*$ is a functor to the $\infty$-category
$\mathcal{S}_*$ of \emph{pointed} spaces. 
\end{definition}
Consider $$\mathscr{G}^\bullet \stackrel{\mathrm{def}}{=}\Sigma^{\infty-1} \fp^\bullet\colon
\Delta \to \sp$$ and the functor
\newcommand{\CAlg}{\mathrm{CAlg}}
\newcommand{\X}{\mathscr{X}}
\[ \X^\bullet = \mathrm{Free}_{\CAlg}( \mathscr{G}^\bullet)\colon \Delta \to \CAlg,  \]
into the $\infty$-category $\CAlg$ of $\e{\infty}$-rings, obtained by applying
the free algebra functor everywhere to $\mathscr{G}$. 
Then $\X^\bullet$, by construction, corepresents the functor $\mathfrak{A}
\colon \mathrm{Fun}(\Delta, \CAlg) \to \mathcal{S}$ in which we are interested. In particular, it suffices to prove \eqref{univformula} for this particular
functor. As we will see in the next paragraph, $\mathscr{G}^\bullet$ takes values in
\emph{connective} spectra and therefore so does $\X^\bullet$. Since we are only interested in differentials in a particular range,
we may (by naturality) only consider the Postnikov section $\tau_{\leq 2t}
\X^\bullet$. 
We get the following basic step. 

\begin{proposition} 
In order to prove \Cref{difftheorem}, it suffices to prove it for
$\tau_{\leq 2t}\X^\bullet$ (and the tautological class). 
\end{proposition}

In fact, we have a reasonable handle on what the functor $\tau_{\leq 2
t}\X^\bullet$ looks like and can \emph{entirely} determine the BKSS. 
To see this, we recall the construction of $\fp^\bullet$; compare also the
discussion in \cite{tots}. 
The functor
\[ \lan_{\Delta^{\leq t} \to \Delta}(\ast) \colon \Delta \to \mathcal{S},  \]
sends any finite nonempty totally ordered set $T$ to the nerve of the category
$\Delta^{\leq t}_{/T}$
of all order-preserving morphisms $\left\{S \to T\right\}$ where: 
\begin{enumerate}
\item  $S$ is a finite, nonempty totally ordered set, and
\item $|S| \leq t+1$. 
\end{enumerate}

\begin{proposition}
$\lan_{\Delta_{\leq t} \to \Delta}(\ast)$
is naturally equivalent to the functor which sends $T \in \Delta$ to the nerve of the
\emph{poset} $P_{\leq t+1}(T)$ of  nonempty subsets of
$T$ of cardinality $\leq t+1$. 
\end{proposition}
\begin{proof}
In fact, for any $T$, there is a natural map $P_{\leq t+1}(T) \to
\Delta^{\leq t}_{/T}$, which is a homotopy equivalence as it is right adjoint
to the functor $\Delta^{\leq t}_{/T} \to P_{\leq t+1}(T)$ which sends
$S \to T$ to $\mathrm{image}(S \to T) \subset T$. 
\end{proof}

In view of the last proposition, one can also consider the following approach to the left Kan extension. 
There is a standard cosimplicial simplicial set sending $[n] \mapsto
\Delta^n$. The functor 
of the proposition
is equivalent to the barycentric subdivision of the
cosimplicial simplicial set $[n] \mapsto
\mathrm{sk}_t \Delta^n$.

As in \cite{tots}, the nerve of $P_{\leq t+1}(T)$, for any choice of $T$, is
(pointwise) homotopy equivalent to a
wedge of $t$-spheres, 
and contractible if $|T| \leq t+1$. 
We get from \eqref{defF}:

\begin{proposition}
The functor 
$\fp^\bullet\colon \Delta \to \mathcal{S}_*$ constructed above has the following properties: 
\begin{enumerate}
\item For any $T$, $\fp(T)$ is always a wedge of copies of $S^{t+1}$ and
$S^{2t+1}$.
\item Restricted to $\Delta^{\leq t}$, $\fp^\bullet$  is contractible. Restricted to
$\Delta^{\leq 2t}$, $\fp^\bullet$ is pointwise a wedge of copies of $S^{t+1}$.
\end{enumerate}
\label{coskfunctor}
\end{proposition}
\subsection{Some technical lemmas}
Our first goal is to  understand the BKSS for $\mathscr{G}^\bullet
= \Sigma^{\infty-1}
\fp^\bullet$. Observe that pointwise, this cosimplicial spectrum is a wedge of copies of $S^t$ and
$S^{2t}$ by \Cref{coskfunctor}.
In order to do this, we need to understand the cosimplicial abelian group
$\pi_* ( \Sigma^{\infty - 1} \fp^\bullet)$. 
We will prove the following:

\begin{proposition} \label{FBKSS}
The cohomology $H^s( \pi_*( \mathscr{G}^\bullet))$ is given by 
\begin{equation} 
H^s( \pi_*( \mathscr{G}^\bullet) )\simeq
\begin{cases} 
\pi_*  S^t & s = t+1, \\
\pi_* S^{2t} & s =  2(t+1).
 \end{cases} 
\end{equation} 
In the spectral sequence, the differential $d_{t+1}$ is an isomorphism. 
\end{proposition}

\begin{figure}[h]
\includegraphics[scale=1]{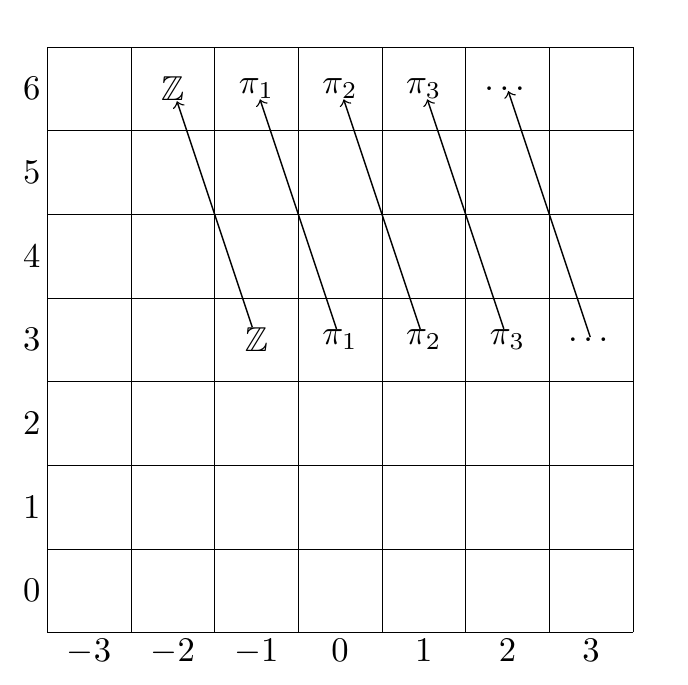}
\caption{Bousfield-Kan spectral sequence for $\mathscr{G}^\bullet$, with $t = 2$. }\centering{$\pi_k$
denotes $\pi_k S^0$}
\label{fig:SSF2}
\end{figure}

The spectral sequence is depicted in \Cref{fig:SSF2}. The proof of \Cref{FBKSS} will take work and will be spread over two
subsections. In the present subsection, our main result is that the
totalization of $\mathscr{G}^\bullet$ (and related cosimplicial
spectra) is contractible, and we will deduce the differentials from that. 
The approach to this is not computational and relies instead on ideas involving the
$\infty$-categorical Dold-Kan correspondence of Lurie.

We recall from {\cite[1.2.8.4]{highertopos}} the \emph{cone} construction,
which 
associates
to
a simplicial set $K$,
the \emph{cone} $K^{\lhd}$. If $K$ is an $\infty$-category, $K^{\lhd}$ is as
well, and
is obtained by adding a new initial object to $K$. 

\begin{lemma} \label{colimcone}
Let $K$ be a simplicial set and $\mathcal{D}$ an $\infty$-category with
colimits. Let $F\colon K^{\lhd} \to \mathcal{D}$ be a functor with the property
that  $F$ carries the cone point to an initial object of $\mathcal{D}$. Then
the natural map
\[ \varinjlim_{K} F|_{K} \to \varinjlim_{K^{\lhd}}  F \]
is an equivalence in $\mathcal{D}$.
\end{lemma} 
\begin{proof} 
It suffices to show\footnote{We are indebted to the referee for substantially
simplifying our original argument here.} that the natural map
\begin{equation} \label{naturalmapD} \mathcal{D}_{K^{\lhd}/} \to \mathcal{D}_{K/} \end{equation}
is an equivalence of $\infty$-categories. 
But we have $\mathcal{D}_{K^{\lhd}/} \simeq \mathcal{D}_{(\Delta^0 \star K)/}
\simeq (\mathcal{D}_{\Delta^0/})_{K/}$ in view of the definition of the
overcategory \cite[\S 1.2.9]{highertopos}, where $\star$ denotes the \emph{join} of simplicial sets
\cite[\S 1.2.8]{highertopos}. However, we also know that the projection
map $\mathcal{D}_{\Delta^0/} \to \mathcal{D}$ is an equivalence since $\Delta^0
\to \mathcal{D}$ maps to an initial object. Therefore, we obtain 
that \eqref{naturalmapD} is an equivalence, as desired. 
\end{proof}

\begin{lemma} \label{kancone}
Let $\mathcal{C}, \mathcal{D}$ be $\infty$-categories and assume that
$\mathcal{D}$ has colimits. Let $F\colon \mathcal{C}^{\lhd} \to \mathcal{D}$ be a
functor such that $F$ carries the cone point to an initial object of
$\mathcal{D}$. Let $\mathcal{C} ' \subset \mathcal{C}$ be a full subcategory.
Then the following are equivalent: 
\begin{enumerate}
\item $F|_{\mathcal{C}}$ is a left  Kan extension of its restriction to
$\mathcal{C}'$. 
\item $F$ is a left Kan extension of its restriction to $\mathcal{C'}^{\lhd}$.
\end{enumerate}
\end{lemma} 
\begin{proof} 
Suppose the first condition is satisfied. Then if $c \in \mathcal{C}$  is
arbitrary, the natural map
\[ \varinjlim_{c' \to c \in \mathcal{C}'_{/c}} F(c') \to F(c)   \]
is an equivalence. 
Now, we have an equivalence of $\infty$-categories $(\mathcal{C}'_{/c})^{\lhd}
\simeq (\mathcal{C}'^{\lhd})_{/c}$, because $\lhd$ adds a new initial object. 
Therefore, for arbitrary $c \in \mathcal{C}$, we also get that the natural map
\[ \varinjlim_{c' \to c \in (\mathcal{C}'^{\lhd})_{/c}} F(c')
\simeq \varinjlim_{c' \to c \in (\mathcal{C}'_{/c})^{\lhd}} F(c') \to F(c)   \]
is an equivalence, thanks to \Cref{colimcone}. At the cone point, the left Kan
extension condition is automatic. Thus, it follows that $F$ is a left Kan extension of
$F|_{\mathcal{C}'^{\lhd}}$. The converse is proved in the same way. 
\end{proof}

\begin{proposition} 
\label{leftKancontractible}
Let $\mathcal{C}$ be a stable $\infty$-category and let $F\colon \Delta^{\leq n} \to
\mathcal{C}$ be any functor. Suppose $F$ is a left Kan extension of its
restriction to $\Delta^{\leq n-1}$. Then $\varprojlim_{\Delta^{\leq n}}
F $ is contractible. 
\end{proposition} 
\begin{proof} 
Observe that the cone $(\Delta^{\leq n})^{\lhd}$ is given by the category
$\Delta^{\leq n}_+$ of the finite totally ordered sets $\left\{[i]\right\}_{-1
\leq i \leq n}$ since $[-1]$ is an initial object of this category. 
Consider the functor $\widetilde{F}\colon \Delta_+^{\leq n} \simeq (\Delta^{\leq n})^{\lhd} \to \mathcal{C}$
extending $F$ that sends the cone point to the initial object (one can always
make such an extension). 
In order to show that $\varprojlim_{\Delta^{\leq n}} F$ is contractible, it
suffices to show that $\widetilde{F}$ is a right Kan extension of $F =
\widetilde{F}|_{\Delta^{\leq n}}$. 

Now, we recall a basic result of Lurie \cite[Lemma 1.2.4.19]{higheralg} (which we use
for the opposite category), a piece of the
$\infty$-categorical version of the Dold-Kan correspondence: given any functor
$G\colon \Delta^{\leq n}_+ \to \mathcal{C}$, $G$ is a right Kan extension of
$G|_{\Delta^{\leq n}}$ if and only if $G$ is a \emph{left} Kan extension of
$G|_{\Delta^{\leq n-1}_+}$. 
In our case, it follows that to show that $\widetilde{F}$ is a right Kan
extension of $F$ (as we would like to see), it suffices to show that
$\widetilde{F}$ is a \emph{left} Kan extension of $\widetilde{F}|_{\Delta_+^{\leq
n-1}}$. But by \Cref{kancone}, this follows from the fact that
$\widetilde{F}|_{\Delta^{\leq n}} = F$ is a left Kan extension of
$\widetilde{F}|_{\Delta^{\leq n-1}} = F|_{\Delta^{\leq n-1}}$. 
\end{proof} 

\subsection{The BKSS for $\fp$}

The goal of this subsection is to complete the proof of \Cref{FBKSS}. 
To begin with, we analyze the BKSS for the functor $\Sigma^\infty_+
\lan_{\Delta^{\leq t} \to\Delta}(\ast)\colon \Delta \to \sp$.

\begin{proposition} 
\label{BKSSLan}
The BKSS for the cosimplicial spectrum $\Sigma^\infty_+
\lan_{\Delta^{\leq t} \to\Delta}(\ast)$ 
satisfies 
\begin{equation} 
E_2^{s,*} = H^s( \pi_* (\Sigma^\infty_+
\lan_{\Delta^{\leq t} \to\Delta}(\ast))) = \begin{cases} 
\pi_*(S^0) & s = 0 \\
\pi_*(S^t) & s = t + 1.
 \end{cases} 
\end{equation} 
The differential $d_{t+1}$ is an isomorphism. (The result for $t =2$ is
displayed in
\Cref{fig:BKSSLan}.)
\end{proposition}

\begin{figure}[h]
\includegraphics[scale=1]{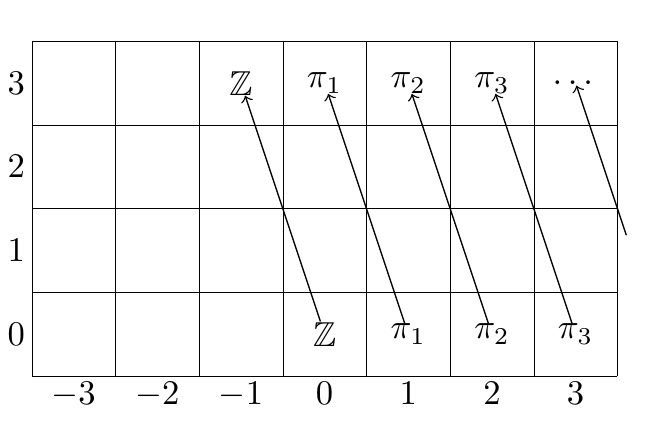}
\caption{Bousfield-Kan spectral sequence for $\Sigma^\infty_+
\lan_{\Delta^{\leq t} \to \Delta}(\ast)$, with
$t = 2$.}
\label{fig:BKSSLan}
\end{figure}

\begin{proof}
Observe that $\lan_{\Delta^{\leq t} \to\Delta}(\ast)$ is,
pointwise, a wedge of $t$-spheres, so to compute the desired cohomology $H^s( \pi_* ( \Sigma^\infty_+
\lan_{\Delta^{\leq t} \to\Delta}(\ast)))$, it suffices to 
do this for $\pi_t$. (The disjoint basepoint contributes the $\pi_*(S^0)$ for
$s =0 $ in cohomology.)
In other words, we may consider the cosimplicial $H \mathbb{Z}$-module 
$M^\bullet = H \mathbb{Z} \wedge \Sigma^\infty_+\lan_{\Delta^{\leq t} \to\Delta}(\ast)$. 
Now we know that, for each $n$, $\pi_*(M^n)$ is concentrated in degrees $0$ and
$t$, and that $\pi_0(M^\bullet)$ is the constant cosimplicial abelian group
$\mathbb{Z}$. Moreover, by \Cref{leftKancontractible}, $\mathrm{Tot}( M^\bullet)$ is contractible.
A look at the spectral sequence for $\mathrm{Tot}( M^\bullet)$ shows that
$H^s( \pi_t M^\bullet)$ must be concentrated in degree $s = t+1$ and must
be a $\mathbb{Z}$ there. The claim about differentials also follows from
contractibility of the totalization. 
\end{proof}

\begin{proof}[Proof of \Cref{FBKSS}]
The definition \eqref{defF} of $\fp^\bullet$ and \Cref{BKSSLan} together give the $E_2$-page of the spectral
sequence, when one uses the long exact sequence in homotopy groups. 
The differentials are forced, again, by \Cref{leftKancontractible} which
implies that $\mathrm{Tot}( \mathscr{G}^\bullet)$
is contractible. 
\end{proof}

\subsection{Completion of the proof}

\newcommand{\Y}{\mathscr{Y}}

Now we need to consider 
the cosimplicial $\e{\infty}$-ring
defined earlier
$$\Y^\bullet \stackrel{\mathrm{def}}{=} \tau_{\leq 2t} \X^\bullet \simeq
\tau_{\leq 2t}\mathrm{Free}_{\CAlg}(\mathscr{G}^\bullet).$$ 
We recall that this is well-defined as a cosimplicial $\e{\infty}$-ring because
$\mathscr{G}^\bullet$ is (pointwise) connective.

In this subsection, we will determine the relevant piece of the BKSS for $\Y$
and then complete the proof of \Cref{difftheorem}. 
We have that
\[ \Y^\bullet \simeq \tau_{\leq 2t}S^0 \vee \tau_{\leq 2t}
\mathscr{G}^\bullet
\vee \tau_{\leq 2t}\left((\mathscr{G}^\bullet)^{\wedge 2}_{h \Sigma_2}\right),
\]
because, by a connectivity argument, no other terms contribute. 
In particular, the cohomology $H^s( \pi_*(\Y^\bullet))$ picks up a copy of
$\pi_*(S^0)$ for $s = 0$ (which is mostly irrelevant).
In \Cref{FBKSS}, we determined the BKSS for $\mathscr{G}^\bullet$; in bidegrees $(t+1, t)$
and $(2t+2, 2t)$, this
picks up copies of $\mathbb{Z}$ such that the first one hits the second one
with a $d_{t+1}$. 
We will prove: 
\newcommand{\ssym}{\widetilde{\sym}}

\newcommand{\sym}{\mathrm{Sym}}

\begin{proposition} 
$E_{2}^{2t+2, 2t} \simeq \mathbb{Z} \oplus \mathbb{Z}/2$ in the BKSS for
$\Y^\bullet$. 
The $\mathbb{Z}/2$ is generated by the square of the class in bidegree $(t+1, t)$.
\end{proposition} 
\begin{proof} 
We will use the notation and results of \Cref{app:symm}. Let $A^\bullet$ be the
cosimplicial abelian group $\pi_t \mathscr{G}^\bullet$, which is levelwise
free and finitely generated. As we have seen
(\Cref{FBKSS}), 
$H^{t+1}(A^\bullet) \simeq \mathbb{Z}$ and the other cohomology of $A^\bullet$
vanishes. Now, using the notation of \Cref{symmdef},
$$ \pi_{2t}(\mathscr{G}^{ \bullet \wedge 2}_{h \Sigma_2}) = 
\begin{cases} 
\sym_2 A^\bullet & t \text{ even} \\
\ssym_2 A^\bullet & t \text{ odd}.
 \end{cases} $$
By \Cref{mainsymmcalculation}, we find that the $E_{2}^{2t + 2, 2t}$ term of
$(\mathscr{G}^\bullet)^{\wedge 2}_{h \Sigma_2}$ is as claimed. 
\end{proof}

We are now ready to complete the proof and determine the differential in the
$\gl_1$ spectral sequence. Using the notation of the beginning of
this section, it follows that 
$E_{t+1}^{ t+1, t} \simeq \mathbb{Z}$ and $E_{t+1}^{2t+1, 2t} \simeq
\mathbb{Z}\oplus \mathbb{Z}/2$, and similarly for $\overline{E}$.
The $d_{t+1}$ carries the $\mathbb{Z}$ into the other $\mathbb{Z}$.
By naturality of the spectral sequence, it follows that there must exist a universal formula
\begin{equation} \label{univform} \overline{d}_{t+1}( \overline{x}) = \overline{ a d_{t+1}(x) + \epsilon x^2},
\quad a \in \mathbb{Z}, 
 \ \epsilon \in \{0, 1\}.  \end{equation}
The main claim is that $a= \epsilon = 1$. Our first goal is to compute $a$.

\begin{lemma} \label{funL} We have an equivalence of $\infty$-categories
between the $\infty$-category $\mathrm{Fun}^L( \sp_{\geq
0}, \sp_{\geq 0}) $ of cocontinuous functors $\sp_{\geq 0} \to
\sp_{\geq 0}$ and $\sp_{\geq 0}$ given by evaluating at the sphere. 
The inverse equivalence sends a connective spectrum $Y$ to the functor $X
\mapsto X \otimes Y$. 
\end{lemma}
\begin{proof} 
It suffices to show that evaluation at the sphere induces an equivalence of
$\infty$-categories $\mathrm{Fun}^L(\sp_{\geq 0}, \sp) \simeq \sp$ (with
inverse given as above). 
But the $\infty$-category $\sp$ is the 
\emph{stabilization} \cite[\S 1.4]{higheralg} of $\sp_{\geq 0}$ (as one sees
 easily from the fact that $\Sigma$ is \emph{fully faithful} on $\sp_{\geq 0}$
 and an equivalence on $\sp$), so that we
have an equivalence
(by \cite[Corollary 1.4.4.5]{higheralg}) $\mathrm{Fun}^L( \sp, \sp) \simeq
\mathrm{Fun}^L( \sp_{\geq 0}, \sp)$ given by restriction. But we know that
$\mathrm{Fun}^L( \sp, \sp) \simeq \sp$ by evaluation at the sphere spectrum,
with inverse given by the smash product (see \cite[\S 4.8.2]{higheralg}). 
\end{proof} 

We need the following fact about $\gl_1$. 
\begin{proposition} \label{gl1sqzero}
Let $X$ be a connective spectrum, and let $S^0 \vee X$ be the square-zero
$\e{\infty}$-ring. Then there is a natural equivalence
of spectra,
\[ \gl_1( S^0 \vee X ) \simeq \gl_1(S^0) \vee X . \]
\end{proposition} 

On homotopy groups, this equivalence is compatible with the purely
algebraic equivalence $\pi_t
\gl_1( S^0 \vee X) \simeq \pi_t (S^0 \vee X) \simeq \pi_t(S^0) \oplus \pi_t(X)
\simeq \pi_t( \gl_1(S^0)) \oplus \pi_t(X)$.

\begin{proof} 
Given the connective spectrum $X$, we can use the 
composite $S^0 \to S^0 \vee X \to S^0$, in which the second map sends $X $ to $ 0 $, 
to get a natural splitting 
\[ \gl_1(S^0 \vee X) \simeq \gl_1(S^0) \vee F(X),  \]
where $F\colon \sp_{\geq 0} \to \sp_{\geq 0}$ is a certain functor that we want to
claim is naturally isomorphic to the identity. 
First, observe that $F$ commutes with colimits. 
Namely, $F$ commutes with filtered colimits (as one can check on homotopy
groups), $F$ takes $*$ to $*$, and given a pushout square
\begin{equation} \label{pushout1}  \xymatrix{
X_1 \ar[d] \ar[r] & X_2 \ar[d] \\
X_3 \ar[r] &  X_4,
}\end{equation}
in $\sp_{\geq 0}$, the analogous diagram
\begin{equation} \label{pushout2} \xymatrix{ F(X_1) \ar[d] \ar[r] & F(X_2) \ar[d] \\
F(X_3) \ar[r] &  F(X_4)
}\end{equation}
is a pushout square in $\sp_{\geq 0}$. This in turn follows by considering long
exact sequences in homotopy groups. 
More precisely, 
given the pushout square \eqref{pushout1}, the diagram of $\e{\infty}$-rings
\[ \xymatrix{
S^0 \vee X_1 \ar[d]\ar[r] &  S^0 \vee X_2 \ar[d]  \\
S^0 \vee X_3 \ar[r] &  S^0 \vee X_3,
}\]
is a homotopy \emph{pullback} in $\e{\infty}$-rings, so that 
applying $\gl_1$ (which is a \emph{right adjoint}) leads to a pullback square
\[ \xymatrix{
\gl_1(S^0 \vee X_1) \ar[d] \ar[r] &  \gl_1(S^0 \vee X_2) \ar[d]  \\
\gl_1(S^0 \vee X_3) \ar[r] &  \gl_1(S^0 \vee X_4),
}\]
and in particular, \eqref{pushout2} is homotopy cartesian too in $\sp_{\geq 0}$. Therefore, it is
homotopy cocartesian as well if we can show that the map
\[ \pi_0( \gl_1( S^0 \vee X_3)) \oplus \pi_0( \gl_1( S^0 \vee X_2))  \to
\pi_0( \gl_1( S^0 \vee X_4))  \]
is surjective. This, however, follows from the analogous fact that $\pi_0(X_3)
\oplus \pi_0(X_2) \to \pi_0(X_4)$ is surjective as \eqref{pushout1} is a pushout.

Therefore, as both $F$ commutes with colimits, $F$ is necessarily of the form
$X \mapsto X \otimes Y$ for some $Y \in \sp_{\geq 0}$, by \Cref{funL}. For $X = H \mathbb{Z}$,
we find $F(X) = H \mathbb{Z}$, so that $H \mathbb{Z} \otimes Y$ is concentrated
in degree zero and is isomorphic to $H \mathbb{Z}$. This forces $Y \simeq S^0$
and proves the claim. 
\end{proof}

\begin{proof}[Proof of \Cref{difftheorem}]
\Cref{gl1sqzero} implies that in the universal formula \eqref{univform}, the
constant $a = 1$. In fact, we know that if $X^\bullet$ is any cosimplicial
spectrum, then the cosimplicial spectra $\gl_1( S^0 \vee X^\bullet)$ and
$\gl_1(S^0) \vee X^\bullet$ are identified in a manner compatible with the
identifications of homotopy groups. In particular, the differentials in
the spectral sequence for $\gl_1(S^0 \vee X^\bullet)$ and in the spectral
sequence for $S^0 \vee X^\bullet$ are identified, forcing $a = 1$.

It remains to show that $\epsilon = 1$. 
For this, we need an example where the
two differentials do \emph{not} agree. 
This will be a generalization of \Cref{tylerexample}. 
Consider 
the $\e{\infty}$-ring $R$ of \Cref{gl1cubezero}, with $n = t$,
so that, in particular, $\gl_1( R) $ has homotopy groups in dimensions $t$
and $2t$ only.  \Cref{gl1cubezero} shows that the
$k$-invariant is \emph{nontrivial}.

Consider the space $X = K( \mathbb{F}_2, t+1)$, and consider the
Atiyah-Hirzebruch spectral sequences for the homotopy groups of $\gl_1(R)^X$
and $R^X$ (these can be identified with BKSS's by choosing simplicial
resolutions of $X$ by points). 
The latter clearly degenerates because $R$ is an Eilenberg-MacLane spectrum,
but we claim that the former does not. 

More precisely, we claim that there is no map 
of spectra
\[ \Sigma^{-1} \Sigma^{\infty } K(\mathbb{F}_2, t+1) \to \gl_1(R),   \]
inducing an isomorphism on $\pi_{t}$. 
The degeneration of the AHSS would certainly imply the existence of such a map. 
To see this, it is equivalent to showing that there is no map of (pointed) spaces
\[ K( \mathbb{F}_2, t+1) \to  BGL_1(R),  \]
with the same properties. If there existed such a map, then we could combine it
with the map $\tau_{\geq 2t+1} BGL_1(R) \simeq K(\mathbb{F}_2, 2t+1) \to
BGL_1(R)$ via the infinite loop structure to obtain a map
\[ K(\mathbb{F}_2, t+1) \times K( \mathbb{F}_2, 2t+1) \to BGL_1(R),  \]
which would be an equivalence by inspection of homotopy groups. 
However, this contradicts \Cref{gl1cubezero}, which shows that the
\emph{space} $BGL_1(R)$  
has a nontrivial $k$-invariant.

This completes the proof of \Cref{difftheorem}.
\end{proof}

\part{Computations}

\section{Picard groups of real $K$-theory and its variants}

Before we embark on the lengthy computations for the Picard groups of the
various versions of topological modular forms, let us work out in detail the case
of real $K$-theory, as well as the Tate $K$-theory spectrum $KO((q))$. In particular, these examples will illustrate our methodology without being computationally cumbersome.

\subsection{Real $K$ theory} \label{sec:PicKO}
In this subsection, we compute the Picard group of $KO$ using $C_2$-Galois
descent from the $C_2$-Galois extension $KO \to KU$ and the
Comparison Tool~\ref{toolGalois} (but not the universal formula of \Cref{difftheorem}).

We begin with the basic case of \emph{complex} $K$-theory.
\begin{example}[Complex $K$-theory]
\label{KUpic}
The complex $K$ theory spectrum has a very simple ring of homotopy groups $KU_*=\Z[u^{\pm 1}]$ with $u$ in degree $2$. In particular, $KU$ is even periodic with a regular noetherian $\pi_0$, so its Picard group is algebraic by \Cref{evenperiodicreg}. The inner workings of \Cref{evenperiodicreg} would use that the only (homogeneous) maximal ideals of $KU_*$ are generated by
prime numbers $p$; for each $p$, there is a corresponding residue field
spectrum, namely mod-$p$ $K$-theory, also known as an extension of the Morava
$K$-theory of height one at the given prime. As the Picard group of $ KU_0=\Z$ is trivial, and $\pic(KU_*) \simeq \Z/2 $,
any invertible $KU$-module is equivalent to either $KU$ or $\Sigma KU$. 
\end{example}

To compute $\pic(KO)$, we start with this knowledge that, thanks to \Cref{KUpic}, $\pi_0 \pics (KU) = \pic (KU)$ is $\Z/2$. We have the spectral sequence from
\eqref{ss:galDescent}
\[ H^*(C_2, \pi_*\pics (KU) ) \Rightarrow \pi_*( \pics (KU))^{hC_2} 
 \]
which will allow us to compute $\pi_0( \pics (KU))^{hC_2} \simeq \pic KO$. We
note that $\pi_1 \pics (KU) \simeq (KU_0)^\times = \Z/2$, and
\[ H^*(C_2,\Z/2) = \Z/2[x],\]
where $x$ is in cohomological degree 1.
The higher homotopy groups of $\pics (KU)$ coincide (as $C_2$-modules) with those of $KU$, suitably shifted by one.

Recall, moreover, that the $E_2$-page of the HFPSS for $\pi_* KO$
is given by 
the bigraded ring
\[ E_2^{*, *} = \mathbb{Z}[u^2, u^{-2}, h_1]/(2 h_1), \quad |u^2|= (4, 0),
\ | h_1 | = (1, 2),  \]
where $u^2$ is the square of the Bott class in $\pi_* KU \simeq
\mathbb{Z}[u^{\pm 1}]$, and $h_1$ detects in homotopy the Hopf map $\eta$. 
The class $h_1$ is in bidegree $(s,t) = (1, 2)$, so it is drawn using Adams indexing in
the $(1,1)$ place. 
The differentials are determined by $d_3(u^2) = h_1^3$ and the spectral
sequence collapses at $E_4$. 
For convenience, we reproduce a picture in \Cref{fig:KOdesc}; the interested reader can find the detailed computation of this spectral sequence in \cite[\S 5]{HeardStojanoska}.

\begin{figure}[h]
\includegraphics[scale=1]{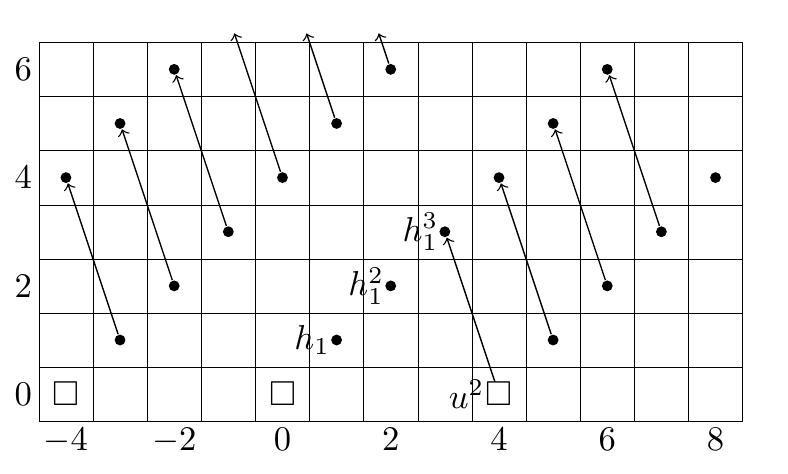}
\caption{Homotopy fixed point spectral sequence for
$\pi_* KO \simeq \pi_*(KU^{hC_2})$ }
\centering{($\bullet$ denotes $\Z/2$ and $\square $ denotes $\Z$)} \label{fig:KOdesc}
\end{figure}

Therefore, the $E_2$-page of the spectral sequence for $(\pics(KU))^{hC_2}$ is as
in~\Cref{fig:picKO}. To deduce differentials, we use our
Comparison Tool~\ref{toolGalois}: in the homotopy fixed point spectral sequence
for $KU$, there are only (non-trivial) $d_3$-differentials. By the
Comparison Tool~\ref{toolGalois}, we conclude that we can ``import" those differentials to the HFPSS for $\pics (KU)$ when they involve terms with $t\geq 4$. In particular, we see that the differentials drawn in \Cref{fig:picKO} are non-zero; moreover, everything that is above the drawn range and in the $s=t$ column either supports or is the target of a non-zero differential. Note that we are not claiming that there are no other non-zero differentials, but these suffice for our purposes.

\begin{figure}[h]
\includegraphics[scale=1]{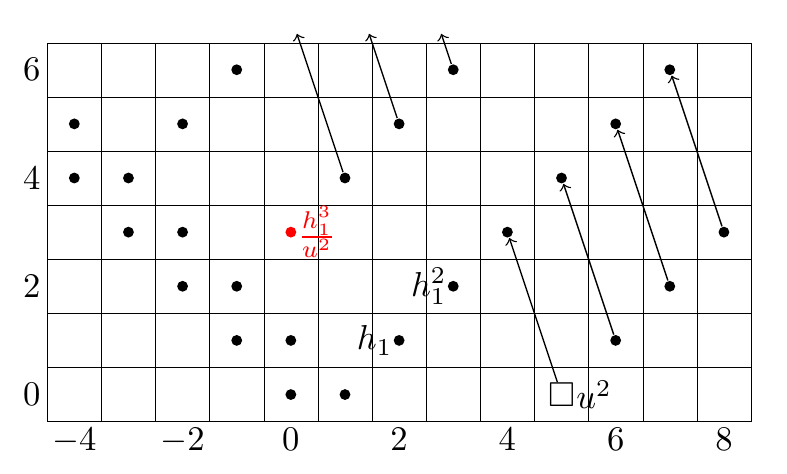}
\caption{Homotopy fixed point spectral sequence for $\pics (KU)^{hC_2}$}\label{fig:picKO}
\end{figure}

We deduce from this that $\pi_{0} \pics (KU)^{hC_2}=\pic (KO)$ has cardinality at most eight.
On the other hand, the fact that $KO$ is $8$-periodic gives us a lower bound
$\Z/8$ on $\pic (KO)$. Thus we get:

\begin{theorem}[Hopkins; Gepner-Lawson \cite{GL}]
\label{picko}
$\pic (KO)$ is precisely $\Z/8$, generated by $\Sigma KO$.
\end{theorem}

\Cref{picko} was proved originally by Hopkins (unpublished) using related
techniques. The approach via descent theory is due to Gepner-Lawson in
\cite{GL}. Their identification of the differentials in the spectral sequence
is, however, different from ours: they use an explicit knowledge of the
structure of $\gl_1(KU)$ with its $C_2$-action (which one does not
have for $\TMF$).

\begin{remark}
In view of \Cref{rem:relativePic}, we conclude that the relative Picard group of the $C_2$-extension $KO\to KU$ is $\pi_{-1}(\gl_1 KU)^{hC_2} \simeq \Z/4$.
\end{remark}

\begin{remark} 
In the usual descent spectral sequence for $KO$, the class $\frac{h_1^3}{u^2}$
in red supports a $d_3$. By \Cref{difftheorem} and the multiplicative
structure of the usual SS, $\frac{h_1^3}{u^2}$ does \emph{not} support a $d_3$ in
the descent SS for $\pic$. 
We saw that above by counting: if $\frac{h_1^3}{u^2}$ did not survive, the
Picard group of $KO$ would be too small. For $2$-local $\TMF$, simple counting
arguments will not suffice and we will actually need to use \Cref{difftheorem}
as well. 
\end{remark} 

\begin{remark}
We can also deduce from the spectral sequence that the cardinality of the relative Brauer group for $KO/KU$, which is isomorphic to $\pi_{-1}(\pics (KU))^{hC_2}$, is most eight. However, we do not know how to construct necessarily non-trivial elements of this Brauer group in order to deduce a lower bound as in the Picard group case.
\end{remark}

\subsection{$KO[q], KO[[q]]$ and $KO((q))$}

We now include a variant of the above example where one adds a polynomial
 (resp. power series, Laurent series) generator, where we will also be able
 to confirm the answer using a different argument. 
This example can be useful for comparison with $\TMF$ using topological
$q$-expansion maps. 
We begin by introducing the relevant $\e{\infty}$-rings.
This subsection will not be used in the sequel and may be safely skipped by the
reader. 

\begin{definition} 
We write for $S^0[x]$ the suspension spectrum $\Sigma^\infty_+ \mathbb{Z}_{\geq
0}$. Since $\mathbb{Z}_{\geq 0}$ is an $\e{\infty}$-monoid in spaces (in fact,
a commutative topological monoid), $S^0[x]$ naturally acquires the structure of an
$\e{\infty}$-ring. 
Given an $\e{\infty}$-ring $R$, we will write $R[x] = R \wedge S^0[x]$. 

We can also derive several other variants: 
\begin{enumerate}
\item We will let $R[[x]]$ denote the $x$-adic completion of $R[x]$, so its
homotopy groups look like a power series ring over $\pi_*R$.  
\item We will let $R[x^{\pm 1}]$ denote the localization $R[x][1/x]$, so its
homotopy groups are given by Laurent polynomials in $\pi_* R$.
\item We will let $R((x)) = R[[x]][1/x]$, so that its homotopy groups look like
formal Laurent series over $\pi_* R$. 
\end{enumerate}
\end{definition} 

On the one hand, $\pi_* (R[x]) \simeq (\pi_*R)[x]$ is a polynomial ring over
$\pi_* R $ on a generator in degree zero. 
On the other hand, as an $\e{\infty}$-algebra under $R$, the universal property of $R[x]$ is
significantly more complicated than that of the ``free''
$\e{\infty}$-$R$-algebra on a generator (which is often denoted
$R\left\{x\right\}$). 
A map $R[x] \to R'$, for an $\e{\infty}$-$R$-algebra $R'$, is equivalent to an
$\e{\infty}$-map
\[ \mathbb{Z}_{\geq 0} \to \Omega^\infty R',\]
where $\Omega^\infty R'$ is regarded as an $\e{\infty}$-space under
\emph{multiplication.}
In general, given a class in $\pi_0 R'$, there is no reason to expect an
$\e{\infty}$-map $R[x] \to R'$ carrying $x$ to it, since $\mathbb{Z}_{\geq
0}$ as an $\e{\infty}$-monoid is quite complicated. 
Classes for which this is possible (together with the associated maps $R[x] \to R'$) have
been called ``strictly commutative'' by Lurie. 

\begin{example} 
There is a map $R[x] \to R$ sending $x \to 1$. This comes from the map of
$\e{\infty}$-spaces $\mathbb{Z}_{\geq 0} \to \ast \to \Omega^\infty
S^0$ where $\ast$ maps to the unit in $\Omega^\infty S^0$. 
\end{example} 
\begin{example} \label{0strictlycomm}
There is a map $R[x] \to R$ sending $x \to 0$.\footnote{We are grateful to the
referee for suggesting this argument over our previous one.} 

To obtain this in the universal case $R = S^0$, we consider the adjunction
\[ (\Sigma^\infty, \Omega^\infty): \mathcal{S}_*  \rightleftarrows \sp.  \]
Here $\mathcal{S}_*$ and $\sp$ are symmetric monoidal 
with the smash product and $\Sigma^\infty$ is a symmetric monoidal functor. 
In particular, $\Sigma^\infty$ carries commutative algebra objects in
$\mathcal{S}_*$ to $\e{\infty}$-ring spectra. 

We start with the commutative monoid $M $ 
with a single element $m$.
Then $M_+  = \left\{\ast, m\right\}\in \mathcal{S}_*$ is a commutative algebra object of $\mathcal{S}_*$
with respect to the smash product: in fact, it is the unit $S^0$ as a
pointed space. 
Similarly, $(\mathbb{Z}_{\geq 0})_+$ is a commutative algebra object of
$\mathcal{S}_*$. 
Now we have equivalences of $\e{\infty}$-ring spectra $\Sigma^\infty (M_+) \simeq S^0$
and $\Sigma^\infty( \mathbb{Z}_{\geq 0})_+ \simeq \Sigma^\infty_+
\mathbb{Z}_{\geq 0}$. 
There is a map of commutative monoids in $\mathcal{S}_*$
\[ (\mathbb{Z}_{\geq 0})_+ \to M_+,  \]
which carries $0 \in \mathbb{Z}_{\geq 0}$ to $m$  and everything else to $\ast$. 
After applying $\Sigma^\infty$, we obtain the desired map $S^0[x] \to S^0$ of
$\e{\infty}$-rings.

\end{example}

The map $R[x] \to R$ given in \Cref{0strictlycomm} has the property that it
exhibits the $R[x]$-module $R$ as the cofiber  $R[x]/x$. 
It follows in particular that if $R'$ is any $\e{\infty}$-$R$-algebra and $x'
\in \pi_0 R' $ is a strictly commutative element, then we can give the cofiber
$R'/x' \simeq R' \otimes_{R[x]} R$ the structure of an
$\e{\infty}$-$R'$-algebra. 

\begin{remark} Consider the sphere spectrum $S^0$. Then no cofiber $S^0/n$ for
$n \notin \left\{\pm 1, 0\right\}$ can admit the structure of an
$\e{\infty}$-ring by, for example, \cite[Remark 4.3]{nilpotence}.\footnote{It
is an unpublished result of Hopkins that no Moore spectrum can even admit the
structure of an $\e{1}$-algebra.} It follows that the only
element of $\pi_0 S^0 \simeq \mathbb{Z}$, besides $0$ and $1$, that can potentially be strictly commutative
is $-1$. Now, $-1$ is not strictly commutative in the $K(1)$-local
sphere $L_{K(1)} S^0$ at the prime $2$ because of the operator $\theta$ of
\cite{hopkinsk1local}: we
have $\theta(-1) = \frac{(-1)^2 - (-1)}{2} = 1 \neq 0$, while power
operations such as $\theta$ annihilate strictly commutative elements. 
Therefore, $-1$ cannot be strictly commutative in $S^0$. (One could have
applied a similar argument with power operations to every other integer, too.) However, we observe that it is strictly commutative in $S^0[1/2]$: the
obstruction is entirely 2-primary (\Cref{nthrootsstrictlycomm} below). 
\end{remark} 

\begin{example} 
Let $a, b \in \pi_0 R$ be strictly commutative elements for $R$ an
$\e{\infty}$-ring. Then $ab$ is also strictly commutative. If $a$ is  a unit, then
$a^{-1}$ is strictly commutative. This follows because there is a natural
addition on $\e{\infty}$-maps $\mathbb{Z}_{\geq 0}  \to \Omega^\infty R$. 
\end{example}

\begin{proposition} 
\label{nthrootsstrictlycomm}
Let $R$ be an $\e{\infty}$-ring with $n$ invertible. Then any $u \in \pi_0 R$
with $u^n = 1$ (i.e., an $n$th root of unity) admits the structure of a
strictly commutative element.
\end{proposition}
\begin{proof} 
We consider the map of $\e{\infty}$-monoids
$\mathbb{Z}_{\geq 0} \to \mathbb{Z}/n\mathbb{Z}$ and the induced map
of $\e{\infty}$-ring spectra
\begin{equation} \label{modn} R[x] \to R \wedge \Sigma^\infty_+
\mathbb{Z}/n\mathbb{Z}.  \end{equation}
Since $\frac{1}{n} \in \pi_0 R$, 
$R \wedge \Sigma^\infty_+ \mathbb{Z}/n\mathbb{Z}$ is \'etale over $R$ and the
homotopy groups are given by $\pi_* R [x]/(x^n - 1)$. We can thus produce a map
of $\e{\infty}$-rings
$R \wedge \Sigma^\infty_+ (\mathbb{Z}/n\mathbb{Z}) \to R$ sending $1 \in
\mathbb{Z}/n\mathbb{Z} $ to $u$ by \'etaleness.\footnote{The \'etale
obstruction theory has been developed by a number of authors; a convenient
reference for the result that we need is \cite[Thm. 8.5.4.2]{higheralg}.} Composing with 
\eqref{modn} gives us the strictly commutative structure on $u$. 
\end{proof} 

Using these ideas, we will be able to give a direct computation of the Picard
group of the $\e{\infty}$-ring $KO[[q]]$. (We have renamed the power series
variable to ``$q$'' in accordance with ``$q$-expansions.'')
\begin{proposition}
\label{KO[[q]]}
The map $\pic (KO)\to \pic (KO[[q]] )$ is an isomorphism, where $q$ is in degree zero.
\end{proposition}
\begin{proof}
Suppose $M$ is an invertible $KO[[q]]$-module such that $M/qM \simeq M
\otimes_{KO[[q]]} KO$ is equivalent to $KO$. We will show that then $M$ is equivalent to $KO[[q]]$ using Bocksteins. 
Specifically, consider the generating class in $\pi_0(M/qM) \simeq \mathbb{Z}$;
we will lift this to a class in $\pi_0 M $. 
It will follow that the induced map $KO[[q]] \to M$ becomes an equivalence after
tensoring with $KO \simeq KO[[q]]/q$. Since $M$ is $q$-adically complete, it
will follow that $KO[[q]] \simeq M$.

By induction on $k$, suppose that:
\begin{enumerate}
\item $\pi_{-1} (M/q^k M) = 0$. 
\item $\pi_0(M/q^k M) \to \pi_0(M/qM)$ is a surjection.
\end{enumerate}
These conditions are clearly satisfied for $k = 1$. 
If these conditions are satisfied for $k$, then the cofiber sequence
of $KO[[q]]$-modules
\[ M/q^k M \to M/q^{k+1} M \to M/qM  \]
shows that they are satisfied for $k+1$. 
In the limit, we find that there is a map $KO[[q]] \to M$ which lifts the
generator of $\pi_0(M/qM)$, which proves the claim. 
\end{proof}

\Cref{KO[[q]]} can also be proved using Galois descent, but
unlike for $KO$, we need to use \Cref{difftheorem}. 

\begin{proof}[Second proof of \Cref{KO[[q]]}]
The faithful $C_2$-Galois extension $KO \to KU$ induces upon base-change a
faithful $C_2$-Galois extension $KO[[q]] \to KU[[q]]$. 
The Picard group of $KU[[q]]$, again by \Cref{evenperiodicreg}, is
$\mathbb{Z}/2$ generated by the suspension. Consider now the descent spectral
sequence for $(\pics(KU[[q]]))^{hC_2}$, which is a modification of the descent spectral
sequence for $KU^{hC_2}$ in \Cref{fig:picKO}. One difference is
that every term with $t \geq 2$ is replaced by its tensor product over
$\mathbb{Z}$ with $\mathbb{Z}[[q]]$; the other is that the $t=1$ line now contains the 
$C_2$-cohomology of the units in $\pi_0 KU[[q]]$, which is a bigger module than $(\pi_0KU)^\times = \Z/2$. Namely, these units are $\Z/2 \oplus q\Z[[q]] $, with trivial $C_2$-action. The resulting $E_2$-page is displayed in \Cref{fig:picKO[[q]]}.

\begin{figure}[h]
\includegraphics[scale=1]{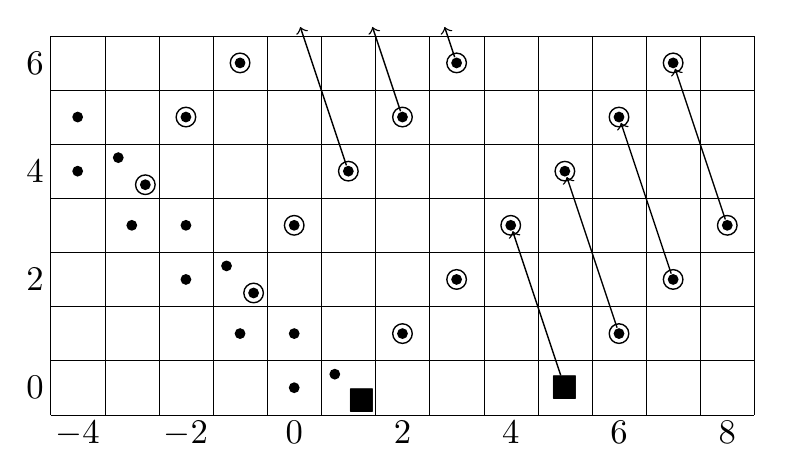}
\caption{Homotopy fixed point spectral sequence for $\pics (KU[[q]])^{hC_2}$}
\centering{$\bullet$ denotes $\Z/2$,  $\cbull$ denotes $\Z/2[[q]]$, and $\blacksquare$ denotes $\Z[[q]]$}\label{fig:picKO[[q]]}
\end{figure}

Since the $d_3$ is the only differential
in the ordinary HFPSS for $\pi_* KO[[q]]$, as before, it follows that the only
contributions to $\pic(KO[[q]])$ can come from the $\mathbb{Z}/2$ with $t=s=0$
(the suspension), the $\mathbb{Z}/2$ with $(s,t) = (1,1)$ (i.e. the algebraic
Picard group), and the $\Z/2[[q]]$ in bi-degree $(s,t) = (3, 3)$. 

But here, $E_2^{3,3}= \mathbb{Z}/2[[q]] \frac{h_1^3}{u^{2}}$ is infinite, so unlike previously, we do not
get the automatic upper bound of eight on $|\pic(KO[[q]])|$. On the other hand,
we can use \Cref{difftheorem} to determine the $d_3$ supported here.
Note that
in the HFPSS for $(KU[[q]])^{hC_2}$, we have
\[ d_3\left( f(q)  \frac{h_1^3}{u^{2}}\right) =  f(q) \frac{h_1^6}{u^4}, \quad f(q) \in
\mathbb{Z}/2[[q]].\]
Therefore, in view of \eqref{univformula}, in the HFPSS for
$\pics(KU[[q]])^{hC_2}$, we have
\[ d_3\left(  f(q)  \frac{h_1^3}{u^{2}}\right) 
= \left( f(q) + f(q)^2\right)\frac{h_1^6}{u^4}.
\]
(Note that a crucial point here is that in the HFPSS for $KO$, squaring or applying $d_3$ to $ \frac{h_1^3}{u^2}$ yields the same result.)
It follows from this that in the HFPSS, the kernel of $d_3$ on $E_2^{3,3}$ is $\mathbb{Z}/2$ generated by $1 \frac{h_1^3}{u^2}$: the equation $f(q) +
f(q)^2 = 0$ has only the solutions $f(q) \equiv 0, 1$. 
Therefore, we do get an upper bound of eight on the cardinality of $\pic(KO[[q]])$ after all,
as nothing else in $E_2^{3,3}$ lives to $E_4$. 
\end{proof} 

\begin{corollary} 
\label{KO((q))}
The maps $KO \to KO[q]$, $KO \to KO((q))$ induce isomorphisms on Picard groups.
\end{corollary} 
\begin{proof} 
This result is not a corollary of \Cref{KO[[q]]} but rather of its second
proof. In fact, the same argument shows that $d_3$ has a $\mathbb{Z}/2$ as
kernel on the relevant term $E_2^{3,3}$,  which gives an upper bound of
cardinality eight on the
Picard group of $KO[q]$ or $KO((q))$ as before.
\end{proof} 

\begin{remark} 
\Cref{KO((q))} cannot be proved using the Bockstein spectral sequence argument
used in the first proof of \Cref{KO[[q]]}. However, a knowledge of the Picard
group  of $KO[[q]]$ can be used to describe enough of the $C_2$-descent spectral
sequence to make it possible to prove \Cref{KO((q))} without the explicit
formula \eqref{univformula}. We leave this to the reader. 
\end{remark}


\newcommand{\GL}[1]{GL_2(\Z/#1)}

\section{Picard groups of topological modular forms}

In the rest of the paper we proceed to use descent to compute the Picard groups
of various versions of topological modular forms.  We will analyze the following
descent-theoretic  situations:
\begin{itemize}
\item The Galois extension $TMF[1/2]\to TMF(2)$, with structure group $\GL{2}$, also known as the symmetric group on three letters.
\item The Galois extension $TMF[1/3]\to TMF(3)$, with structure group $\GL{3}$; this is a group of order $48$, which is a nontrivial extension of the binary tetrahedral group and $C_2$.
\item \'Etale descent from the (derived) moduli stack of elliptic curves or its compactification.
\end{itemize}

In each of these cases, we will start with the knowledge of the original
descent spectral sequence, computing the homotopy groups of the global sections or homotopy fixed point spectrum. This information plus some additional computation of the differing cohomology groups will provide the data for the $E_2$-page of the descent spectral sequence for the Picard spectrum. The additional computations are somewhat lengthy, hence we are including them separately in the Appendix. 

\subsection{The Picard group of $TMF[1/2]$}\label{sec:TMF1/2}

When $2$ is inverted, the moduli stack of elliptic curves $\mell$ has a
$\GL{2}$-Galois cover by $\mell(2)$, the moduli stack of elliptic curves with
full level $2$ structure. This remains the case for the derived versions of
these stacks, and on global sections gives a faithful Galois extension
$TMF[1/2] \to TMF(2)$ by \cite[Theorem 7.6]{affine}. The extension is useful for the purposes of descent as the homotopy groups of $TMF(2)$ are cohomologically very simple.

To be precise, we have that
\[ TMF(2)_* = \Z[1/2][\lambda_1^{\pm 1},\lambda_2^{\pm 1}][(\lambda_1-\lambda_2)^{-1} ],\]
where the (topological) degree of each $\lambda_i$ is four. To see this, one can use the
presentation of the moduli stack $\mell(2)$ from \cite[\S 7]{TmfDualityp3}. 
There it is computed that $\mellc(2)$ is equivalent to (the stacky) $\Proj\,
\Z[1/2][\lambda_1,\lambda_2]$. Moreover, the substack classifying smooth curves,
i.e., $\mell(2)$, is the locus of non-vanishing of $ \lambda_1^2 \lambda_2^2 (\lambda_1-\lambda_2)^2$.
More precisely, $\mell(2)$, as a stack, is the $\mathbb{G}_m$-quotient of the
ring $\mathbb{Z}[1/2][\lambda_1, \lambda_2, ( \lambda_1^2 \lambda_2^2
(\lambda_1 - \lambda_2))^{-1}]$, where the $\mathbb{G}_m$-action is as follows:
a unit $u$ acts as $\lambda_i \mapsto u^2 \lambda_i$ for $i = 1, 2$, so that
it is an open substack of a \emph{weighted projective stack}. 

In particular, $TMF(2)_*$ has a unit in degree $4$, and is zero in degrees not divisible by $4$. It will be helpful to write $TMF(2)_*$ differently, so as to reflect this periodicity more explicitly; for example, we have that $TMF(2)_* = TMF(2)_0 [\lambda_2^{\pm 1}]$, and
\begin{align}\label{eq:pi_0TMF(2)}
TMF(2)_0 = \Z[1/2][s^{\pm 1}, (s-1)^{-1}], 
\end{align}
where $s=\frac{\lambda_1}{\lambda_2}$. Therefore, \Cref{even2} applies to give the following conclusion.
\begin{lemma}\label{picTMF(2)}
The Picard group of $TMF(2)$ is $\Z/4$, generated by the suspension $\Sigma TMF(2)$.
\end{lemma}

\begin{remark}
The proof of \Cref{even2} relies on the construction of ``residue field"
spectra; let us specify what they are in the case at hand. The maximal ideals
in $TMF(2)_0$ are $\mathfrak{m}=(p,f(s))$, where $p$ is an odd prime and $f(s)$
a monic polynomial irreducible modulo $p$ (and not congruent mod $p$ to $s,
s-1$).
For each of these ideals, we have an associative ring spectrum (the ``residue
field'') with homotopy groups $TMF(2)_*/\mathfrak{m}$ by \cite{angeltveit}; denote it temporarily by $TMF(2)/\mathfrak{m}$. After extending scalars so that $f$ splits, we get that $TMF(2)/\mathfrak{m}$ is a product of (extensions of) mod $p$ Morava $K$-theory spectra at height one or two, one for each zero of $f$. By \cite[V.4.1]{Silverman}, the factor associated to the zero $a$ of $f$ has height two precisely when 
\[ \sum_{i=0}^{(p-1)/2} \binom{(p-1)/2}{i} a^i \]
is zero modulo $p$. 
\end{remark}

Next we use descent from $TMF(2)$ to $TMF[1/2]$ to obtain the following result.

\begin{theorem}\label{thm:picTMF1/2}
The Picard group of $TMF[1/2]$ is $\Z/72$, generated by the suspension $\Sigma TMF[1/2]$. In particular, this Picard group is algebraic.
\end{theorem}

\begin{proof}
We use the homotopy fixed point spectral sequence \eqref{ss:galDescent}
\begin{align}\label{ss:TMF(2)}
H^s(\GL{2}, \pi_t\pics (TMF(2))) \Rightarrow \pi_{t-s}\pics (TMF(2))^{h\GL{2}}.
\end{align}
To begin with, note that the homotopy groups $\pi_t \pics (TMF(2))$ for $t\geq 2$ are isomorphic to $\pi_{t-1} TMF(2)$ as $\GL{2}$-modules. This tells us that the $t\geq 2$ part of the $E_2$-page of the HFPSS \eqref{ss:TMF(2)} for $\pics (TMF(2))$ is a shifted version of the corresponding part for $TMF(2)$.

The latter is immediately obtained from the analogous computation for $Tmf(2)$ in \cite{TmfDualityp3} (depicted in Figure 2 of loc. cit.), as we now describe. Recall that $TMF(2) \simeq Tmf(2)[\Delta^{-1}]$; the non-negative homotopy groups $\pi_{\geq 0} Tmf(2) $ are the graded polynomial ring $\Lambda=\Z[1/2][\lambda_1, \lambda_2]$ \cite[Proposition 8.1]{TmfDualityp3}, and the class $\Delta \in \pi_{24}Tmf(2)$ is 
\[ \Delta = 16 \lambda_1^2 \lambda_2^2 (\lambda_2 - \lambda_1)^2\] 
by \cite[Proposition 10.3]{TmfDualityp3}. Now, by \cite[Prop. 10.8]{TmfDualityp3} we have that 
\[ H^*\big(\GL{2}, \pi_* TMF(2) \big) = H^*\big(\GL{2}, \Lambda \big)[\Delta^{-1}]. \]
 In particular, the invariants  $H^0\big(\GL{2}, \Lambda \big)[\Delta^{-1}] $
 are the ring of $\Delta$-inverted modular forms \[\Z[1/2][c_4,c_6,\Delta^{\pm 1}]/ (12^3 \Delta -c_4^3 +c_6^2) .\] The higher cohomology $ H^{>0}\big(\GL{2},\Lambda \big)$ is computed in \cite[\S 10.1]{TmfDualityp3}, and in particular is killed by $c_4$ and $c_6$. Consequently, 
 \[H^{>0}\big(\GL{2}, \pi_{\geq 0} TMF(2)\big) = H^{>0}\big(\GL{2},\Lambda \big) = H^{>0}(\GL{2}, \pi_{\geq 0}Tmf(2)). \] 
 
Let us recall (the names of) certain interesting classes in these cohomology groups: 
\begin{enumerate}
\item 
There is $a $ in $H^1(\GL{2}, \pi_4 TMF(2))=\Z/3$, hence in $H^1\big(\GL{2},
\pi_5 \pics(TMF(2))\big)$ (so, $a $ is in bidegree $(s,t)=(1,5)$ in the Picard
HFPSS, and depicted in position $(s,t-s)=(1,4) $ using the Adams convention).
In homotopy, this element detects the Greek letter element $\alpha_1 $ in the Hurewicz image in $TMF[1/2]$.
\item There is $b $ in $H^2\big(\GL{2}, \pi_{13} \pics(TMF(2))\big) = \Z/3$ ($b$ is in bidegree $(2,13)$ or position $(2,11)$); in homotopy it detects $\beta_1$.
\end{enumerate}

Then, $H^{>0}\big(\GL{2}, TMF(2)_* \big)$ is precisely the ideal of $\Z/3 [a, b][\Delta^{\pm 1}]/(a^2)$ of positive cohomological degree. For example
\[H^5\big(\GL{2}, \pi_5 \pics (TMF(2))\big) =H^5\big(\GL{2}, \pi_4 TMF(2)\big) = \Z/3, \]
generated by $a b^2 \Delta^{-1}$. We see this class depicted red below in \Cref{fig:picTMF(2)}.

Next, we turn to the information which is new for the Picard HFPSS, i.e., the group cohomology of $\pi_0 $ and $\pi_1$ of the spectrum $\pics (TMF(2))$. By \Cref{picTMF(2)}, we know that the zeroth homotopy group is $\Z/4$, and since it is generated by the suspension $\Sigma TMF(2)$, the action of $\GL{2}$ on this $\Z/4$ is trivial. Even though for our purposes only the invariants $H^0\big(\GL{2}, \pi_0\pics (TMF(2)) \big)$ are necessary, we can in fact compute all the cohomology groups. This is done in \Cref{G2trivialZ4} of \Cref{sec:grpcohTMF(2)}.

The last piece of data needed for the determination of the $E_2$-page of the Picard HFPSS is the group cohomology with coefficients in $\pi_1 \pics(TMF(2)) = (\pi_0 TMF(2))^\times $. This is done in \Cref{G2pi0units}. The range $s\leq 15$ and $-6\leq t-s \leq 7$ of spectral sequence is depicted in \Cref{fig:picTMF(2)}. Note that in this range, the $t-s=0$ column has three non-zero entries: there is a $\Z/4$ for $s=0$, $\Z/6$ for $s=1$, and $\Z/3$ for $s=5$.

\begin{figure}[h!]
\includegraphics[scale=0.9]{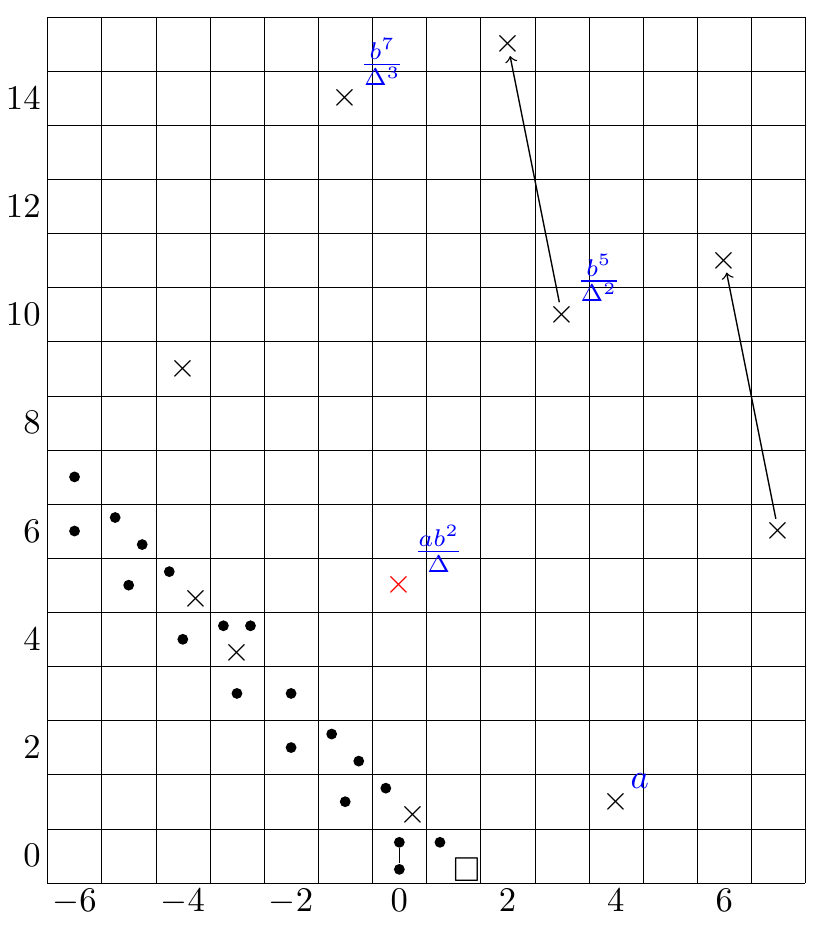}
\caption{Homotopy fixed point spectral sequence for $(\pics (TMF(2)))^{h\GL{2}}$
}
\vspace{-4mm}
\caption*{($\square$ denotes $\Z$, $\bullet$ denotes $\Z/2$, and $\times$ denotes $\Z/3$) }
\label{fig:picTMF(2)}
\end{figure}

Now we are ready to study the differentials in the HFPSS for $\pics
(TMF(2))^{h\GL{2}}$. Comparison with the HFPSS for the $\GL{2}$-action on
$TMF(2)$ gives a number of differentials, using our Comparison
Tool~\ref{toolGalois}. To distinguish between the differentials in the two spectral sequences, let us denote by $d_r^o$ those in the HFPSS of $TMF(2)$. The superscript $o$ stands for ``original."

Recall that in the HFPSS for $TMF(2)$, there are non-zero $d_5^o$ and $d_9^o$ differentials, which are obtained, for example, by a comparison with the HFPSS for $Tmf(2)$ which is fully determined in \cite{TmfDualityp3}. In particular, in the HFPSS for $TMF(2)$, the first differential is $d^o_5(\Delta)=a b^2$, and the rest of the $d^o_5$'s are determined by multiplicativity and the fact that $a$ and $b$ are permanent cycles. In particular, we have
\begin{align}\label{diffd5}
d^o_5\left(\frac{b^5}{\Delta^2}\right) =  \frac{a b^7}{\Delta^3}, \qquad \qquad 
d^o_5\left(\frac{ b^3}{\Delta}\right) = -  \frac{a b^5}{\Delta^2}.
\end{align}
Next (and last) is $d^o_9$; we have that $d_9^o(a \Delta^2)=b^5$. Consequently, we also have
\begin{align}\label{diffd9}
d^o_9\left(\frac{a b^2}{\Delta}\right) =  \frac{ b^7}{\Delta^3}.
\end{align}

Let us now see which of these differentials also occur in the HFPSS for $\pics
(TMF(2))$; according to Comparison
Tool~\ref{toolGalois}, the $d_5$-differentials are imported
in the range $t>5$, and the $d_9$-differentials in the $t>9$ range. In
particular, the differentials in \eqref{diffd5} are the same in the Picard
HFPSS; these are the two differentials drawn in \Cref{fig:picTMF(2)}. Moreover,
everything in the zero column and above the depicted region, i.e., such that $s=t>16$, either supports a differential or is killed by one which originates in the $t>9$ range. Hence, everything above the depicted region is killed in the spectral sequence and nothing survives to the $E_\infty $-page.

Note, however, that we cannot (and should not attempt to) import the differential \eqref{diffd9}; this would be a $d_9$-differential with $t=5$, so it does not satisfy the hypothesis of Comparison
Tool~\ref{toolGalois}.

Let us analyze the potentially remaining contributions to $\pi_0 \pics(TMF(2))^{\GL{2}}$; regardless of what the rest of the differentials could possibly be, we have:
\begin{itemize}
\item a group of order at most $4$ (and dividing $4$) in position $(0,0)$,
\item a group of order at most $6$ (and dividing $6$) in position $(0,1)$, and 
\item a group of order at most $3$ (and dividing $3$) in position $(0,5)$.
\end{itemize}
Therefore $\pic (TMF [1/2] )= \pi_0  \pics(TMF(2))^{\GL{2}} $ has order at most $4\times 6 \times 3 = 72$, and dividing $72$. This is an upper bound. But we also have a well-known lower bound: the suspension $\Sigma TMF[1/2]$ generates a nontrivial element of $\pic (TMF[1/2])$ of order $72$ because $TMF[1/2]$ is $72$-periodic. Thus we have proven the result.
\end{proof}

\begin{remark}\label{rem:Pic1/2}
 Our computations give an independent proof of the result of Fulton-Olsson
 \cite{FultonOlsson} that the Picard group of the classical moduli stack of
 elliptic curves $\mell$ over $\Z[1/2]$ is $\Z/12$. (Fulton-Olsson carry out
 the analysis over any base, though.)
 This is a toy analog of the above analysis, as we now see.

The Picard \emph{groupoid} of the moduli stack $\mell[1/2]$ is the
 homotopy fixed points of the $GL_2(\mathbb{Z}/2)$-action on the Picard
 groupoid of $\mell(2)$. Now the Picard \emph{group} of $\mell(2)$ is
 $\mathbb{Z}/2$, as $\mell(2)$ is an open subset in a  weighted projective
 stack over a UFD, so that quasi-coherent sheaves on $\mell(2)$ correspond
 simply to graded modules over $\mathbb{Z}[1/2, \lambda_1,
 \lambda_2, (\lambda_1^2 \lambda_2^2 (\lambda_1 - \lambda_2))^{-1}]$ and the
 only nontrivial invertible object is the shift by one of the unit. Note that
 this is the \emph{algebraic} setting: the generator of $\pic (\mell(2))$ would correspond to the \emph{two-fold} suspension of $TMF(2)$.

 Next, in the HFPSS for computing $\pic( \mell[1/2])$,
 we see by the above computation of \[H^1\left(\GL{2}, \Gamma(
 \mell(2), \mathcal{O}^{\times})\right) \] that one gets a contribution of order $6$.
 Together with $\pic(\mell(2))=\Z/2$ from the previous paragraph, we get that $|\pic( \mell[1/2])| \leq 12$, but we know that $\omega$ has order
 twelve, so we are done. 
 \end{remark}

\subsection{The Picard group of $TMF[1/3]$}\label{sec:TMF1/3}

This section will be similar to \Cref{sec:TMF1/2}, but with more complicated computations as is to be expected from $2$-torsion. In this case we will use the $\GL{3}$-Galois extension $TMF[1/3] \to TMF(3)$, coming from the Galois cover $\mell(3)\to \mell[1/3]$ of the moduli stack of elliptic curves with $3$ inverted by the moduli stack of elliptic curves equipped with a full level $3$-structure.
	
From \cite[\S 4.2]{TmfDescentp2}, we can immediately compute the homotopy groups of $TMF(3)$: the moduli stack $\mell(3)$ is affine, and is given as the locus of non-vanishing of 
\[ \Delta = 3^{-5}\zeta (1-\zeta)\gamma_1^3 \gamma_2^3(\gamma_1+\zeta \gamma_2)^3(\gamma_2-\zeta \gamma_1)^3 \]
in the compact moduli stack $\mellc(3) = \proj\,
\Z[1/3,\zeta][\gamma_1,\gamma_2]$. Here $\gamma_i$ are variables in (topological) degree $2$,
and $\zeta$ is a primitive third root of unity, whose appearance is due to the
fact that the Weil pairing on the $3$-torsion points of an elliptic curve
equips $\mellc(3)$ with a map to $\spec \,\Z[1/3,\zeta]$.\footnote{The map is given by the usual Weil pairing on the locus of smooth curves; for what it does at the cusps, see for example \cite[IV.3.21]{DeligneRapoport}.} Hence the descent spectral sequence computing $TMF(3)_*$ collapses to give
\[
TMF(3)_* =\Z[1/3,\zeta][\gamma_1^{\pm 1},\gamma_2^{\pm 1}] [(\gamma_1+\zeta\gamma_2)^{-1}, (\gamma_2-\zeta\gamma_1)^{-1}].
\]
Written differently, we have that $TMF(3)_* = TMF(3)_0[\gamma_2^{\pm 1}]$, and
\begin{align}\label{eq:TMF(3)_0}
TMF(3)_0 = \Z[1/3,\zeta][t^{\pm 1}, (1-\zeta t)^{-1}, (1+\zeta^2 t)^{-1} ] , 
\end{align}
for $t=\frac{\gamma_1}{\gamma_2}$. In particular $TMF(3)_0$ is regular noetherian, and $TMF(3)$ is even periodic. Thus, \Cref{evenperiodicreg} 
(together with the fact that the ring $\mathbb{Z}[\zeta, t]$ and hence any of
its localizations has unique factorization)
implies the following conclusion.
\begin{lemma}
The Picard group $\pic (TMF(3))$ is $\Z/2$, generated by $\Sigma TMF(3)$.
\end{lemma}

Naturally, we will use this lemma as an input in computing the HFPSS for the associated Picard spectra.

\begin{theorem}\label{thm:picTMF1/3}
The Picard group of $TMF[1/3]$ is $\Z/192$, generated by the suspension $\Sigma TMF[1/3]$. In particular, this Picard group is algebraic.
\end{theorem}
\begin{proof}
As is to be expected, we use the HFPSS \Cref{ss:galDescent}
\begin{align}\label{ss:TMF(3)}
H^s( \GL{3}, \pi_t \pics (TMF(3)) ) \Rightarrow \pi_{t-s} \pics (TMF(3))^{h\GL{3}}.
\end{align}
The homotopy groups $\pi_t (\pics (TMF(3)))$ for $t\geq 2$ are isomorphic to $\pi_{t-1} TMF(3)$ as $\GL{3}$-modules, hence the $t\geq 2$ part of the $E_2$-page of the HFPSS for $\pics (TMF(3))$ is same as the corresponding part in the HFPSS for $TMF(3)$. We will use the fact that $TMF(3) \simeq Tmf(3)[\Delta^{-1}]$ to identify this part of the spectral sequence for $TMF(3)$ and therefore for $\pics (TMF(3))$.

The $E_2$-page of the HFPSS computing the homotopy groups of $Tmf\hat{ _2}$ as $(Tmf(3)\hat{ _2})^{h\GL{3}}$ is computed in \cite{TmfDescentp2}, and depicted in Figure 9 of loc. cit. Since we are working with $3$ inverted, and $2$ and $3$ are the only primes dividing the order of $\GL{3}$, we conclude that
 \[H^{>0}(\GL{3}, \pi_* Tmf(3)) = H^{>0}(\GL{3}, \pi_* Tmf(3)\hat{ _2}) .\] 
The invariants $H^0(\GL{3}, \pi_{\geq 0}Tmf(3))$ are the ring of modular forms $\Z[1/3][c_4,c_6,\Delta]/(12^3 \Delta - c_4^3 +c_6^2) $.

Let $\Gamma$ denote the graded ring $\Z[1/3,\zeta][\gamma_1,\gamma_2] $. As in the case of level $2$-structures, we have that
\[ H^*(\GL{3}, \pi_* TMF(3) ) = H^*(\GL{3}, \Gamma)[\Delta^{-1}].\]
In the group cohomology of $\Gamma$, computed and depicted in Figure 7 of
\cite{TmfDescentp2}, there are a number of interesting torsion classes,
including:
\begin{enumerate}
\item $h_1 $ in bidegree $(s,t)=(1,2)$, depicted in position $(s,t-s)=(1,1)$, which detects (the Hurewicz image of) the Hopf map $\eta$ in homotopy.
\item $h_2$ in position $(1,3)$, which detects (the Hurewicz image of) the Hopf map $\nu$.
\item $d$ in position $(2,14)$, which detects in homotopy the class known as $\kappa$.
\item $g$ in position $(4,20)$, which detects in homotopy the class $\bar \kappa$.
\item $c$ in position $(2,8)$, which detects in homotopy the class $\epsilon$.
\end{enumerate}
The homotopy elements detected by these classes satisfy some relations; for example, 
\[\eta^3 = 4\nu, \quad \kappa \nu^2 = 4 \bar \kappa.\]
Let us also name one of the less famous elements in the descent spectral
sequence for $tmf_{(2)}$, which also appears in the HFPSS for $TMF[1/3]$.
Namely, there is a $\Z/2$ in position $(1,5)$; we will denote the generating class by
the generic name $x$ (in \cite{computation} it bears the name $a_1^2h_1$).

All torsion classes with the exception of (powers of) $h_1$ are annihilated by $c_4$ and $c_6$. In the Picard spectral sequence, all of these classes appear shifted by one to the right; we have labeled some such classes in \Cref{fig:picTMF(3)z1}. A ``zoomed in" portion of the Picard spectral sequence is depicted in \Cref{fig:picTMF(3)z2}. There, and in all of the related spectral sequences, lines of slope 1 denote $h_1 $-multiplication, and lines of slope $1/3$ denote $h_2$-multiplication.

\begin{figure}[h!]
\includegraphics[scale=0.9]{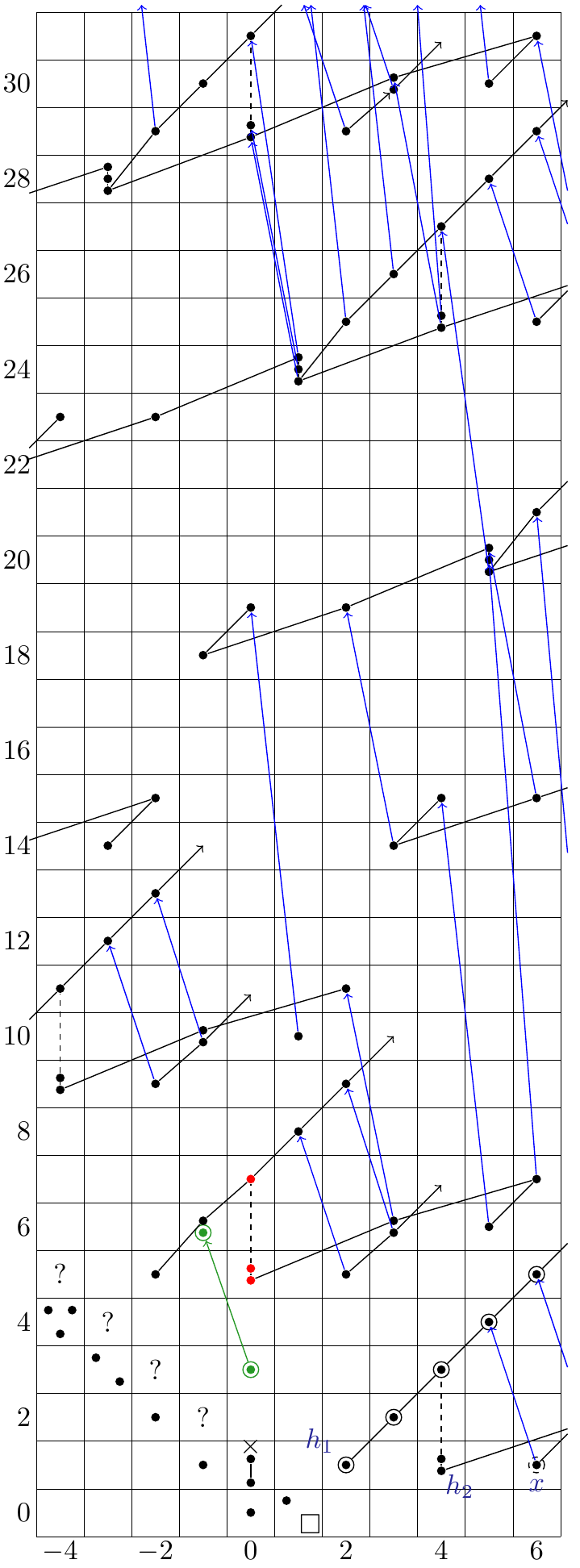}
\caption{Homotopy fixed point spectral sequence for
$\pics(TMF(3))^{h\GL{3}}$ }\vspace{-4mm}
\caption*{\emph{``Zoomed out" version with some $h_1$-omissions}}
\vspace{-4mm}
\caption*{($\square$ denotes $\Z$, $\bullet$ denotes $\Z/2$, $\cbull$ denotes $\Z/2[j]$, and $\times$ denotes $\Z/3$) }\label{fig:picTMF(3)}
\end{figure}

A ``zoomed out" portion of the Picard HFPSS \eqref{ss:TMF(3)} is depicted in \Cref{fig:picTMF(3)}; the elements that are to the right of the $t=2$ line are, of course, a shift of the corresponding elements in the spectral sequence for $TMF[1/3]$. However, to avoid cluttering the picture, a family of classes is not shown.  The family consists precisely of the $h_1$-power multiples of non-torsion classes. An exception is made for the elements depicted in green, namely $h_1^3\frac{c_4c_6}{\Delta}$ and $h_1^6\frac{c_4^2}{\Delta}$, as well as the tower supported on $1$, which do belong to this family, but are nonetheless depicted. In the ``zoomed in" \Cref{fig:picTMF(3)z1} this family is also not shown.

More specifically, the non-torsion subring of the $E_2$-page of the
$TMF[1/3]$-spectral sequence is precisely the part in cohomological degree $0$
and consists of the ring of modular forms
$MF_*[1/3]=\Z[1/3][c_4,c_6,\Delta^{\pm 1}]/(12^3 \Delta - c_4^3+c_6^2) $. On
the $E_2$-page, these support infinite $h_1$-multiples, i.e., $MF_*[1/3] [h_1]/(2h_1)$ is a subring of the $E_2$-page. Note that in degree zero, $MF_0[1/3] = \Z[1/3, j]$, where $j=\frac{c_4^3}{\Delta}$ is the classical $j$-invariant.
What we have omitted drawing in \Cref{fig:picTMF(3)} and \ref{fig:picTMF(3)z1}
are all of the elements coming from this subring, with the exception of the
mentioned classes. For comparison, these elements are drawn in the
smaller-range \Cref{fig:picTMF(3)z2}.

\begin{figure}[h!]
\includegraphics[scale=.9]{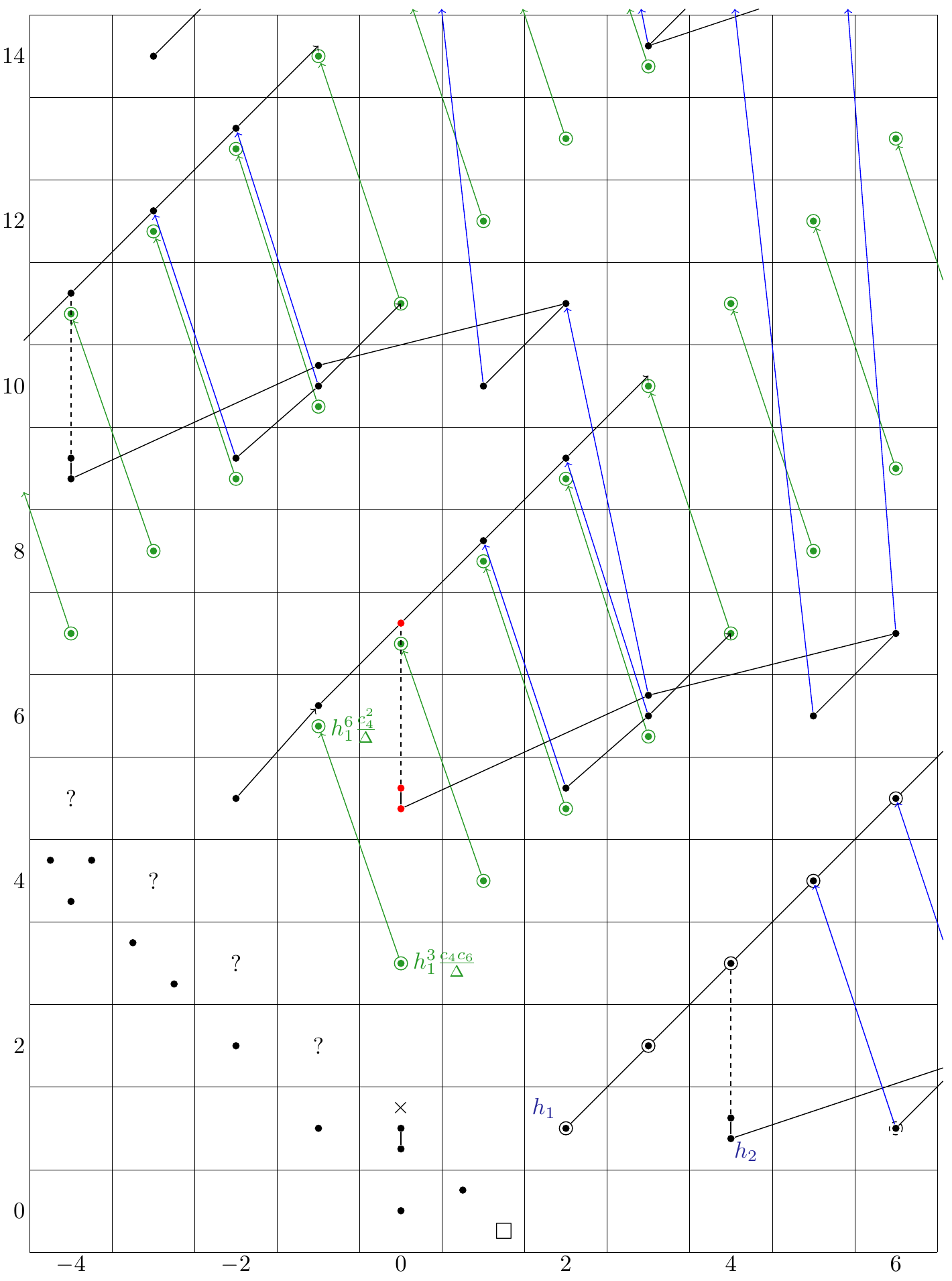}
\caption{Homotopy fixed point spectral sequence for
$\pics(TMF(3))^{h\GL{3}}$}
\vspace{-4mm}
\caption*{\emph{``Zoomed in" version without omissions}}
\vspace{-4mm}
\caption*{  ($\square$ denotes $\Z$, $\bullet$ denotes $\Z/2$, $\cbull$ denotes $\Z/2[j]$, and $\times$ denotes $\Z/3$) }
\label{fig:picTMF(3)z2}
\end{figure}

\begin{remark}
The two classes $h_1^3\frac{c_4c_6}{\Delta}$ and $h_1^6\frac{c_4^2}{\Delta}$, which we have depicted in green, do not appear in the spectral sequence for $Tmf[1/3]$, since they involve a negative power of $\Delta$.
Another difference between the $Tmf$ and $TMF$ situation is that in the $E_2$-page of the latter, there are infinite groups, isomorphic to $\Z/2[j]$ and generated by $h_1$, $h_1^2$, $h_1^3$, etc., in positions $(1,1)$, $(2,2)$, $(3,3)$, etc. Moreover, the element $x$ in position $(1,5)$ also generates an infinite $\Z/2[j]$, as do all of its $h_1$-multiples.
\end{remark}

\begin{figure}[h!]
\includegraphics[scale=0.75]{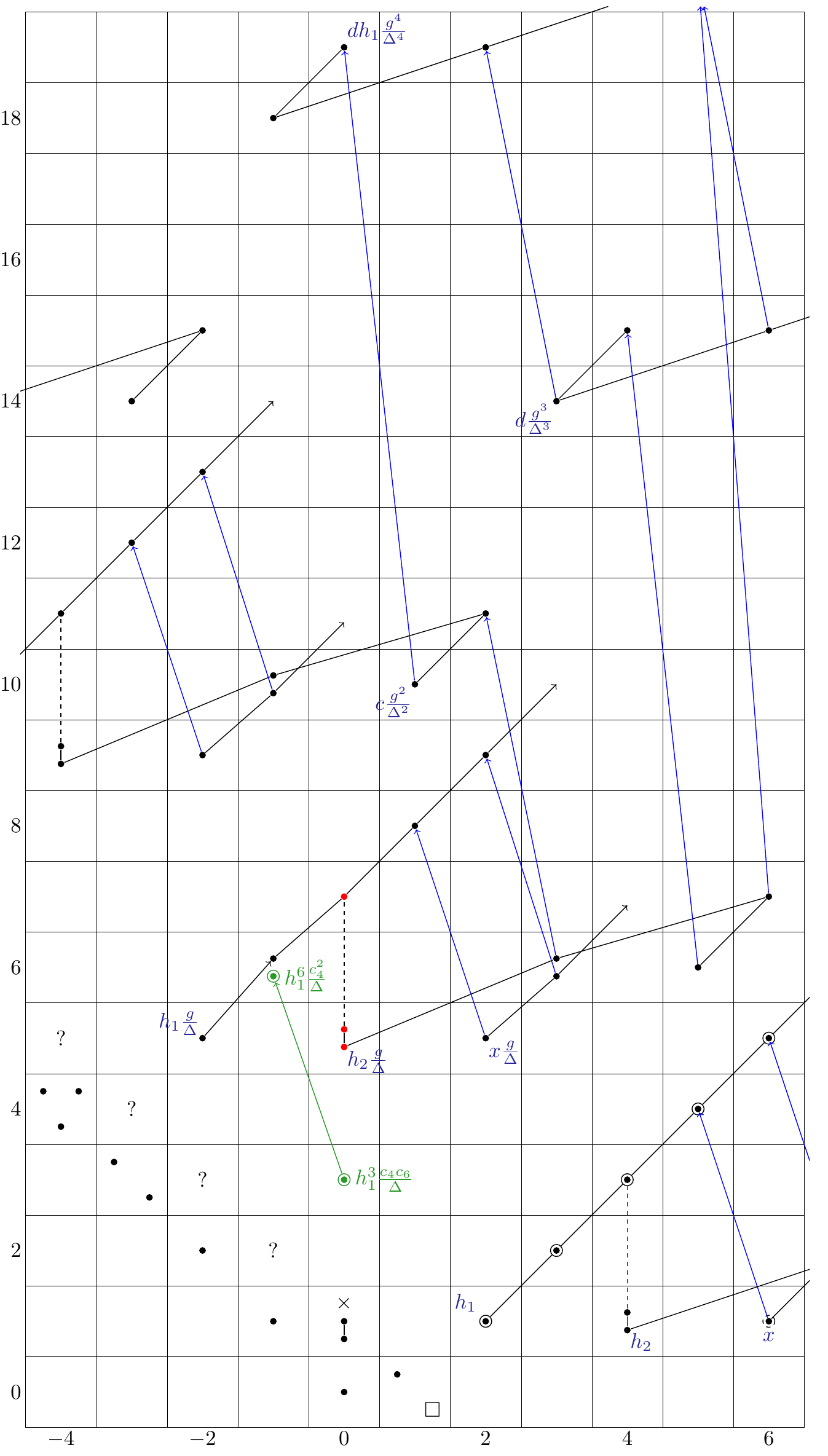}
\caption{Homotopy fixed point spectral sequence for
$\pics(TMF(3))^{h\GL{3}}$}
\vspace{-4mm}
\caption*{\emph{``Zoomed in" version with some $h_1$-omissions}}
\vspace{-4mm}\caption*{  ($\square$ denotes $\Z$, $\bullet$ denotes $\Z/2$, $\cbull$ denotes $\Z/2[j]$, and $\times$ denotes $\Z/3$) }
\label{fig:picTMF(3)z1}
\end{figure}

Note that in the range that we are considering (namely, $t>1$), the HFPSS for the ${\GL{3}}$-action on $Tmf(3)$ coincides with the descent spectral sequence for $Tmf[1/3]$ as the sections of $\mathcal{O}^{top}$ over $\mellbc[1/3]$, and the differentials in the latter have been fully determined in Johan Konter's master thesis \cite{konter}. Of course, these differentials really come from the connective $tmf$, whose descent spectral sequence is fully computed in \cite{computation}. In these spectral sequences, $d^o_3$ is the first non-trivial differential, followed by $d^o_5, d^o_7, d^o_9\dots d^o_{23}$. In particular, we have the following differentials \cite[Sec. 8]{computation}:
\begin{equation}\label{diffd3-11}
\begin{aligned}
d^o_3(c_6)&= c_4h_1^3 \qquad \qquad &d^o_3(x)&= h_1^4\\
d^o_5(\Delta) & =  g h_2 \qquad \qquad &d^o_7(4\Delta) &=g h_1^3\\
d^o_9(\Delta^2h_1) &=   g^2 c \qquad\qquad  &d^o_{11}(d \Delta^2) &=g^3 h_1,
\end{aligned}
\end{equation}
and a number of others.

Let us see now which of these differentials we can import using our
Comparison Tool~\ref{toolGalois}. In the $TMF[1/3]$ spectral sequence, we have that $d_3^o
(h_1^3\frac{c_4c_6}{\Delta}) = h_1^6\frac{c_4^2}{\Delta} $; in the Picard SS,
the element corresponding to $h_1^3\frac{c_4c_6}{\Delta}$ has $t=3$, thus we
\emph{cannot} import this differential. We deal with this class later, i.e., in
the next paragraph. However, all the other classes which are on the $s=t$
column and are $h_1$-power multiples of non-torsion classes, i.e., members of
the family which we have not drawn in \Cref{fig:picTMF(3)}, are well within the
$t>3$ range, so that we can indeed conclude by Comparison Tool~\ref{toolGalois} that they
either support a differential or are killed by one. For example, the
$h_1$-multiple of the differential just discussed does happen, i.e., in the Picard SS we have $ d_3(h_1^4\frac{c_4c_6}{\Delta}) = h_1^7\frac{c_4^2}{\Delta} $. In particular, we need not worry about these omitted classes any more.

Now we turn to the question of whether any differentials are supported on the $(s,t-s)= (3,0)$ position in the HFPSS for $\pics (TMF(3))^{h\GL{3}}$. For this purpose we use the universal formula \eqref{univformula} of \Cref{difftheorem}, just as we did in the second proof of \Cref{KO[[q]]}. We have that $E_2^{3,3}$ of the Picard spectrum HFPSS is $\Z/2[j]$ generated by $h_1^3 \frac{c_4c_6}{\Delta}$; the corresponding element in the original HFPSS has
\[ d_3^o \left(h_1^3 \frac{c_4c_6}{\Delta} \right)  = h_1^6 \frac{c_4^2}{\Delta}.\] 
Now we have that 
\[ \left( h_1^3 \frac{c_4c_6}{\Delta} \right)^2 = h_1^6 \frac{c_4^2 c_6^2}{\Delta^2} = (j-12^3)h_1^6\frac{c_4^2}{\Delta} = j h_1^6 \frac{c_4^2}{\Delta},\]
using the fact that $12^3\Delta = c_4^3-c_6^2 $ and that by definition, $j=\frac{c_4^3}{\Delta}.$ Therefore, we conclude by  \eqref{univformula} that in the Picard HFPSS, the differential $d_3:E_3^{3,3} \to E_3^{6,5}$ is given by
\[ d_3 \left(f(j) h_1^3 \frac{c_4c_6}{\Delta} \right) = (f(j)+j f(j)^2) h_1^6\frac{c_4^2}{\Delta},  \] 
where $ \left(f(j) h_1^3 \frac{c_4c_6}{\Delta} \right)  $ is an arbitrary element of $E_3^{3,3}$. However, $(f(j)+j f(j)^2) $ in $\Z/2[j]$ is zero only if $f(j)$ is zero, hence this $d_3$ is injective and has trivial kernel. (Note this is an interesting difference between the present situation and the one in \Cref{KO[[q]]}.) Consequently, $E_4^{3,3}$ is zero.

Further use of Comparison Tool~\ref{toolGalois} determines that all the differentials we have
drawn in blue in \Crefrange{fig:picTMF(3)}{fig:picTMF(3)z1} are non-zero. Note that of the classes in the $s=t$
column, i.e., the one which contributes to the Picard group of $TMF[1/3]$, everything with $s\geq 8$ is killed. However, $h_2 \frac{g}{\Delta}$, generating a $\Z/4$ in $s=5$, and $h_1^3 \frac{g }{\Delta}$ generating a $\Z/2$ in $s=7$, remain. In the original spectral sequence, the first one of these supported a $d^o_5$ and a $d_{13}^o$, and the second supported a $d^o_{23}$.

Next we need to determine the rest of the spectral sequence, i.e., the part which involves $\pi_0 $ and $\pi_1$ of the Picard spectrum of $TMF(3)$. Detailed computations for this are deferred to \Cref{sec:computeTMF(3)}. The piece in which we are most interested is $ H^1(\GL{3}, \pi_1\pics(TMF(3)) ) $, a cyclic group of order $12$ according to \Cref{grpcohGL3}; we have also determined $H^*(\GL{3}, \pi_0 \pics(TMF(3)) )$ in \Cref{Z2cohG3} using a more general result of Quillen's.

At this point we are ready to make conclusions about the Picard group of
$TMF[1/3]$: in the $t=s$ vertical line of the HFPSS, i.e., the one that abuts to $\pi_0 \pics (TMF[1/3] )=\pic (TMF[1/3])$, nothing above the $s=7$ line survives the spectral sequence. The following might survive:
\begin{itemize}
\item at most a group of order $2$ in position $ (0,0)$,
\item at most a group of order $12$ in $(1,0)$,
\item at most a group of order $4$ in $(5,0)$, and
\item at most a group of order $2$ in $(7,0)$.
\end{itemize}
The upshot is that we get an upper bound of $2\times 12\times 4\times 2=192$ on the order of the Picard group. But $TMF[1/3]$ is $192$-periodic, so this upper bound must also be a lower bound. In conclusion, \[\pic(TMF[1/3]) = \Z/192\]  as claimed, generated by $\Sigma TMF[1/3]$.
\end{proof}

\begin{remark}\label{rem:Pic1/3}
As in \Cref{rem:Pic1/2}, we can use some of our computations to reprove Fulton-Olsson's \cite{FultonOlsson} result that the moduli stack of elliptic curves $\mell[1/3]$ also has a Picard group $\Z/12$. Namely, we start with the knowledge that $\pic (\mell(3))$ is trivial, as $\mell(3)$ is the prime spectrum of a UFD. Then, we consider the Picard HFPSS for the algebraic stack $\mell[1/3]$, which must collapse. The only contribution towards the Picard group is 
\[ H^1\left(\GL{3}, \Gamma(\mell(3),{\mathcal{O}}^\times) \right), \]
which we saw by \Cref{grpcohGL3} has order $12$. But $\omega$ has order $12$, hence $\pic (\mell[1/3])$ is cyclic of order $12$.
\end{remark}


\subsection{Calculation of $\pic(\TMF)$}\label{subs:TMF}

In this section we will compute the Picard group of the integral periodic version of topological modular forms $TMF$. 
The result, as stated in the introduction, is:

\begin{introtheoremA}
The Picard group of integral $TMF$ is $\mathbb{Z}/576$, generated by $\Sigma
TMF$.
\end{introtheoremA}

\begin{proof}

There is no nontrivial Galois extension of the integral $TMF$ by
\cite[Theorem 10.1]{galois}, 
but we can use \'etale descent, as $TMF$ is obtained as the global sections of the sheaf $\mathcal{O}^{top}$ of even-periodic $E_\infty$-rings on the moduli stack of elliptic curves. 
Namely, we can use \Cref{descentss} because the map $M_{ell} \to M_{FG}$ is
known to be affine. 
The spectral sequence is
\[ H^s(M_{ell}, \pi_t \pics \mathcal{O}^{top}) \Rightarrow \pi_{t-s} \Gamma  (\pics \mathcal{O}^{top}), \]
and we are interested in $\pi_{0}$. Using \Cref{descentss}, the $E_2$-page of this spectral sequence is given by 
(for $t-s \geq 0$)
\[E_2^{s,t} = \begin{cases}
\begin{aligned}
&\Z/2 & t=s=0\\
&H^s(\mell, \mathcal{O}_{\mell}^\times) & t=1\\
&H^s(\mell, \omega^{(t-1)/2}) & t\geq 3, \text{ odd}\\
&0 & \text{otherwise.}
\end{aligned}
\end{cases}
 \]

Over a field $k$ of characteristic $\neq 2, 3$, Mumford \cite{Mumford} showed that \[H^1((\mell)_k,  \mathcal{O}_{\mell}^\times)
\simeq \mathbb{Z}/12,\]
i.e., the Picard group of the moduli stack is $ \mathbb{Z}/12$, generated by the line bundle $\omega$ that assigns to an elliptic curve the dual of its Lie algebra. 
This result is also true over $\mathbb{Z}$ by the work of Fulton-Olsson
\cite{FultonOlsson}. However, using descent we can reprove that result. Namely,
in \Cref{rem:Pic1/2,rem:Pic1/3} we saw that the Picard groups of both
$\mell[1/2]$ and $\mell[1/3]$ are $\Z/12$, both generated by $\omega$. Cover
the integral stack $\mell$ by these two; their intersection is $\mell[1/6]$,
which is the weighted projective stack $\proj \, \Z[1/6][c_4,c_6]$ (with $c_4$
and $c_6$ in degrees\footnote{These are the algebraic degrees, which get
doubled in topology.} 4 and 6 respectively), and which therefore has Picard group $\Z/12$ also generated by $\omega$. The descent spectral sequence for $\pics$ associated to this cover gives the result.

Because $\mell[1/6]$ has no higher cohomology, the groups $H^s(\mell, \omega^{(t-1)/2})$, when $s>0$, are given as the direct sum of the corresponding cohomology groups of $\mell[1/2]$ and $\mell[1/3]$. These groups, in turn, are isomorphic to 
\[H^s(\GL{p}, \pi_{t-1}  TMF( p)) = H^s(\GL{p}, H^{0}(\mell( p ), \omega^{(t-1)/2}) ) ,\] where $p$ is $2$ or $3$, as the map $\mell( p) \to \mell[1/p]$ is Galois, and $\mell( p)$ has no higher cohomology. We computed these groups in the previous examples.

The machinery of \Cref{sec:SSLemma} now allows us to compare this Picard
descent spectral sequence to the one which computes the homotopy groups of
$TMF$. From \Cref{gl1RvsR} and an analogue of Comparison Tool~\ref{toolGalois}, we conclude that the differentials involving $3$-torsion classes wipe out everything above the $s=5$ line, and those involving $2$-torsion classes wipe out everything above the $s=7$ line. These differentials are identical to what happens in the homotopy fixed point spectral sequences in the previous two examples. We conclude that the following are the only groups that can survive:
\begin{itemize}
\item at most a group of order $2$ in $(t-s,s)= (0,0)$,
\item at most a group of order $12$ in $(0,1)$,
\item at most a group of order $12$ in $(0,5)$, and
\item at most a group of order $2$ in $(0,7)$.
\end{itemize}
This gives us an upper bound $2^6 3^2=576$ on the cardinality of $\pi_{0}$, which is exactly the periodicity of $TMF$. The spectral sequence is depicted in \Cref{fig:picTMF}.
\end{proof}

\begin{figure}[h!]
\includegraphics[scale=0.75]{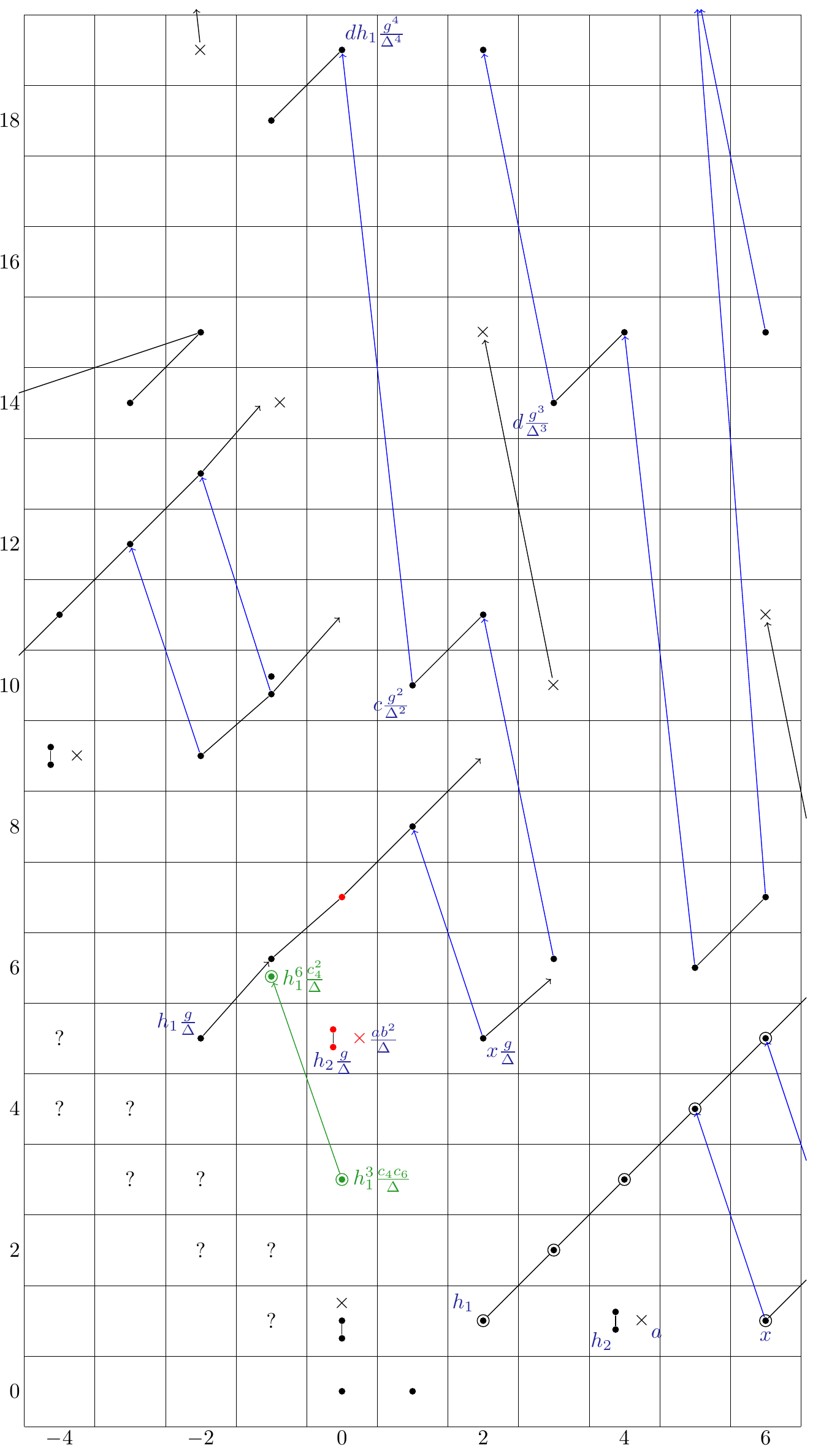}
\caption{Descent spectral sequence for $\Gamma(\pics \mathcal{O}^{top})$ on
$\mellb$} \vspace{-4mm} 
\caption*{\emph{with some $h_1$-omissions as in \Cref{fig:picTMF(3)}}}
\vspace{-4mm}
\caption*{  ($\square$ denotes $\Z$, $\bullet$ denotes $\Z/2$, $\cbull$ denotes $\Z/2[j]$, and $\times$ denotes $\Z/3$) }\label{fig:picTMF}
\end{figure}

\subsection{Calculation of $\pic(\Tmf)$}

We will now prove the following result stated in the introduction.
\begin{introtheoremB}
The Picard group of $Tmf$ is $\mathbb{Z} \oplus \mathbb{Z}/24$, generated by
$\Sigma Tmf$ and a certain 24-torsion invertible module. 
\end{introtheoremB}

Note that while $Tmf[1/n]$, for $n=2,3$, can be described as the homotopy fixed point spectrum $Tmf(n)^{hGL_2(\Z/n)}$ just as in the periodic case, the extension $Tmf[1/n]\to Tmf(n)$ is \emph{not} Galois, and therefore we cannot use Galois descent to compute the Picard group. However, we can use Theorem \ref{descentss} for the compactified moduli stack $\mellc$. 

First, we need a lemma.
\begin{lemma} \label{lem:omega12}
Let $\mathcal{L}$ be the line bundle on $\mellc$ obtained by gluing the trivial
line bundles on $\mell = \mellc[\Delta^{-1}]$ and $\mellc[c_4^{-1}]$ via the
clutching function $j$. Then $\mathcal{L} \simeq \omega^{-12}$.
\end{lemma} 
\begin{proof} 
To give a section of $\mathcal{L} \otimes \omega^{12}$ over $\mellc$ is equivalent to giving
sections $s_1 \in \Gamma( \mell, \omega^{12})$ and $s_2 \in \Gamma(
\mellc[c_4^{-1}], \omega^{12})$ such that $(j s_1)|_{\mell[c_4^{-1}]} =
(s_2)|_{\mell[c_4^{-1}]}$. We take $s_1  = \Delta$ and $s_2 = c_4^3$, and we
get a nowhere vanishing section of $\mathcal{L} \otimes \omega^{12}$.
\end{proof}

\begin{proof}[Proof of Theorem~\ref{thm:PicTmf}]
The relevant part of the Picard descent spectral sequence is similar to that
of
$TMF$, with the following exceptions: the algebraic part $H^1(\mellc,
\mathcal{O}^\times)$ is now $\Z$ generated by $\omega$, according to Fulton-Olsson \cite{FultonOlsson}, and all
the torsion groups are now finite, i.e., there are no $\Z/2[j]$'s appearing. In particular, $E_2^{3,3}$ is zero, and we have 
\begin{itemize}
\item at most a group of order $2$ in $(t-s,s)= (0,0)$,
\item a sub-quotient of $\Z$ in $(0,1)$,
\item at most a group of order $12$ in $(0,5)$, and
\item at most a group of order $2$ in $(0,7)$,
\end{itemize}
as potential contributions to the $s=t$ line of the $E_\infty$-page. The depiction is in \Cref{fig:picTMFcpt}.

Note that the $\Z/2$ in $(0,0)$, which corresponds to a single suspension of the even-periodic spectra that $Tmf$ is built from, is represented by $\Sigma Tmf$ in the Picard group of $Tmf$. Similarly, the element $1\in \Z = E_2^{0,1} = \pic (\mellc)$ corresponds to the line bundle $\omega$, which topologically is represented by $\Sigma^2 Tmf$. Thus these groups survive to the $E_\infty$-page and are related by an extension. The rest of the $E_\infty$-filtration now tells us that $\pic( Tmf)$ sits in an extension
\[0\to A \to \pic (Tmf) \to \Z \to 0, \]
where $A$ is a finite group of order at most $24$.

We claim that $A=\Z/24$ and therefore $\pic(Tmf) =\Z\oplus \Z/24$. To see this, we will construct a line bundle $\mathcal{I}$ such that $\mathcal{I}^{\otimes 24} \simeq \mathcal{O}^{top}$, but no lower power of $\mathcal{I} $ is equivalent to $\mathcal{O}^{top}$.

In order to proceed with the construction, we make the preliminary observation that the modular function $j=\frac{c_4^3}{\Delta}$ is a homotopy class in $\pi_0 TMF[c_4^{-1}]$, i.e. it survives the descent spectral sequence
\[ H^*(\mellc[\Delta^{-1}, c_4^{-1}], \omega^* ) \cong H^*(\mell, \omega^*)[c_4^{-1}]  \Rightarrow \pi_*TMF[c_4^{-1}]. \]
In fact, it is an invertible element of $\pi_0 TMF[c_4^{-1}]$. We reason as follows.
The torsion in the $E_2$-page consists only of $h_1$-towers supported on the non-torsion classes, since all other torsion classes in $H^*(\mell, \omega^*)$ are annihilated by $c_4$. Therefore, when $c_4$ is inverted only $d_3$-differentials can be non-zero, and they wipe out everything above the line $s=3$. As $\Delta$ and $c_4$ do not support any of those differentials, $j$ is a permanent cycle, as is $j^{-1}$.

\begin{construction}\label{construct}
Consider the cover of $\mellc$ by $\mellc [\Delta^{-1}] = \mell$ and $\mellc[c_4^{-1}]$ which fit in the pushout diagram
\[ \xymatrix{ 
\mellc [\Delta^{-1}, c_4^{-1}] \ar[r] \ar[d] & \mellc[\Delta^{-1}] \ar[d] \\
\mellc [c_4^{-1}] \ar[r] & \mellc.
}\]
Let $\mathcal{J}$ be the line bundle on the \emph{derived} moduli stack
$\mellbc = (\mellc, \otop)$ obtained by gluing $\mathcal{O}^{top}$ on $\mellc[\Delta^{-1}]$
and $\mathcal{O}^{top} $ on $\mellc[c_4^{-1}]$ using the clutching function
$j=\frac{c_4^3}{\Delta}$ on $\mellc[\Delta^{-1},c_4^{-1}]$. 
\end{construction}

\begin{figure}[h!]
\includegraphics[scale=0.75]{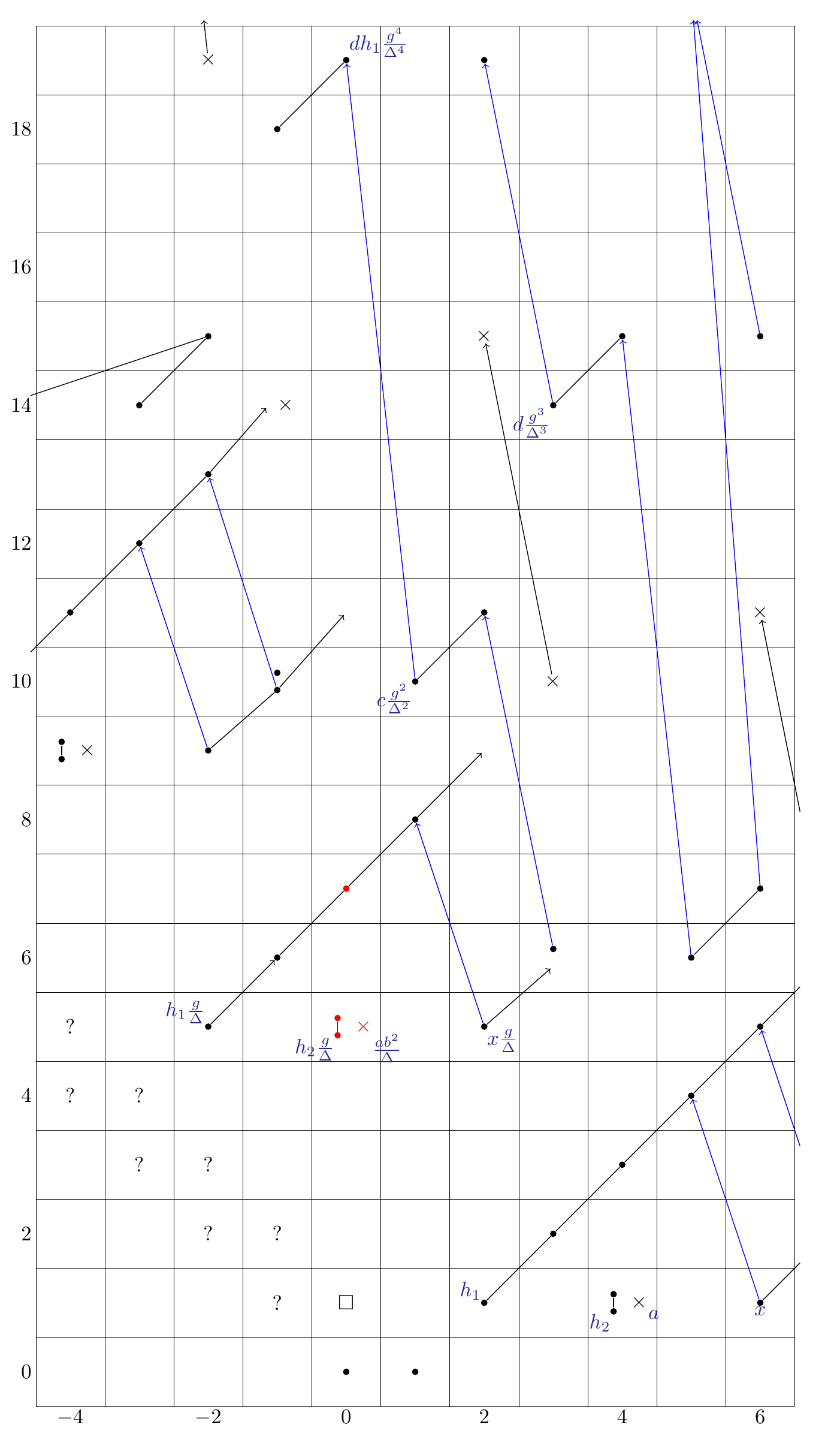}
\caption{Descent spectral sequence for $\Gamma(\pics \mathcal{O}^{top})$ on
$\mellbc$} \vspace{-4mm} \caption*{ ($\square$ denotes $\Z$, $\bullet$ denotes $\Z/2$, and $\times$ denotes $\Z/3$) }\label{fig:picTMFcpt}
\end{figure}

We claim that $\mathcal{J}$ is not a suspension of $\mathcal{O}^{top}$, and that $\mathcal{I} = \Sigma^{24}\mathcal{J}$ is an element of the Picard group of order $24$.

To see the first assertion, note that by \Cref{lem:omega12}, $\pi_0\mathcal{J} $ is $\omega^{ -12}$, so if
$\mathcal{J}$ is a suspension of $\mathcal{O}^{top}$, it ought to be
$\Sigma^{-24}\mathcal{O}^{top}$. However, $\Sigma^{-24}\mathcal{O}^{top}$
restricted to $\mellc[\Delta^{-1}]$ is $\Sigma^{-24}\mathcal{O}^{top}|_{\mellc[\Delta^{-1}]}$, whereas
$\mathcal{J}$ restricts to $\mathcal{O}^{top}|_{\mellc[\Delta^{-1}]}$. 

This argument can be repeated with any power $\mathcal{J}^{\otimes m}$ such that $m$ is not divisible by $24$. In this case, $\pi_0 \mathcal{J}^{\otimes m}$ is $\omega^{-12m}$, so if $\mathcal{J}^{\otimes m}$ were a suspension of $\mathcal{O}^{top}$, it would be the $(-24)m$-th suspension. At the same time, $\mathcal{J}^{\otimes m}$ restricts to \[(\mathcal{O}^{top})^{\otimes m}|_{\mellc[\Delta^{-1}]} = \mathcal{O}^{top}|_{\mellc[\Delta^{-1}]}\] upon inverting $\Delta$. If $\mathcal{J}^{\otimes m}$ were a suspension, therefore, one would have that $\Sigma^{-24m} \mathcal{O}^{top} |_{\mellc[\Delta^{-1}]} \simeq \mathcal{O}^{top}|_{\mellc[\Delta^{-1}]}$. By Theorem~\ref{hopkins1}, this holds if and only if $m$ is divisible by $24$. 

This shows that the order of $\mathcal{J}$ in $\pic(\mathcal{O}^{top})/\Z $,
where the $\Z $ is generated by $\Sigma\mathcal{O}^{top}$, is at least $24$. The spectral sequence argument above, however, showed that this quotient has order at most $24$. 
\end{proof}

The same analysis shows that $\pic (Tmf_{(2)} )= \Z\oplus \Z/8$ and $\pic
(Tmf_{(3)}) = \Z\oplus \Z/3$, the torsion being generated by the respective
localizations of $\mathcal{I}$. Moreover, $\pic( Tmf_{( p)})=\Z$ when $p>3$.

\subsection{Relation to the $E_2$-local Picard group}

Notice that $\mathcal{I} $ is the only ``exotic" element in all of our examples
involving the various forms of topological modular forms. Let us see how it
relates to the exotic piece of the Picard group of the category of $E_2$-local
spectra, i.e. modules over the $E_2$-local sphere spectrum. The exotic phenomena only occur at $p=2$ and $p=3$, but since only the $3$-primary $E_2$-local Picard group is known, let us concentrate on that case for the remainder of this section.

In \cite{GHMR}, the authors compute $\kappa_2$, the exotic part of the Picard group of the category of $3$-primary $K(2)$-local spectra; they show $\kappa_2 = \Z/3 \times \Z/3$.

In addition, they look at the localization map from the $E_2$-local category to the $K(2)$-local category, and show that it induces an isomorphism $\kappa_{\mathcal{L}_2} \to \kappa_2 $, where $\kappa_{\mathcal{L}_2} $ denotes the exotic $E_2$-local Picard group. 

Consider now the commutative diagram
\[ \xymatrix{
\kappa_{\mathcal{L}_2} \ar[r] \ar[d]_{t} &\kappa_2 \ar[d]^{t_{K(2)}}\\
\pic (Tmf_{(3)}) \ar[r] & \pic (Tmf_{K(2)}),
} \]
in which the horizontal maps are given by $K(2)$-localization, and the vertical
maps are given by smashing with $Tmf$ and $Tmf_{K(2)}$, respectively. 
In \cite[Theorem 5.5]{GHMR}, the authors show that there is an element $P$ of $\kappa_2$
such that $L_{K(2)}(P \wedge Tmf_{K(2)}) \simeq \Sigma^{48} Tmf_{K(2)}$, i.e.,
$t_{K(2)}P = 48 \in \Z/72 \subseteq \pic (Tmf_{K(2)})$. Under the top
horizontal isomorphism, this $P$ lifts to an element $\widetilde P$ of
$\kappa_{\mathcal{L}_2}$, such that $t(\widetilde P)$ has order three in
$\pic(\Tmf_{(3)})$ and such that the $K(2)$-localization of $t(\widetilde P)$ is $L_{K(2)}(\Sigma^{48} Tmf)$. Therefore, $t(\widetilde P) $ must be twice the class of $\mathcal{I}$. In other words, the exotic element $\widetilde P$ of $\kappa_{\mathcal{L}_2}$ is detected as an exotic element of $\pic (Tmf_{(3)})$.

The other $\Z/3$ in $\kappa_2$, i.e., $\kappa_2 $ modulo the subgroup generated
by $P$, is generated by a spectrum $Q$ such that $t_{K(2)} Q = 0$. This $Q$
lifts to $\widetilde Q \in \kappa_{\mathcal{L}_2}$, still of order $3$, which
must map under $t$ to an element of order $3$ in $\pic (Tmf_{(3)})$ which is in the kernel of the bottom localization map. But there are no non-trivial elements of finite order in this kernel, hence $\widetilde Q$ is not detected in $\pic( Tmf_{(3)})$.

Perhaps at the prime $2$ as well there is an element of the exotic $E_2$-local Picard group which is detected in the torsion of $\pic (Tmf_{(2)})$.

\counterwithout{lm}{subsection}
\counterwithin{lm}{section}

\part*{Appendices}


\appendix
\section{Group cohomology computations for $TMF(2)$}\label{sec:grpcohTMF(2)}

In this section of the appendix we compute the group cohomology for the $\GL{2}$-action on $\pi_0 \pics (TMF(2)) = \Z/4$ (with trivial action), and on $\pi_1 \pics( TMF(2)) = TMF(2)_0^\times$ with the natural action. The group $\GL{2}$ is the symmetric group on three letters, so it has a (unique) normal subgroup of order $3$, which we denote $C_3$, with quotient $C_2$. We can therefore use the associated Lyndon-Hochschild-Serre spectral sequence (LHSSS)
\begin{align}\label{ss:GL2(Z2)}
 H^p(C_2, H^q(C_3, M)) \Rightarrow H^{p+q}(\GL{2}, M) 
 \end{align}
for $\GL{2}$-modules $M$. 

Let us first deal with the easier case.

\begin{lemma}\label{G2trivialZ4}
The group cohomology for the $\GL{2}$-action on the trivial module $\Z/4$ is
\[ H^*(\GL{2}, \pi_0 \pics( TMF(2))) = \begin{cases}
\Z/4, \qquad *=0\\	
\Z/2, \qquad *>0.
\end{cases}
\]
\end{lemma}
\begin{proof}
Since $3$ is invertible in $\Z/4$, we have that $H^*(C_3, \Z/4) = \Z/4$ concentrated in degree zero, and with trivial action by $C_2 = \GL{2}/C_3$. Hence the LHSSS \eqref{ss:GL2(Z2)} collapses, giving 
\[ H^s(\GL{2},\Z/4) = H^s(C_2, \Z/4), \]
which is $\Z/4$ for $s=0$ and $\Z/2$ otherwise.
\end{proof}

Next we will compute the group cohomology for the action of $\GL{2}$ on $\pi_1 \pics (TMF(2))$ which is the multiplicative group of units in $\pi_0 TMF(2)$. For brevity, we call this module $M$, and to begin with, we explicitly describe the action of $\GL{2}$ on $M$.

Let $\sigma$ and $\tau$ be the generators of $\GL{2}$ of order $3$ and $2$ respectively as chosen in \cite[Lemma 7.3]{TmfDualityp3}; of course, $\sigma $ generates the normal subgroup $C_3$. It follows from \Cref{eq:pi_0TMF(2)} that $M$ is isomorphic to $\Z/2 \oplus \Z^{\oplus 3}$, where $\Z/2$ is multiplicatively generated by $-1$, and the $\Z$'s are multiplicatively generated by $2,s$ and $(s-1)$.
The action is determined by \cite[Lemma 7.3]{TmfDualityp3}, where it is shown that the chosen generators $\sigma $ and $\tau$ act as
\begin{align*}
\sigma: s \mapsto \frac{s-1}{s}, \qquad \qquad \tau:  s\mapsto \frac{1}{s}.
\end{align*}
Written additively, so that $m = (\epsilon, k, a, b) \in M$ represents $(-1)^{\epsilon} 2^k s^a (s-1)^b  \in TMF(2)_0^\times$, the action is given by
\begin{align*}
\sigma: m &\mapsto (\epsilon+b, k , -a-b, a), \\ 
\tau: m & \mapsto ( \epsilon+b, k, -a-b, b).
\end{align*}

We use this information to compute $H^*(C_3,M)$ as a $C_2$-module. We get that
\begin{align*}
H^s(C_3, M) = \begin{cases}
\Z/2 \oplus \Z,\quad & s=0,\\
(\Z/3), & s\equiv 0,1(4), \,s>0,\\
(\Z/3)_{sgn}, & s\equiv 2,3 (4),\, s>0.
\end{cases}
\end{align*} 
This gives the $E_2$-page of the LHSSS \cref{ss:GL2(Z2)}, which must collapse and give that
\begin{equation}\label{eq:grpcohTMF(2)}
\begin{aligned}
H^s(\GL{2},M) =\begin{cases}
\Z/2\oplus \Z, \quad &\text{for } s=0,\\
\Z/2\oplus \Z/3, \quad &\text{for } s\equiv 1 (4),\\
\Z/2\oplus \Z/2, \quad &\text{for } s\equiv 2 (4),\\
\Z/2, \quad &\text{for } s\equiv 3 ( 4),\\
\Z/2\oplus \Z/2\oplus \Z/3, \quad &\text{for } s\equiv 0 ( 4),\, s>0.
\end{cases}
\end{aligned}
\end{equation}
We have thus proven the following result.

\begin{proposition}\label{G2pi0units}
The group cohomology for the $\GL{2}$-action on $\pi_0 \pics(TMF(2))$ is as in \Cref{eq:grpcohTMF(2)}. In particular, we have that $H^1(\GL{2},TMF(2)_0^\times)=\Z/6$.
\end{proposition}

\section{Group cohomology computations for $TMF(3)$}\label{sec:computeTMF(3)}

This section of the appendix is devoted to computing the group cohomology for $\GL{3}$ acting on $\pi_1\pics(TMF(3)_0)^\times$; we also determine the cohomology of  $\pi_0 \pics(TMF(3))=\Z/2$ as a simple consequence of a result of Quillen \cite{QuillenGL}. The group $\GL{3}$ has order $48$ and has the binary tetrahedral group as normal subgroup, in the guise of $SL_2(\Z/3)$. We have found it difficult to compute the higher cohomology groups of $(TMF(3)_0)^\times$, but since we are only using $H^1(\GL{3}, (TMF(3)_0)^\times)$ in \Cref{sec:TMF1/3}, we will concentrate on computing this group only.

In this section, we denote $(TMF(3)_0)^\times$ by $M$. From
\Cref{eq:pi_0TMF(2)}, we see that $M \subset TMF(3)_0$ is isomorphic to $\Z/2\oplus\Z/3 \oplus \Z^{\oplus 4}$ multiplicatively generated by $-1, \zeta, (1-\zeta), t, (1-\zeta t),$ and $(1+\zeta^2 t)$. (To see the appearance of $(1-\zeta)$, note that $(1-\zeta)^2 = -3\zeta$.) The $\GL{3}$-module structure is determined in \cite[\S 4.3]{TmfDescentp2}; to describe it, let $x,y,z$ be the elements of $\GL{3}$ chosen in loc. cit. Explicitly,
\[ x=\begin{pmatrix}
0 & -1\\ 1 & 0
\end{pmatrix}, \qquad y = \begin{pmatrix}
-1 & -1 \\
-1 & 1
\end{pmatrix}, \qquad z=\begin{pmatrix}
0 &-1\\
1 & -1
\end{pmatrix}.
 \]
Then $x$ and $y$ generate a quaternion group $Q_8$, and $x,y,z$ generate $SL_2(\Z/3)$. Let $\sigma$ be the matrix $\begin{pmatrix}
1 & 0\\ 0 & -1
\end{pmatrix}
$. These generate the whole group, and their action on the element $t=\frac{\gamma_1}{\gamma_2}$ is as determined in loc. cit.\footnote{Actually, the formulas in loc. cit. determine a \emph{right} action, although the \emph{left} action that we include here is almost the same.} to be
\[ x(t) = - \frac{1}{t}, \qquad y(t) = \zeta^2\frac{1-\zeta t}{1+\zeta^2 t}, \qquad z(t) = \zeta \frac{t}{1+\zeta^2 t}, \qquad \sigma(t) = \frac{1}{t}.  \]
The rest is is determined by the fact that everything fixes $\Z[1/3]\subset TMF(3)_0$, a matrix $A$ in $\GL{3}$ takes $\zeta$ to $\zeta^{\det A}$, and the action respects the ring structure.

To be brutally explicit, let $m=(\epsilon, \alpha, \beta, a, b, c)  \in M$ denote the element $(-1)^\epsilon \zeta^\alpha (1-\zeta)^\beta t^a (1-\zeta t)^b (1+\zeta^2t)^c$. Then the generators $x,y,z,\sigma \in \GL{3}$ act as
\begin{equation}\label{action}
\begin{aligned}
x:m\mapsto & (\epsilon+a+c, \alpha+b-c, \beta, -a-b-c, c,b),\\
y:m\mapsto & (\epsilon+b+c, \alpha-a-c, \beta, b,a,-a-b-c),\\
z:m\mapsto & (\epsilon, \alpha + a,\beta, a,c, -a-b-c ),\\
\sigma:m\mapsto & (\epsilon+\beta+b,-\alpha-\beta-b+c,\beta,-a-b-c,b,c).	
\end{aligned}
\end{equation}

Since we know a set of generators and relations for $\GL{3}$, and the action is given explicitly, we can compute $H^1$ directly as crossed homomorphisms modulo coboundaries. We found it a little bit simpler, however, to do this for $SL_2(\Z/3)$, and then use the Lyndon-Hochschild Serre spectral sequence for the extension
\[ 1\to SL_2(\Z/3) \to \GL{3} \to C_2 \to 1, \]
in which $C_2$ is generated by the image of $\sigma \in \GL{3}$.
The contributions to $H^1(\GL{3}, M)$ are from $H^1(SL_2(\Z/3), M )^{C_2}$ and $H^1(C_2, M^{SL_2(\Z/3)})$, and there is a potential differential
\begin{align}\label{eq:d2diff}
 d_2: H^1(SL_2(\Z/3), M )^{C_2} \to H^2(C_2, M^{SL_2(\Z/3)}). 
 \end{align}

To compute these groups and the differential, we note that the invariants $M^{SL_2(\Z/3)} $ are the submodule $\Z/2\oplus \Z/3 \oplus \Z$ with $a=b=c=0$. Here, $\ker(1+\sigma) = \im (1-\sigma)$, so that $H^1(C_2, M^{SL_2(\Z/3)}) = 0$.

Next, suppose $f: SL_2(\Z/3) \to M$ represents a class in $H^1(SL_2(\Z/3), M )^{C_2}$, i.e. it is a crossed homomorphism which is $\sigma$-invariant modulo coboundaries. Since each $f(g)$ is in the kernel of the norm of $g$, we must have that 
\begin{align*}
f(x) &= (\epsilon_x, c_x, 0, a_x, -c_x, c_x),\\
f(y) & = (\epsilon_y, -a_y-c_y, 0, a_y, -a_y, c_y),\\
f(z) &= (0, \alpha_z, 0,0, b_z, c_z).
\end{align*}
The relations $x^2=y^2$, $xyx=y$, $xz=zy^3$, and $zyx=yz$, imply that 
\begin{align*} 
a_x+c_x = a_y + c_y, \qquad b_z = -c_x, \qquad c_z= c_y, \qquad
\epsilon_x = c_x + c_y, \qquad \epsilon_y = a_x.
\end{align*}
One directly checks that any crossed homomorphism of this form is $\sigma$-invariant modulo coboundaries. 
Finally, suppose an $f$ of this form is itself a coboundary, i.e. there is an $m = (\epsilon, \alpha, \beta, a, b, c)\in M$, such that $f(g) = gm - m$ for all $g\in SL_2(\Z/3)$. Then $ 4b = a_x + 3 c_x - 2 c_y$, $a = b-a_x-c_x + 2c_y$, $c= b- c_x$, and $\alpha_z = a$. Consequently, 
\begin{align}
H^1(SL_2(\Z/3), M)^{C_2} = \Z/12. 
\end{align}

It remains to compute the differential \eqref{eq:d2diff}. This is a transgression, and we have an explicit formula for it, for example in \cite[\S 3.7]{Koch} or \cite[\S I.6]{Neukirch}. One checks that this formula gives that $d_2$ is zero in our case.
Thus we have proved the following.

\begin{proposition}\label{grpcohGL3}
The cohomology group $H^1(\GL{3}, TMF(3)_0^\times )$ is cyclic of order $12$.
\end{proposition}

Although not directly affecting the computation of $\pic(TMF[1/3])$, we record the following result of Quillen that determines a few more entries in the spectral sequence \eqref{ss:TMF(3)}. 

\begin{proposition}{\cite[Lemma 11]{QuillenGL}}\label{Z2cohG3}
The cohomology ring $H^*(\GL{3},\Z/2)$ is $\Z/2[c_1, c_2] \otimes \Lambda(e_1,e_2)$, where the cohomological degrees are $|c_i|=2i$, and $|e_i|=2i-1$.
\end{proposition}


\section{Derived functors of the symmetric square}
\label{app:symm}
The purpose 
of this appendix is to prove the necessary auxiliary results on symmetric
squares of cosimplicial abelian groups. 

\begin{definition} 
\label{symmdef}
Let $A$ be an abelian group. We let $\sym_2(A) = (A
\otimes A)_{C_2}$ be the coinvariants for the $C_2$-action on $A \otimes A$
given by permuting the factors. 
We also let $\ssym_2(A)$ denote the $C_2$-coinvariants in $(A \otimes A)
\otimes \mathbb{Z}_{\epsilon}$ where the first factor is given the permutation
action and $\mathbb{Z}_{\epsilon}$ is the sign representation.
Note that if $A$ is a free abelian group, then the 2-torsion in $\ssym_2(A)$ is
canonically isomorphic 
to $A \otimes_{\mathbb{Z}} \mathbb{F}_2$ via the ``Frobenius'' map
\[ A/2A \to \ssym^2(A), \quad a \mapsto a \otimes a. \]
\end{definition} 

In \cite{Priddy73}, Priddy gives a complete description of the actions of the
symmetric algebra functor on cosimplicial vector spaces, or equivalently the
analog of the Steenrod algebra for cosimplicial algebras. 
We will only need a small piece of this, which we state next. 
We note that the generators in question are the Steenrod squares applied to
the fundamental class $\iota$.
For example, the generator in maximal degree is the cup square.

\begin{proposition}{\cite[Theorem 4.0.1]{Priddy73}}
\label{priddythm}
Let $A^\bullet$ be a cosimplicial $\mathbb{F}_2$-vector space. Suppose that
$H^{t+1}(A^\bullet) \simeq \mathbb{F}_2$ and the cohomology of $A^\bullet$ is
concentrated in degree $t + 1$ by a class $\iota$. 
Then 
$$ 
H^i(\sym^2 A^\bullet) \simeq \begin{cases}
\mathbb{F}_2  & \text{if } t+1 \leq i \leq 2(t + 1) \\
0 & \text{otherwise} 
\end{cases} .$$
\end{proposition}

\begin{proposition} 
\label{mainsymmcalculation}
\label{symmsquare}
Let $t \geq 2$ and let $A^\bullet$ be a levelwise free, finitely generated cosimplicial abelian group with
$H^*(A^\bullet) $ concentrated in degree $\ast = t+1$ and $H^{t+1}(A^\bullet) =
\mathbb{Z}$ generated by $\iota$. Then: 
\begin{enumerate}
\item If $t$ is even, $H^{2t+2} ( \sym_2 A^\bullet) \simeq \mathbb{Z}/2$,
generated by  $\iota^2$. 
\item If $t$ is odd, then $H^{2t+2}( \ssym_2 A^\bullet) \simeq \mathbb{Z}/2$,
generated by $\iota^2$. 
\end{enumerate}
\end{proposition} 
\begin{proof} 
Consider first the 
case $t$ even. 
In this case, we have maps
of cosimplicial abelian groups
\[ \sym_2 A^\bullet \to A^\bullet \otimes
A^\bullet \to \sym_2 A^\bullet  \]
where the first map is the norm map and the second map is projection. 
The composite is multiplication by two. Note that $H^{2t+2}( A^\bullet \otimes
A^\bullet) \simeq \mathbb{Z}$, but since $t$ is even, the $C_2$-action is the
sign representation, so that the map $H^{*}(\sym_2 A^\bullet) \to
H^{*}(A^\bullet \otimes A^\bullet)$ must be the zero map as it lands in the
$C_2$-invariants on cohomology. 
In particular, the cohomology of $\sym_2 (A^\bullet)$ is all annihilated by 2. 
By the universal coefficient theorem, it suffices to show that $H^{2t+2}(\sym_2
A^\bullet \otimes_{\mathbb{Z}} \mathbb{Z}/2) \simeq \mathbb{Z}/2$ and 
$H^{k}(\sym_2
A^\bullet \otimes_{\mathbb{Z}} \mathbb{Z}/2) = 0$ for $k > 2t + 2$,
which is the statement of \Cref{priddythm}. In addition, we see that $\iota^2$
is a generator, as desired, by working modulo $2$. 

Now suppose $t$ is odd.
Again, using the norm maps
\[ \ssym_2 A^\bullet \to A^\bullet \otimes A^\bullet \otimes \epsilon \to
\ssym_2 A^\bullet,  \]
we find that the cohomology of $\ssym_2 A^\bullet$ is annihilated by two.
We note that $\ssym_2 A^\bullet \otimes_{\mathbb{Z}}
\mathbb{F}_2 \simeq \sym_2 A^\bullet \otimes_{\mathbb{Z}} \mathbb{F}_2$ at
the level of cosimplicial abelian groups, but working with the underived
tensor product is problematic here as $\ssym_2 A^\bullet$ has 2-torsion.
If we take the derived tensor product $\ssym_2(A^\bullet)
\stackrel{\mathbb{L}}{\otimes} \mathbb{F}_2$, we obtain in addition a copy of
$A^\bullet \otimes_{\mathbb{Z}} \mathbb{F}_2$ (i.e., the 2-torsion in
$\ssym_2 A^\bullet$)  in $\pi_1$ that does not contribute in the
relevant dimensions, so we may ignore it. 
Now, by \Cref{priddythm}, we know that $H^{k}(\ssym_2 A^\bullet \otimes_{\mathbb{Z}}
\mathbb{F}_2) \simeq \mathbb{F}_2$ for $k = 2t+2$ and $0$ for $k > 2t+2$, so
that we can apply the universal coefficient theorem as in the previous case. 
We conclude that $\iota^2$ is a generator similarly. 
\end{proof}

\bibliographystyle{amsalpha}
\bibliography{biblio}

\end{document}